\documentclass[aos,preprint]{imsart}

\RequirePackage[OT1]{fontenc}
\RequirePackage{amsthm,amsmath}
\RequirePackage[numbers]{natbib}
\RequirePackage[colorlinks,citecolor=blue,urlcolor=blue]{hyperref}

\arxiv{math.PR/0000000}

\usepackage{latexsym,amsmath,amsfonts}
\usepackage{algorithmicx,algpseudocode}
\usepackage{amssymb}
\usepackage{pstricks,pst-node,pst-tree}
\usepackage{thmtools,thm-restate}
\usepackage{graphicx,epsfig,epstopdf}
\usepackage{natbib}

\usepackage{multirow}

\startlocaldefs
\numberwithin{equation}{section}
\theoremstyle{plain}

\newtheorem{theorem}{Theorem}
\newtheorem{proposition}{Proposition}
\newtheorem{corollary}{Corollary}
\newtheorem{lemma}{Lemma}

\newtheorem{example}{Example}

\newtheorem{remark}{Remark}

\def\eeta{\bar h}

\def\wX{\widetilde{X}}
\def\wY{\widetilde{Y}}
\def\wW{\widetilde{W}}
\def\dZ{\dot{Z}}

\def\Ep{{\rm E}}
\def\En{{\mathbb{E}_n}}
\def\Gn{{\mathbb{G}}_n}

\def\P{{\rm P}}

\newcommand{\Var}[1]{\mathbb{V}\left(#1\right)} 

\def\G{\mbox{\sf G}}

\endlocaldefs

\begin{document}

\begin{frontmatter}
\title{Subvector Inference in PI Models with Many Moment Inequalities}
\runtitle{Confidence Regions Partially Identified}
\thankstext{T1}{First Version: February 20, 2017; Current Version: \today. We thank comments and suggestions from the participants of the 2017 Triangle Econometric Conference, MMS2017, Spring School on Structured Inference in Spreewald, and Econometric Seminars at Yale University, Universit\'{e} de Montr\'{e}al, and Vanderbilt University. Of course, any and all errors are our own. The research of the second author was supported by NSF Grant SES-1729280.}

\begin{aug}
\author{A. Belloni, F. Bugni and V. Chernozhukov}




\end{aug}

\begin{abstract}
This paper considers inference for a function of a parameter vector in a partially identified model with {\it many moment inequalities}. This framework allows the number of moment conditions to grow with the sample size, possibly at exponential rates. Our main motivating application is subvector inference, i.e., inference on a single component of the partially identified parameter vector associated with a treatment effect or a policy variable of interest.

Our inference method compares a {\it MinMax} test statistic (minimum over parameters satisfying $H_0$ and maximum over moment inequalities) against critical values that are based on bootstrap approximations or analytical bounds. We show that this method controls asymptotic size uniformly over a large class of data generating processes despite the partially identified many moment inequality setting. The finite sample analysis allows us to obtain explicit rates of convergence on the size control. Our results are based on combining non-asymptotic approximations and new high-dimensional central limit theorems for the MinMax of the components of random matrices, which may be of independent interest. Unlike the previous literature on functional inference in partially identified models, our results do not rely on weak convergence results based on Donsker's class assumptions and, in fact, our test statistic may not even converge in distribution. Our bootstrap approximation requires the choice of a tuning parameter sequence that can avoid the excessive concentration of our test statistic. To this end, we propose an asymptotically valid data-driven method to select this tuning parameter sequence. This method generalizes the selection of tuning parameter sequences to problems outside the Donsker's class assumptions and may also be of independent interest. Our procedures based on self-normalized moderate deviation bounds are relatively more conservative but easier to implement.

\end{abstract}



\end{frontmatter}

\section{Introduction}

This paper contributes to the growing literature on inference in partially identified econometric models defined by a large number of moment inequalities. As discussed by \citet{tamer2010partial}, \citet{canay/shaikh:2017}, and \citet{ho/rosen:2017}, partially identified moment inequality models arise naturally in a large variety of economic problems and have been increasingly used in the empirical literature. As argued in \citet{chernozhukov2013testing}, the number of moment inequalities implied by many econometric applications is frequently very large relative to the sample size. Examples of this include \citet{bajari2007estimating}, \citet{ciliberto2009market}, \citet{pakes2013moment}, \citet{beresteanu2011sharp}, \citet{galichon2011set}, \citet{chesher2013instrumental}, and \citet{chesher2017generalized}.

There is a substantial literature on inference of the entire parameter vector in a partially identified moment inequality model. An earlier strand of this literature considered partially identified models defined by a finite number of unconditional moment inequalities\footnote{See \citet{chernozhukov2007estimation, andrews/berry/jiabarwick:2004, romano2008inference, rosen2008confidence, andrews2009validity, andrews2010inference, canay2010inference, bugni:2010, bugni2016comparison, andrews2012inference, romano2014practical, menzel2014consistent, bugni2015specification}, and \citet{pakes/porter/ho/ishii:2015}, among many others.} and conditional moment inequalities\footnote{See \citet{kim:2008, ponomareva:2010, andrews2013inference,chernozhukov2013intersection,lee2013testing,andrews2014nonparametric,armstrong2014weighted, armstrong2015asymptotically, armstrong2016multiscale, chetverikov2018adaptive}, and \citet{armstrong2018choice}, among many others.}. More recently, \citet{chernozhukov2013testing} study the problem of the entire parameter vector in partially identified models with {\it many moment inequalities}. According to this asymptotic framework, the number of moment inequalities is allowed to grow with the sample size, possibly at exponential rates. Furthermore, the moment inequalities in this framework are allowed to be {\it unstructured}, in the sense that no restrictions are imposed on their correlation structure.\footnote{This feature distinguishes their framework from a model with conditional moment inequalities. While conditional moment conditions can generate an uncountable set of unconditional moment inequalities, their covariance structure is restricted by the conditioning structure.} As \citet{chernozhukov2013testing} explain, the many moment inequalities framework substantially expands the scope of applications by allowing the number of unstructured moment inequalities to be much larger than the sample size.

The main contribution of this paper is to propose inference for a function of a parameter vector in a partially identified models with {\it many moment inequalities}. Our main motivating application is {\it subvector inference}, i.e., inference on a single component of the partially identified parameter vector that is associated with a treatment effect or a policy variable of interest. Our inference method is based on the {\it MinMax} test statistic, where the minimum is computed over the parameter values that satisfy the null hypothesis (i.e.\ profiling) and the maximum is computed over an index of the moment inequalities. We propose comparing the MinMax test statistic against critical values that can be based on bootstrap approximations or analytical bounds. We show that the resulting inference method with either of these critical values controls asymptotic size uniformly over a large class of data generating processes.

There is a recent literature on the problem of inference for a function of a partially identified parameter in a moment inequality model. This literature focuses on moment inequality models with finitely many moment conditions, which can be restrictive in certain applications. We now summarize this literature. \citet{andrews2009validity} and \citet{andrews2010inference} propose conducting projection-based inference, i.e., intersecting the confidence set for the entire parameter space with the subset of the parameter space that satisfies the null hypothesis. \citet{bugni2017inference} show that projection-based inference can result in power losses relative to profiling-based methods. \citet{romano2008inference} consider profiling-based inference with critical values constructed using subsampling. In turn, \citet{bugni2017inference} and \citet{kaido2016confidence} consider profiling-based methods with critical values constructed using the bootstrap.
\citet{gafarovinference} considers subvector inference in affine moment inequality models. As already mentioned, these references restrict attention to partially identified models with a finite number of moment conditions. This restriction becomes essential in proving the asymptotic validity of the proposed inference. For example, it implies that profiling-based test statistics, such as the MinMax statistic, converge in distribution to a function of a Gaussian process. In contrast, these statistics may not even converge in distribution in the many moment inequality setting considered in this paper.

More recently, \citet{chernozhukov2015constrained} consider inference for subvector on partially identified conditional moment restriction models with a conditioning covariate of restricted dimension. This restriction imposes enough structure on the problem to allow them to use Koltichinskii's Hungarian couplings and approximate the relevant empirical processes by the corresponding Gaussian processes. By contrast, our approach is designed to be valid in a framework with many moment inequalities that does not impose the structure produced by having a small number continuous conditional moment inequalities, enabling a much broader set of applications.

Our inference method compares the MinMax test statistic with critical values based on bootstrap approximations or analytical bounds. Both methods have their relative merits. The first set of main results pertains to bootstrap approximations that exploit correlation structures to increase power. These procedures were inspired by the ideas in \citet{bugni2017inference} which considered a finite number of inequalities but for a broader class of test statistics. By considering many moment inequalities to cover a wide range of applications, we develop new theoretical arguments to establish the validity of our bootstrap procedures. In particular, non-asymptotic bounds and new high-dimensional central limit theorems for the MinMax statistics were developed.\footnote{This is in sharp contrast to \citet{bugni2017inference}, who use Donsker's functional central limit theorem to derive a limiting distribution.} Our second set of main results is on the construction of analytical bounds based on self-normalized moderate deviation theory which are computationally easier but do not exploit the underlying correlation structure. We build upon arguments in \citet{chernozhukov2013testing} which considered the vector inference problem but also allowed for many moment inequalities.

As it is usual in this literature, our bootstrap approximation requires a threshold sequence $(\kappa_n)_{n=1}^{\infty}$ that determines whether each moment inequality is considered to be sufficiently close to binding or not. In models with finitely many moment inequalities, the threshold sequence is required to satisfy $\kappa_n \to \infty$ and $\kappa_n/\sqrt{n} \to 0$.
We show that the many moment inequality setting imposes additional requirements on the threshold sequence. In order to facilitate this choice for practitioners, we propose a data-driven method to select this threshold sequence in an asymptotically valid way.

In deriving our primary results, we also obtain several auxiliary findings that might be of independent interest. First, we derive a new non-asymptotic coupling result for the MinMax of an empirical process (see Theorem \ref{thm:clt:minmax:discrete}). A key ingredient to obtain such coupling is a novel use of a smooth approximation of the MinMax functional. Second, the data-driven procedure proposed to estimate the anti-concentration of the MinMax statistic seems to be widely applicable to other contexts, as it allows one to bypass the need to development of anti-concentration bounds that are only available to a limited class of statistics, such as the Max statistic. Both results are new and could be of independent interest beyond our contribution.

The paper is organized as follows. Section \ref{SEC:SETUP} formally describes the setting and problem. We provide an overview of some of our main results in Section \ref{Sec:Overview} where we discuss insights and key issues. Section \ref{SEC:INFERENCESUBVECTORPI} derives our main theoretical results for different procedures. We discuss in detail the anti-concentration properties associated with the MinMax statistics in Section \ref{SEC:ANTICONCENTRATION}. Section \ref{SEC:MinMaxCoupling} provides new coupling results for the MinMax of non-centered empirical processes and for the sum of independent random matrices which are of independent interest.

\section{Setup and Test Statistics}\label{SEC:SETUP}

Let $\left( S,\mathcal{S}\right) $ be a measurable space, and let $%
(W_{i})_{i=1}^{n}$ denote an i.i.d.\ sequence of random variables taking
values on $\left( S,\mathcal{S}\right) $ with common distribution $P\in
\mathcal{P}$, and let $\theta \in \Theta \subseteq \mathbb{R}^{d_{\theta }}$
denote the parameter of the model. The econometric model predicts that true
parameter value, $\theta ^{\ast }\in \Theta $, satisfies the following
collection of $p_{I}$ moment inequalities and $p_{E}$ moment equalities:
\begin{equation}\label{eq:MI}
\begin{array}{rl}
& \Ep[{m}_{I,j}(W,\theta )]~\leq ~0\text{ for }j=1,\ldots ,p_{I},  \\
& \Ep[{m}_{E,j}(W,\theta )]~=~0\text{ for }j=1,\ldots ,p_{E},
\end{array}
\end{equation}
If we let $m_{j}(W,\theta )\equiv m_{I,j}(W,\theta )$ for $j=1,\dots ,p_{I}$
and $m_{p_{I}+j}(W,\theta )\equiv m_{E,j}(W,\theta )$ for $j=1,\dots ,p_{E}$%
, and $m_{p_{I}+p_{E}+j}(W,\theta )\equiv -m_{E,j}(W,\theta )$ for $%
j=1,\dots ,p_{E}$, \eqref{eq:MI} can be equivalently re-expressed as
 moment inequality model with $p\equiv p_{I}+2p_{E}$
moment inequalities:
\begin{equation}\label{eq:MI2}
\Ep[m_{j}(W,\theta )]~\leq ~0\text{ for }j \in [p]\equiv\{1,\ldots ,p\},
\end{equation}%

We allow the econometric model to be partially identified, i.e., the moment
inequalities in \eqref{eq:MI2} do not necessarily restrict $\theta^{*}$
to a single value, but rather constrain it the identified set, given by:
\begin{align}
\Theta_{I}~\equiv ~\{ \theta \in \Theta ~:~ \Ep[m_{j}(W,\theta )]\leq 0%
\text{ for }j \in [p] ~\}.  \label{eq:IdSet}
\end{align}

We are implicitly allowing the distribution $\P$ of the data to change with the sample size. In particular, the dimensionality of $\Theta$, denoted by $d_\theta$, and the number of moment inequalities $p$ to depend on $n$. In
particular, we are primarily interested in the case in which $p=p_{n}\to
\infty$, but the subscripts are omitted to keep the notation simple. In
particular, $p$ can be much larger than the sample size $n$.

Let $h:\Theta \rightarrow \mathbb{R}^{d_{h}}$ be a known function, let $%
h[\Theta ]$ denote the image of $h$, and let $\eeta \in h[\Theta ]$ be an
arbitrary parameter value. Our goal is to conduct the
following hypothesis test:
\begin{equation}\label{eq:HypTest}
H_{0}:h(\theta ^{\ast })=\eeta ~~\text{ vs. }~~H_{1}:h(\theta ^{\ast })\neq
\eeta .
\end{equation}%
The main application of our test is the subvector inference
problem, i.e., testing whether a subset of the partially identified
parameter vector is equal to a particular value or not. For example, we
could test whether, say, the $s$'th coordinate of $\theta ^{\ast }$, $\theta
_{s}^{\ast }$, is equal to a particular value $\eeta \in \mathbb{R}$ or not,
i.e.,
\begin{equation}\label{eq:HypTest_subvector}
H_{0}:\theta _{s}^{\ast }=\eeta ~~\text{ vs. }~~H_{1}:\theta _{s}^{\ast }\neq
\eeta .
\end{equation}%
\eqref{eq:HypTest_subvector} is a particular example of
\eqref{eq:HypTest} with $h(\theta )=\theta _{s}$.

To test \eqref{eq:HypTest}, we propose the
\textquotedblleft MinMax\textquotedblright\ test statistic, defined as
follows:
\begin{equation}\label{eq:Tstat}
T_{n}(\eeta )~~\equiv ~~\inf_{\theta \in \Theta (\eeta )}\max_{j\in [p]}~\sqrt{n%
}\bar{m}_{\theta,j}/\hat{\sigma}_{\theta,j},
\end{equation}%
where $\Theta (\eeta )$ is the \textquotedblleft null set\textquotedblright\
associated to \eqref{eq:HypTest}, given by:
\begin{equation*}
\Theta (\eeta )~\equiv ~h^{-1}[\{\eeta \}]=\{~\theta \in \Theta ~:~h(\theta
)=\eeta ~\}
\end{equation*}%
and
\begin{align*}
\bar{m}_{\theta,j} & ~\equiv ~\frac{1}{n}\sum_{i=1}^{n}m_{j}(W_{i},\theta ), \ \ \ \ \ \hat{\sigma}_{\theta,j}^{2}  ~\equiv ~\frac{1}{n}%
\sum_{i=1}^{n}\{m_{j}(W_{i},\theta )-\bar{m}_{\theta,j}\}^{2}.
\end{align*}%
Finally, if $\hat{\sigma}_{\theta,j}=0$ in \eqref{eq:Tstat}, we
define $\bar{m}_{\theta,j}/\hat{\sigma}_{\theta,j}\equiv \infty \times
{\rm sign}(\bar{m}_{\theta,j})$.

Given the test statistic in \eqref{eq:Tstat} and nominal size $\alpha \in (0,1)$, our goal is to present critical values $c_{n}(\eeta ,\alpha)$ that can be associated with $ T_{n}(\eeta)$ to produce asymptotically valid inference for \eqref{eq:HypTest}. Specifically, we propose to reject $H_{0}$ in \eqref{eq:HypTest} if and only if
\begin{equation}\label{eq:HT}
T_{n}(\eeta)>c_{n}(\eeta ,\alpha ).
\end{equation}%
By exploiting the duality between hypothesis tests and confidence sets, we construct a confidence set for $h(\theta ^{\ast })$ by collecting all parameter values $\eeta $ for which we do not reject, i.e.,
\begin{equation}\label{eq:CS}
C_{n}(1-\alpha )~\equiv ~\{\eeta \in h[\Theta ]~:~T_{n}(\eeta )\leq c_{n}(\eeta
,\alpha )\}.
\end{equation}

Our formal results will have the following structure. Under $H_{0}$, we show that
\begin{equation}
\P\left(T_{n}(\eeta )>c_{n}(\eeta ,\alpha )\right)~~\leq ~~\alpha + Cn^{-c},
\label{eq:CSvalid}
\end{equation}%
where $c$ and $C$ are constants that depend only on other constants of the problem (i.e.\ they do not depend on $\eeta$, $\P$, or $n$). As a corollary of this, the convergence in \eqref{eq:CSvalid} occurs uniformly over a suitable class of probability distributions.

The main contribution of this paper is to propose methods
to approximate the critical value $c_{n}(\eeta ,\alpha )$ that satisfies \eqref{eq:CSvalid} in the many inequalities setting. To this end, we develop an approximation to the
asymptotic distribution of the MinMax test statistic $T_{n}(\eeta)$ based on suitably constructed bootstrap procedures and based on self-normalized moderate deviation theory.

%

\begin{remark}[Many Conditional moment inequalities]
The analysis developed here can be applied to models defined by many conditional moment inequalities/equalities where the number of moment restrictions can be unbounded, see e.g. \citet{chernozhukov2013intersection} and \citet{andrews2017inference}. Formally, we have
$$
\Theta_I = \left\{ \theta \in \Theta : \begin{array}{rl}
 \Ep[m_\tau(W,\theta)\mid X]\leq 0, \ \ \tau \in \mathcal{I},  \\ \Ep[m_\tau(W,\theta)\mid X]= 0, \ \ \tau \in \mathcal{E} \, \, \\
\end{array}\right\}
$$
where $\mathcal{I}$ is a set of indices for the moment inequalities and $\mathcal{E}$ for the moment equalities. As we show below, the results developed here are directly applicable to these models as well. We first reexpress the conditional moment conditions
into equivalent unconditional ones via instrumental functions $g \in \mathcal{G}$, given by
$$
\Theta_I = \bigcap_{g \in \mathcal{G}}\left\{\theta \in \Theta : \begin{array}{rl}
   \Ep[m_\tau(W,\theta)g(X)]\leq 0, \  \tau \in \mathcal{I}, \\  \Ep[m_\tau(W,\theta)g(X)]= 0, \ \ \tau \in \mathcal{E}
\end{array}\right\}
$$
\end{remark}

\section{Overview of Main Results}\label{Sec:Overview}

In this section we provide a simplified discussion of the issues and proposed methods to construct critical values that are asymptotically valid in the sense of \eqref{eq:CSvalid}. This section intends to motivate the key issues we face and more general results and other procedures are discussed in Section \ref{SEC:INFERENCESUBVECTORPI}.

An important feature of the MinMax statistic (\ref{eq:Tstat}) is that it is a functional of non-centered sums. Therefore it is convenient to highlight the centering and rewrite the MinMax statistic as
$$T_{n}(\eeta) = \inf_{\theta \in \Theta (\eeta) }\max_{j\in[p]} \left\{\hat v_{\theta,j} +\sqrt{n}\Ep[
m_{j}\left(W, \theta \right)] /\hat {\sigma}_{\theta,j}\right\},$$
where $\hat v_{\theta,j}$ is the centered empirical process indexed by $\theta \in \Theta$ and $j\in [p]$, i.e.,
$$\hat v_{\theta,j}~\equiv~ n^{-1/2}\sum_{i=1}^n \{ m_j(W_i,\theta)- \Ep[
m_{j}\left(W, \theta \right)] \}/\hat \sigma_{\theta,j}.$$

There are two main insights in the derivation of critical values for the MinMax statistics. Under $H_0$, $T_n(\eeta) \leq \max_{j\in[p]} \sqrt{n}\bar m_j(\theta^*)/\hat \sigma_{\theta^*,j} \leq \max_{j\in [p]} \hat v_{\theta^*,j} $. However, also under $H_0$, $h(\theta^*)$ is known and yet the parameter vector $\theta^*$ may not be known.
Therefore, in addition to considering which inequalities can be binding, it is important to consider which values of the parameter $\theta$ are candidates to be $\theta^*$ based on the moment inequalities and $h(\theta^*)=\eeta$. Note however that these two ``selection'' problems have very different consequences in the analysis of the size of the test. That is, overselecting inequalities leads to potentially conservative but valid asymptotic size control, while over selecting values of $\theta$, because of the minimum, can lead to a procedure that fails to control asymptotic size.
In particular, using a critical value based on bootstrapping $\min_{\theta \in \Theta(\eeta)}\max_{j\in [p]} \hat v_{\theta,j}$ will fail to control asymptotic size even if $p=2$ for simple data generating processes.\footnote{See Example \ref{Ex:Fail} in the Appendix. For the minimum of the sum of violations similar issues have been shown in \citet{bugni2017inference}.}

We begin with a bootstrap-based procedure for the construction of a critical value. Consider the critical value $c_{n}(\eeta ,\alpha )$ based on the penalized bootstrap procedure
\begin{equation}\label{def:CVoverview} c_{n}(\eeta ,\alpha ) = \mbox{conditional} \ \ (1-\alpha)\mbox{-quantile of} \ R_n^* \ \mbox{given} \  (W_i)_{i=1}^n \end{equation}
where the random variable $R_n^*$, conditional on the data, is defined as
\begin{equation}\label{Def:MPBa}R_n^* \equiv \inf_{\theta \in \Theta(\eeta) }\max_{j\in[p]}  \underbrace{\frac{1}{\sqrt{n}}\sum_{i=1}^n g_i\frac{m_j(W_i,\theta)-\bar m_{\theta,j}}{\hat \sigma_{\theta,j}}}_{\mbox{zero-mean Gaussian process}} +\kappa_n^{-1}\underbrace{\frac{\sqrt{n}\bar{m}_{\theta,j}}{\hat{\sigma}_{\theta,j}}}_{\mbox{centering}} \end{equation}
where $\kappa_n$ is a tuning parameter, and $(g_i)_{i=1}^n$ are i.i.d.\ standard Gaussian random variables. It has been noted in the literature that the centering term cannot be consistently estimated in a uniformly fashion. Therefore the use of $\kappa_n=1$ would not make (\ref{def:CVoverview}) a valid choice. In (\ref{def:CVoverview}) we set $\kappa_n$ to dominate the effective noise in the centering, namely  $$\bar w_n \geq \sup_{\theta\in \Theta(\eeta)}\max_{j\in[p]}|\sqrt{n}\hat\sigma^{-1}_{\theta,j}\{\bar m_{\theta,j}-\Ep[m_j(W,\theta)]\}|$$ with high probability. With that in mind we keep $\kappa_n$ as small as possible not to remove the true centering $\sqrt{n}\hat\sigma^{-1}_{\theta,j}\Ep[m_j(W,\theta)]$ completely (because the minimum over $\theta$ would lead to a too small critical value). We note that because of the many inequalities setting, we typically have $\bar w_n \to \infty$. \footnote{It follows we will be able to approximate $\bar w$ via a separate bootstrap procedure providing a data driven way to compute $\bar w$ that is theoretically valid.}

Moreover, the approximation errors of the new coupling results should not affect the size. Letting $\delta_n$ denote such coupling approximation error, we have that 
\begin{align*}
	\P(T_n(\eeta) \geq c_{n}(\eeta ,\alpha )) & \leq \P(R_n^{*}\geq c_{n}(\eeta ,\alpha ) - \delta_n )+ Cn^{-c}\\
	& \leq \alpha + \P( |R_n^{*} - c_{n}(\eeta ,\alpha )|\leq \delta_n  ) + Cn^{-c}.
\end{align*}
The validity of the bootstrap approximation is guaranteed by $\P( |R_n^{*} - c_{n}(\eeta ,\alpha )|\leq \delta_n )\to 0$. In other words, the distribution of $R_n^*$ should not concentrate too much mass around $c_{n}(\eeta ,\alpha )$ as $n$ diverges. It is not hard to show that this concentration is bounded by $\delta_n$ multiplied by the maximum value of the density function of $R_n^{*}$. In moment inequality models with fixed number of moment conditions, this does not pose any problems as the limit of the corresponding density function is typically bounded. However, in moment inequality models with many moment inequalities, these statistics can have very different behavior. For example, \citet{chernozhukov2013testing} show that the distribution of the Maximum statistic can concentrate but not too fast. Indeed it is shown that maximum density of the Maximum statistic is bounded by a logarithmic factor of the number of moment conditions $p$. In this paper, we are interested in the concentration properties associated to the MinMax statistic. Based on the previous derivation, we define unconditional and conditional  anti-concentration parameters associated with $R_n^*$ as follows:
\begin{equation}
	\label{eq:AC}
	\mathcal{A}_{n} ~\equiv~ \sup_{\epsilon \geq \delta_n} \frac{1}{\epsilon}\P( |R_n^{*} - c_{n}(\eeta ,\alpha )|\leq \epsilon ).
\end{equation}
\begin{equation}
	\label{eq:ACcond}
	\mathcal{A}_{n}(W) ~\equiv~ \sup_{\epsilon \geq \delta_n} \frac{1}{\epsilon}\P( |R_n^{*} - c_{n}(\eeta ,\alpha )|\leq \epsilon \mid (W_i)_{i=1}^n).
\end{equation}
Importantly, the anti-concentration property that is needed pertains to the bootstrap-based statistic $R_n^{*}$, and not the original statistics $T_n(\eeta)$. Because of this feature, we can investigate $\mathcal{A}_{n}$ via bootstrap of the quantity $\mathcal{A}_{n}(W)$ conditionally on the data. This approximation is described in detail in Section \ref{SEC:ANTICONCENTRATION}.

The corollary below shows that the choice of critical value (\ref{def:CVoverview}) effectively controls the size, in the sense of (\ref{eq:CSvalid}), uniformly over a set of data generating processes. The conditions are simple to allow easier interpretability but allows for the many inequality setting with potentially $p\gg n$. (Our results hold much more generally as stated in Theorems  \ref{thm:MSB:inferenceSubvector} and \ref{thm:PB:inferenceSubvector}.)

\begin{corollary}\label{cor:mainresult1}
Assume that (i) $m_j$ and its componentwise derivatives are uniformly bounded by $C_1$,\footnote{$\left|\partial m_j(W,\theta)/{\partial \theta_k}\right|\leq C_1$ and $\left|m_j(W,\theta)\right|\leq C_1$ for all $W$, $j\in [p]$ and $\theta \in \Theta(\eeta)$.} (ii) $\Theta(\eeta)$ is convex and uniformly bounded in $\ell_\infty$-norm,\footnote{$\sup_{\theta \in \Theta(\eeta)}\|\theta\|_\infty \leq C_1$.} (iii) $\inf_{j\in[p],\theta\in \Theta(\eeta)}{\rm Var}(m_j(W,\theta))\geq c_1$ and (iv) that a polynomially minorant condition holds with $\vartheta_n \geq c_1$ and $\delta \geq c_1$.\footnote{for any $\theta \in \Theta(\eeta)\setminus \Theta_I$,  ${\displaystyle \max_{j\in[p]}}  \Ep\left[
m_{j}\left( W,\theta\right) \right] /\sigma _{\theta
_{n},j}   \geq \vartheta_n \min \{ \delta ,{\displaystyle \inf_{\tilde{\theta}\in
\Theta (\eeta) \cap \Theta_I }} \| \theta-\tilde{\theta} \| \}$.} Also assume that $\kappa_n$ satisfies
\begin{equation}\label{eq:cor:cond}
\left(\frac{\log^4 (p n^{d_\theta})}{n}\right)^{1/6}+ \kappa_n\frac{d_\theta^{1/2}\log(p  n^{ d_\theta}) }{n^{1/2}} + \frac{\bar w_n}{\kappa_n} \leq \frac{C_2n^{-c_2}}{\mathcal{A}_{n}}.
\end{equation}
where $\bar w_n$ is the $(1-n^{-c_2})$-quantile of the effective noise. Then, under $H_0$,
$$ \P(T_n(\eeta) \geq c_{n}(\eeta ,\alpha ) ) \leq \alpha + C n^{-c} ,$$
where $c$ and $C$ are constants that depend only on $c_1,C_1,c_2,C_2$.
\end{corollary}

Corollary \ref{cor:mainresult1} provides low-level conditions that apply to many cases of practical interest. For example, it includes moment conditions that are Lipschitz continuous and it allows the dimension of the parameter space $d_\theta$ and the number of moment inequalities to grow with the sample size. In particular, it allows for $p\gg n$.

\begin{example}[Polynomially many inequalities and large $d_\theta$]
In addition to  conditions (i)-(iv) in Corollary \ref{cor:mainresult1},  assume that $p=n^C$ for some fixed $C>1$, $d_\theta = n^a$ for some $a<1/4$, and that the anti-concentration parameter satisfies $\mathcal{A}_{n} \leq C\log^{3/2}n$. Then, 
condition (\ref{eq:cor:cond}) holds provided $\kappa_n = n^b$ for any $b\in ( \frac{1}{2}a, \frac{1}{2}-\frac{3}{2}a)$.
\end{example}

\begin{example}[Exponentially many inequalities]
In addition to  conditions (i)-(iv) in Corollary \ref{cor:mainresult1},  assume that $p\geq n^{\log n}$, $d_\theta\leq C\log n$, and the anti-concentration parameter satisfies $\mathcal{A}_{n}\leq C\log^{3/2}p$. Then, 
condition (\ref{eq:cor:cond}) holds if $\kappa_n \in [ n^{c_2} \log^2p \ , \ n^{1/2-c_2}\log^{-5/2}p]$, and $n^{-1/6+c_2}\log^{13/6}p = o(1)$.
\end{example}

%

Corollary \ref{cor:mainresult1} requires the choice of $\kappa_n$ to be appropriate. Such requirements generalize the requirements in \citet{bugni2017inference} where it was required $\kappa_n \to \infty$ and $\sqrt{n}/\kappa_n\to \infty$. Theorem \ref{thm:PB:inferenceSubvector} characterizes the finite sample impact of $\kappa_n$ for a non-Donsker class of functions. Theorem \ref{thm:PB:inferenceSubvector} also motivates data driven choices of $\kappa_n$ that accounts for the anti-concentration of the process $R_n^*$ which might be non-trivial since $\mathcal{A}_{n}$ could grow.

The bootstrap procedure discussed above and the others studied in Section \ref{Sec:Bootstrap} adapt to the correlation structure of the process to improve power. However, that is achieved by levering conditions on the number of inequalities $p$, sample size $n$, effective noise $\bar w_n$, and anti-concentration $\mathcal{A}_{n}$.

An alternative approach is to rely on union bounds that are more conservative but are valid under weaker conditions. Moreover, since the bounds are based on the marginal distribution, there is no need to handle the anti-concentration. With that in mind, similarly to \citet{chernozhukov2013testing}, ideas from self-normalized moderate deviation theory can be applied to control size as we describe next. By construction of the test statistics $T_n(\eeta)$ satisfies
\begin{equation}\label{eq:SN} \P(T_n(\eeta) > c_{n}(\eeta ,\alpha ) ) \leq \min_{\theta \in \Theta(\bar h)} \P\left( \max_{j\in [p]} \frac{\sqrt{n}\bar m_{\theta,j}}{\hat \sigma_{\theta,j}} > c_{n}(\eeta ,\alpha ) \right).\end{equation}
Moreover, $H_0$ implies that $\Ep[m_j(W_i,\theta^*)]\leq 0$, and when combined with the union bound we have
\begin{equation}\label{eq:SN2} \P(T_n(\eeta) > c_{n}(\eeta ,\alpha ) ) \leq \sum_{j=1}^p \P\left( \sum_{i=1}^n\frac{m_j(W_i,\theta^*)-\Ep[m_j(W_i,\theta^*)]}{\sqrt{n}\hat \sigma_{\theta^*,j}} > c_{n}(\eeta ,\alpha ) \right),\end{equation}
where the last sum is approximately self-normalized. Such self-normalization has been exploited in the moderate deviation literature to establish central limit theorems that are valid in the tails of the distribution under mild moment conditions, see e.g. \citet{delapena}. In turn this leads to pivotal choices for the critical value to control size of the test by setting $$c_{n}(\eeta ,\alpha ) \approx \Phi^{-1}(1-\alpha/p).$$ Since $c_{n}(\eeta ,\alpha )$ is of the order of $\sqrt{\log(p/\alpha)}$, it  grows in the many moment inequality setting where $p\to \infty$. Therefore the Gaussian approximation should hold sufficiently deep in the tails to include $c_{n}(\eeta ,\alpha )$. This leads to the restriction $\log^3 (p/\alpha) =o(n)$ under suitable moment requirements. The following corollary of Theorem \ref{thm:SNresult} provides a result.

\begin{corollary}\label{cor:SNresult} Assume that $H_0$ holds, i.e., $\theta^{*} \in \Theta(\eeta)$, and that the are constants $0<c_1<1/2$ and $C_1>0$ such that $\max_{j\in[p]}\Ep[|m_j(W,\theta^*)|^3] \leq C_1$, $\min_{j\in[p]}{\rm Var}(m_j(W,\theta^*))\geq c_1$, and
$\log^{3/2}(p/\alpha) \leq C_1 n^{1/2-c_1}.$  Then,
$$ \P ( T_n(\eeta) > c_{n}(\eeta ,\alpha ) ) \leq \alpha + Cn^{-c},$$
where $c$ and $C$ are constants that depend only on other constants of the problem.
\end{corollary}

We also consider two-step versions of the self-normalization procedure that can improve power relative to the self-normalization procedure described above. However, because of the MinMax structure, substantial departure from \citet{chernozhukov2013testing} is needed to establish the validity of a proposed critical value. These described in detail in Section \ref{Sec:SN:cv}.


\section{Critical Values and Theoretical Guarantees}\label{SEC:INFERENCESUBVECTORPI}

In this section, we propose several methods to construct critical values $c_{n}(\eeta ,\alpha )$ in the hypothesis test procedure in (\ref{eq:HT}). We consider critical values based on the bootstrap approximations and based on self-normalized moderate deviation approximations. While the resulting hypothesis tests are shown to control asymptotic size, they have different power properties and are asymptotically valid under different conditions.


\subsection{Bootstrap-based critical values}\label{Sec:Bootstrap}

In the following subsections we propose and analyze bootstrap-based method for the construction of critical values that appropriately control size despite the high-dimensionality of the process being considered. For each $\theta \in \Theta$ and $j\in[p]$, we will denote the bootstrap process
\begin{equation}\label{def:GMB}
	\hat v_{\theta,j}^{\ast }~\equiv~\frac{1}{\sqrt{n}}\sum_{i=1}^n \xi_i \{m_j(W_i,\theta)-\bar m_{\theta,j}\}/\hat \sigma_{\theta,j},
\end{equation} where $(\xi_i)_{i=1}^n$ are i.i.d.\ $N(0,1)$ random variables independent of the data $(W_i)_{i=1}^n$. We will make the following assumption to analyse the bootstrap based critical values. 
~\\
\noindent {\bf Condition MB.} {\it
The set $\Theta(\eeta)$ is convex and the following conditions hold: (i) $\|\theta\|_\infty \leq \sqrt{n}$ and $\max_{ j \in [p]} \|\nabla_\theta\Ep[m_j(W,\theta)]/\sigma_{\theta,j}\| \leq L_G$ for every $\theta \in \Theta(\eeta)$. The set of functions $\{m_j(\cdot,\theta)/\sigma_{\theta,j} : \theta \in \Theta(\eeta), j\in [p]\}$ is VC type with measurable envelope $F$ and constants $\bar A$ and $v\geq 1$.\footnote{See Section \ref{SEC:PROCESSES} for a definition.} Moreover, for constants $b\geq \sigma >0$ we have  $\sup_{\theta\in\Theta(\eeta)}\max_{j\in[p]}\Ep[|m_j(W,\theta)/\sigma_{\theta,j}|^k] \leq \sigma^2 b^{k-2}$, $k=2,3,4$, and $\Ep[F^q]^{1/q} \leq b$ for $q>6$.  (ii) $\max_{ j \in [p]} \Ep[\{m_j(W,\theta)/\sigma_{\theta,j}-m_j(W,\tilde\theta)/\sigma_{\tilde\theta,j}\}^2] \leq L_C\|\theta - \tilde\theta\|^{\chi}$ for every $\theta, \tilde \theta \in \Theta(\eeta)$ for some $\chi \geq 1$. (iii) For every $\theta \in \Theta(\eeta)\setminus \Theta_I$ we have  $$\max_{j\in[p]}  \Ep\left[
m_{j}\left( W,\theta\right) \right] /\sigma _{\theta
_{n},j}   \geq \vartheta_n \min \{ \delta ,\inf_{\tilde{\theta}\in
\Theta (\eeta) \cap \Theta_I } \| \theta-\tilde{\theta} \| \}.$$ }

These conditions impose the existence of fourth moments of $m_j(W,\theta)/\sigma_{\theta,j}$ for each $j \in [p]$ and $\theta \in \Theta$. It also states that the functions $m_j(\cdot,\theta)/\sigma_{\theta,j}$ are well behaved in the sense of being a VC type, which covers a many applications of interest. Condition MB(iii) is a polynomial minorant condition as we move away from the identified set similar to the conditions imposed in \citet{chernozhukov2007estimation} and \citet{bugni2017inference}. However, we also allow for the parameter $\vartheta_n \to 0$ in our analysis.


For a sequence $\gamma_n = o(1)$, we define
\begin{align}
	&\bar w_n \equiv (1-\gamma_n)\mbox{-quantile of} \ \sup_{\theta\in \Theta(\eeta), j\in [p]}(|v_{\theta,j}|\vee|\hat v_{\theta,j}^*|)\label{eq:wn}\\
	&t_n^\sigma \equiv (1-\gamma_n)\mbox{-quantile of} \ \sup_{\theta \in \Theta(\eeta), j\in[p]} (|\hat\sigma_{\theta,j}/\sigma_{\theta,j}-1|\vee |\sigma_{\theta,j}/\hat\sigma_{\theta,j}-1|)\label{eq:tn}\\
	&K_n  \equiv (d_\theta/\chi) \log (nL_G) + v(\log n \vee \log( p\bar Ab/\sigma) ).\label{eq:Kn}
\end{align}
Throughout the paper we will typically consider $\gamma_n = n^{-c}$ for a suitably small constant $c>0$. In words, $\bar w_n$ is an effective measure of the noise in the problem, $t_n^\sigma$ is a uniform rate of convergence for the estimation of the standard deviation on the moment conditions, and $K_n$ controls the entropy associated with the class of functions induced by the moment conditions. Importantly, we will be able to approximate $\bar w_n$ via bootstrap simulations of the Gaussian process $\{\hat v^*_{\theta,j}: \theta \in \Theta, j\in [p]\}$ conditional on the data. Under mild conditions, we have that $\bar w_n \lesssim \sqrt{d_\theta \log(pn/\gamma_n)}$ and $t_n^\sigma \bar w_n = o(1)$.

\subsubsection{Discard Resampling or DR Bootstrap}

The first strategy to construct critical values is based on a bootstrap procedure that discards parameter values and moment inequalities that are problematic for an asymptotic approximation. The definition of this bootstrap statistic requires certain sample objects. For some sequence $M_n \geq \bar w_n$, let:
\begin{align*}
		\widehat \Theta_n(\eeta) & \subseteq \{ \theta \in \Theta(\eeta) : \max_{j\in[p]} \sqrt{n}\bar m_{\theta,j}/\hat\sigma_{\theta,j} = T_n(\eeta)\}\\
	\widehat \Psi_\theta & \equiv \{ j \in [p] : \sqrt{n}\bar m_{\theta,j}/\hat\sigma_{\theta,j}  \geq \max_{\tilde j\in[p]}\sqrt{n}\bar m_{\theta,\tilde j}/\hat\sigma_{\theta,\tilde j} - M_n\}.
\end{align*}
By definition, $\widehat \Theta_n(\eeta)$ is an arbitrary non-empty subset of the minimizers that define the MinMax statistic and, for each $\theta \in \Theta$, $\widehat \Psi_\theta$ denotes the subset of the moment inequalities that are sufficiently close to being binding. In principle, we can choose $\widehat \Theta_n(\eeta)$ to be any of the minimizers that define the MinMax statistic.

The DR bootstrap statistic is defined as follows:
\begin{equation}\label{Def:MSB}
	R_n^{DR\ast} ~~\equiv~~ \inf_{\theta \in \widehat \Theta_n(\eeta)  }~\max_{j\in\hat \Psi_\theta} ~\hat v_{\theta,j}^{\ast },\end{equation}
where $\hat v_{\theta,j}^{\ast }$ is the Gaussian multiplier bootstrap defined in (\ref{def:GMB}). For $\alpha \in (0,1)$, we define the corresponding conditional critical value as follows:
\begin{equation*}
	c_n^{DR}(\eeta,\alpha) ~\equiv~ \text{conditional}~~(1-\alpha)\mbox{-quantile of} \ \ R_n^{DR\ast} \ \ \mbox{given} \ (W_i)_{i=1}^n.
\end{equation*}

The following result establishes the relationship between the original MinMax statistic and the DR bootstrap statistic, and it exploits this result to show the asymptotic validity of the DR bootstrap approximation.

\begin{theorem}\label{thm:MSB:inferenceSubvector}
Assume that Condition MB is satisfied, $K_n^3 \leq n$, and that $M_n / \bar w_n \geq  ({L_G}/{\vartheta_n}+ 4 + t_n^\sigma{4}/{3} ){2(1+t_n^\sigma)^2}/{(1-t_n^\sigma)^2}$.
Then, under $H_0$,
$$ \P(T_n(\eeta) > t ) \leq \P(R^{DR*}_n > t - C\delta_{n}^{DR} ) + C\{\gamma_n+n^{-1}\}, $$
where $c$ and $C$  are constants that depend only on other constants of the problem, and
$$ \begin{array}{rl}
\delta_{n}^{DR}  & \equiv  \frac{Cb\sigma^2 K_n}{\gamma_n^{3/q}n^{1/2}} + \frac{C(b\sigma^2 K_n^2)^{1/3}}{\gamma_n^{1/3}n^{1/6}} +  C L_C^{1/2}\left(\frac{C\sigma K_n^{1/2}}{\gamma_n^{1/q}n^{1/2}\vartheta_n}\right)^{\chi/2} \frac{K_n^{1/2}}{\gamma_n^{1/q}} + \frac{CbK_n}{\gamma_n^{1/q}n^{1/2-1/q}}
 \end{array}.$$
Moreover, we have 
$$ \P(T_n(\eeta) > c_n^{DR}(\eeta,\alpha) ) \leq \alpha + C\mathcal{A}^{DR}_{n}\delta_n^{DR}+ C\{\gamma_n+n^{-1}\}, $$ where $\mathcal{A}^{DR}_{n}$ is as in (\ref{eq:AC}) but with $R^{*}_n$, $c_n(\eeta,\alpha)$, and $\delta_{n}$ replaced by $R^{DR*}_n$ and $c_n^{DR}(\eeta,\alpha)$, and $\delta_{n}^{DR}$, respectively.
 \end{theorem}

Theorem \ref{thm:MSB:inferenceSubvector} shows how the tail probability of the statistic of interest can be bounded by the tail probability of the bootstrap statistic (up to an approximation error). The result allows for  both $p$ and $d_\theta$ to increase with the sample size. In particular, $p\gg n$ is allowed. However, our proof of the validity of this procedure requires the choice of $M_n$ to be such that binding inequalities are not missed. In particular, it would suffice to choose $M_n$ such that $M_n/\bar w_n \to \infty$ in cases where $\vartheta_n \geq c$ and $L_G \leq C$, which are common in applications. Importantly, we can rigorously approximate $\bar w_n$ as the $1-n^{-c}$ quantile of $\sup_{\theta\in \Theta(\eeta), j\in [p]}|\hat v_{\theta,j}^*|$ which provides a data-driven approach to set $M_n$. Note that this represents an additional restriction relative to $M_n\to \infty$ required by \citet{bugni2017inference} in the context of finite number of moment inequalities.\footnote{Note that $\bar w_n$ would be uniformly bounded if $p$ and $d_\theta$ are fixed.} Instead, our requirement that $M_n/\bar w_n \to \infty$ indicates that $M_n$ is adaptive to the setting under consideration.

To guarantee that the DR bootstrap approximation provides asymptotic size control we require that $\gamma_n+\mathcal{A}^{DR}_{n}\delta_n^{DR} \approx 0$, i.e., the distribution DR bootstrap statistic cannot concentrate excessively at the quantile of interest relative to the approximation error $\delta_n^{DR}$. (See Section \ref{SEC:ANTICONCENTRATION} for further discussion.) The following corollary of Theorem \ref{thm:MSB:inferenceSubvector} provides simple sufficient conditions that covers many data generating process of interest.

\begin{corollary}
Assume that Condition MB holds, $K_n^3 \leq n$, $L_G/\vartheta_n \leq C_1$, $\bar w_n/M_n \leq \log^{-1} n$, and that $\gamma_n + \mathcal{A}^{DR}_{n}\delta_n^{DR} \leq C_1n^{-c_1}$. Then, under $H_0$, $$ \P( T_n(\eeta) > c_n^{DR}(\eeta,\alpha) ) \leq \alpha + Cn^{-c},$$ where $c$ and $C$ are constants that depend only on other constants of the problem.
\end{corollary}

\begin{remark}[Anti-concentration]
In cases where we choose $\widehat \Theta_n(\eeta)$ to be a singleton, we can show that $\mathcal{A}^{DR}_{n} \leq C\log^{1/2} p$ by the known anti-concentration properties associated with the maximum of Gaussian variables. Moreover, we have $\mathcal{A}^{DR}_{n}(W) \leq C\{\log (1+|\widehat\Psi_{\hat\theta}|)\}^{1/2}\leq C\log^{1/2} p$. Moreover, Lemma \ref{lem:density} in the appendix shows that if the cardinality of $\widehat \Theta_n(\eeta)$ is bounded by $k$, then we have $\mathcal{A}^{DR}_{n} \leq Ck\log^{1/2} p$. \footnote{We conjecture the dependence is a logarithmic factor in the covering number of $\widehat \Theta_n(\eeta)$. (See Section \ref{SEC:ANTICONCENTRATION} for a discussion.)}
\end{remark}

\subsubsection{Penalized Resampling or PR Bootstrap}

The second strategy to construct critical values is based on a bootstrap procedure that takes into account the non-centered feature of the empirical process via penalization. The Penalized Resampling or PR bootstrap statistic is defined as follows
\begin{equation}\label{Def:MPB}R_n^{PR\ast }\equiv \inf_{\theta \in \Theta (\eeta) }\max_{j\in[p]}\left\{ \hat v_{\theta,j}^{\ast } +\kappa^{-1}_n\sqrt{n}\bar{%
m}_{\theta,j} /\hat{\sigma}_{\theta,j} \right\}\end{equation}
where $\kappa_n \geq 1$ is a user-specified tuning parameter and the $\hat v^{\ast}_{\theta,j}$ is the bootstrap process defined in (\ref{def:GMB}). For $\alpha \in (0,1)$, we define the corresponding conditional critical value as follows:
\begin{equation*}
	c_n^{PR}(\eeta,\alpha) ~\equiv~ \text{conditional}~~(1-\alpha)\mbox{-quantile of} \ \ R_n^{PR\ast} \ \ \mbox{given} \ (W_i)_{i=1}^n.
\end{equation*}

Similarly to Theorem \ref{thm:MSB:inferenceSubvector}, the following result establishes the relationship between the original MinMax statistic and the PR bootstrap statistic, and it exploits this result to show the asymptotic validity of the PR bootstrap approximation.


\begin{theorem}\label{thm:PB:inferenceSubvector}
Assume that Condition MB is satisfied and $K_n^3\leq n$. Then, under $H_0$,
$$ \P(T_n(\eeta) \geq t ) \leq \P(R^{PR*}_n \geq t - C\delta_{n}^{PR} ) + C\{\gamma_n+n^{-1}\}, $$
where $c$ and $C$  are constants that depend only on other constants of the problem, and
$$ \delta_{n}^{PR} \equiv
\frac{L_G\kappa_n \sigma^2K_n}{\gamma_n^{2/q}n^{1/2}\vartheta_n^2} +  \frac{(b\sigma^2 K_n^2)^{1/3}}{\gamma_n^{1/3}n^{1/6}} + \frac{(b\sigma)^{1/2} K_n^{3/4}}{\gamma_n^{1/q}n^{1/4}} + \frac{b K_n}{\gamma_n^{1/q}n^{1/2-1/q}} $$ $$ + \frac{\bar w_n}{\kappa_n}+ L_C^{1/2}\left(\frac{\kappa_n\sigma K_n^{1/2}}{n^{1/2}\vartheta_n\gamma_n^{1/q}}\right)^{\chi/2} \frac{K_n^{1/2}}{\gamma_n^{1/q}}.$$
Moreover, we have 
 $$ \P(T_n(\eeta) \geq c_n^{PR}(\eeta,\alpha) ) \leq \alpha + C\mathcal{A}^{PR}_{n}\delta_{n}^{PR} + C\{\gamma_n+n^{-1}\}^{1/2},  $$  where $\mathcal{A}^{PR}_{n}$ is as in (\ref{eq:AC}) but with $R^{*}_n$, $c_n(\eeta,\alpha)$, and $\delta_{n}$ replaced by $R^{PR*}_n$ and $c_n^{PR}(\eeta,\alpha)$, and $\delta_{n}^{PR}$, respectively.
\end{theorem}

Theorem \ref{thm:PB:inferenceSubvector} bounds the probability distribution function of $T_n(\eeta)$ based on approximation errors and the probability distribution of $R^{PR*}_n$. The core of the proof constructs intermediary processes for which we can apply Theorems \ref{thm:clt:minmax} and \ref{thm:clt:minmax:Gaussian}. 
Theorem \ref{thm:PB:inferenceSubvector} also provides guideline on how to choose $\kappa_n$. Indeed we need to ensure that $\kappa_n/\bar w_n \to \infty$ and $n^{-1/2}\kappa_n (L_G\sigma^2K_n/\vartheta_n^2)\to 0$. These highlights the role of considering many moment inequalities on the threshold sequence. In fact, the moment inequality literature with a fixed number of moment conditions typically just requires that $\kappa_n\to \infty$ and $n^{-1/2}\kappa_n \to 0$.

Provided that $\gamma_n$ such as $\gamma_n + \delta_{n}^{PR} \mathcal{A}^{PR}_{n} \leq C_2n^{-c_2}$, Theorem \ref{thm:PB:inferenceSubvector} implies that the critical value based on the PR bootstrap approximation provides asymptotic size control. However, note that the anti-concentration condition impacts the choice of threshold sequence $\kappa_n$ as it impact $\delta_{n}^{PR}$. (See Section \ref{SEC:ANTICONCENTRATION}.) Corollary \ref{cor:mainresult1} stated earlier provided conditions under which we obtain (\ref{eq:CSvalid}).

\begin{remark}[Refinements on centering]
For the vector inference problem, \citet{romano2014practical} discussed an alternative approach to incorporate the centering in the bootstrap-based statistics. Also, see the discussion in Comment 4.4 in \citet{chernozhukov2013testing}. In the vector inference problem, the hypothesis $H_0:\theta^* = \bar{\theta}$ completely specifies the true parameter value. Letting $\mu_j=\Ep[m_j(\bar\theta,W_1)]$, the value $\tilde \mu_{j}= \min\{0, \hat \mu_j + \hat \sigma_{j} \bar c_n/\sqrt{n}\}$ which satisfies $\mu_j \leq \tilde \mu_j$ with probability exceeding $1-\gamma_n$ under $H_0$ by taking $\bar c_n$ as the $1-\gamma_n$ quantile of the effective noise $\max_{j\in [p]} \sqrt{n}(\hat \mu_j - \mu_j)$. Thus $\sqrt{n}\tilde \mu_{j}/\hat\sigma_j$ can be used as a valid centering for the vector inference problem to construct critical values. In the subvector inference problem considered in this paper, we have $\Ep[m_j(W,\theta)] \leq \bar m_{\theta,j} + \hat \sigma_{\theta,j} \bar w_n/\sqrt{n}$ for all $\theta \in \Theta(\eeta)$ and $j\in [p]$ with probability exceeding $1-\gamma_n$. However, we cannot take the minimum with zero as $H_0:h(\theta^*) = \bar{h}$ does not completely specify completely true parameter value. When compared with the recentered used in the PR bootstrap as defined in (\ref{Def:MPB}), we note that neither recentering
$$ \kappa_n^{-1}\sqrt{n}\bar m_{\theta,j}/\hat\sigma_{\theta,j} \ \ \ \mbox{and} \ \ \ \sqrt{n}\bar m_{\theta,j}/\hat \sigma_{\theta,j} + \bar w_n $$
dominates each other in general. Indeed, the recentering for the PR bootstrap will be smaller if and only if
$$ -\bar w_n < (1-\kappa_n^{-1})\bar m_{\theta,j}/\hat\sigma_{\theta,j}.$$
Therefore we can take a recentering
$$ \min\{ \kappa_n^{-1}\sqrt{n}\bar m_{\theta,j}/\hat\sigma_{\theta,j}, \ \ \sqrt{n} \bar m_{\theta,j}/\hat \sigma_{\theta,j} + \bar w_n \}.$$
The analysis of a bootstrap procedure based on this recentering follows essentially the same proof as of the penalized bootstrap and therefore it is omitted.
\end{remark}


\begin{remark}[Data-driven choice of $\kappa_n$]
Implicitly, both the PR bootstrap statistic $R_n^{PR*} =R_n^{PR*}(\kappa_{n})$ and its anti-concentration parameters $\mathcal{A}^{PR}_{n}=\mathcal{A}^{PR}_{n}(\kappa_{n})$ and $\mathcal{A}_n^{PR}(W,\kappa_n)$ depend on the threshold sequence $\kappa_{n}$. The following is an implementable procedure to make an adaptive data-driven choice for $\kappa_n$. \\
~\\
Step 1. Approximate $\bar w_n$ as the $1-n^{-c}$ quantile of $\sup_{\theta\in \Theta(\eeta), j\in [p]}|\hat v_{\theta,j}^*|$.\\
Step 2. Compute the anti-concentration parameter $\mathcal{A}^{PR}_{n}(W,\kappa)$ of $R_n^{PR*}(\kappa)$. \\
Step 3. Define $\kappa_n \equiv \inf\{ \kappa : \kappa \geq \bar w_n \mathcal{A}^{PR}_{n}(W,\kappa) n^{c}\}$.\\
Step 4. Compute critical value $c_n^{PR}(\eeta,\alpha)$ based on $R_n^{PR*} \equiv R_n^{PR*}(\kappa_n)$.
~\\


\end{remark}

\subsubsection{Minimum Resampling or MR Bootstrap}

We can combine the two previous bootstrap approximations by taking the minimum of their bootstrap statistics. Formally, the Minimum Resampling or MR bootstrap statistic is defined as follows:
$$ R^{MR\ast}_n = \min\{ R^{DR\ast}_n, R^{PR\ast}_n\}. $$
For $\alpha \in (0,1)$, we define the corresponding critical values as follows
$$c_n^{MR*}(\eeta,\alpha) = \mbox{conditional} \ \ (1-\alpha)\mbox{-quantile of} \ \ R_n^{MR\ast} \ \ \mbox{given} \ (W_i)_{i=1}^n.$$

The following result builds on Theorems \ref{thm:MSB:inferenceSubvector} and \ref{thm:PB:inferenceSubvector}. It establishes a relationship between the original MinMax statistic $T_n(\eeta)$ and the MR bootstrap statistic $R_n^{PR\ast}$, and it shows the asymptotic validity of the MR bootstrap approximation.

\begin{theorem}\label{thm:MT:inferenceSubvector}
Assume that the conditions in Theorems \ref{thm:MSB:inferenceSubvector} and \ref{thm:PB:inferenceSubvector} hold. Then, under $H_0$,
$$ \P(T_n(\eeta) \geq t ) \leq \P(R^{MR*}_n \geq t - C\delta_{n}^{MR} ) + C\{\gamma_n+n^{-1}\}, $$
where $c$ and $C$  are constants that depend only on other constants of the problem, $\delta_{n}^{MR} \equiv \delta_{n}^{DR}+\delta_{n}^{PR}$, and $\delta_{n}^{DR}$ and $\delta_{n}^{PR}$ are as defined in Theorems \ref{thm:MSB:inferenceSubvector} and \ref{thm:PB:inferenceSubvector}, respectively.
Moreover, we have 
 $$ \P(T_n(\eeta) \geq c_n^{MR}(\eeta,\alpha) ) \leq \alpha + C\mathcal{A}^{MR}_{n}\delta_{n}^{MR} + C\{\gamma_n+n^{-1}\}^{1/2},  $$  where $\mathcal{A}^{MR}_{n}$ is as in (\ref{eq:AC}) but with $R^{*}_n$, $c_n(\eeta,\alpha)$, and $\delta_{n}$ replaced by $R^{MR*}_n$ and $c_n^{MR}(\eeta,\alpha)$, and $\delta_{n}^{MR}$, respectively.
\end{theorem}

The proof of Theorem \ref{thm:MT:inferenceSubvector} shows how the finite sample analysis used here lands itself for the combination of valid procedures. The approach allows us to avoid considering the limit distribution and the issues associated with its existence. Indeed, we do not require the existence of a distributional limit but instead control approximation errors for each $n$.

%

\subsection{Self-Normalization based Critical Values}\label{Sec:SN:cv}

In this section, we discuss inference using critical values based on self-normalized moderate deviation theory. Although the resulting inference is potentially more conservative than the one based on the bootstrap, the method is easier to compute and asymptotically valid for a wider class of data generating processes.

The self-normalized inference was originally proposed by \citet{chernozhukov2013testing} in the context of inference on the parameter vector $\theta^*$ in the partially identified many moment inequality model. The ideas proposed in this section follow closely the arguments in \citet{chernozhukov2013testing}.

For $\alpha \in (0,1)$, we define the self-normalized or SN critical value as follows:
\begin{equation}\label{def:cSN}
	c_{n}^{SN}(p,\alpha) ~\equiv~ \frac{\Phi^{-1}(1-\alpha/p)}{\sqrt{1-\Phi^{-1}(1-\alpha/p)^2/n}},
\end{equation}
where $\Phi$ is the cumulative distribution function of the standard normal distribution.


\begin{theorem}[Validity of SN method]\label{thm:SNresult}
	Assume that $H_0$ holds, i.e., $\theta^{*} \in \Theta(\eeta)$, and that there are constants $0<c_1<1/2$ and $C_1 > 0$ such that
$$ \max_{j\in[p]}\Ep[|\sigma_{\theta^*,j}^{-1}\{m_{j}(W, \theta^*)- \Ep[m_{j}(W,\theta^*)]\}|^3] \log^{3/2}(p/\alpha) \leq C_1 n^{1/2-c_1}.$$
Then,
$$ \P ( T_n(\eeta) > c_{n}^{SN}(p,\alpha) ) \leq \alpha + Cn^{-c},$$
where $c$ and $C$ are constants that depend only on $\alpha,c_1,C_1$.
\end{theorem}

The proof of Theorem \ref{thm:SNresult} follows from the proof of Theorem 4.1 in \citet{chernozhukov2013testing}. Note that the assumptions of  Theorem \ref{thm:SNresult} are only required to hold at the true parameter value $\theta^*$, rather than every $\theta \in \Theta(\eeta)$.

Next, we discuss a two-step version of the SN procedure. Similar to analogous procedures defined in \citet{chernozhukov2013testing}, our two-step SN procedure takes advantage of restricting attention to relevant parameter values and moment inequalities. However, unlike the two-step procedure defined in \citet{chernozhukov2013testing}, our two-step SN procedure takes into account the fact that $H_0$ does not completely specify the true parameter vector $\theta^*$.

For every $\theta \in \Theta(\eeta)$, and $\gamma_n \in (0,\alpha/4)$ define the following objects:
\begin{align*}
	\hat J_n^{SN}(\theta) &~\equiv~ \{ j \in [p]: \sqrt{n} \bar m_{\theta,j}/\hat \sigma_j(\theta)  > - 2c_{n}^{SN}(p,\gamma_n)\}\\
	c_{n}^{SN,2S}(\theta,\alpha)& ~\equiv~ \left\{ \begin{array}{rl}\frac{\Phi^{-1}(1-(\alpha-3\gamma_n)/|\hat J_n^{SN}(\theta)|}{\sqrt{1-\Phi^{-1}(1-(\alpha-3\gamma_n)|\hat J_n^{SN}(\theta)|)^2/n}}, & \mbox{if} \  |\hat J_n^{SN}(\theta)|\geq 1,\\
	0, & \mbox{if} \  |\hat J_n^{SN}(\theta)|=0. \end{array}\right.\\
	\widehat \Theta_n^{SN} &~\equiv~ \{ \theta \in \Theta(\bar h): \ \max_{j\in[p]} \sqrt{n}\bar m_{\theta,j}/\hat\sigma_{\theta,j} \leq c^{SN}(p,\gamma_n)\}.
\end{align*}
For every candidate parameter $\theta \in \Theta(\eeta)$, $\hat J_n^{SN}(\theta)$ is the subset of moment inequalities that are sufficiently close to binding and $c_{n}^{SN,2S}(\theta,\alpha)$ is the two-step SN critical value associated with the parameter value $\theta$. Finally, $\widehat \Theta_n^{SN}$ is a subset of $\Theta(\eeta)$ that is sufficiently close to the identified set.


\begin{theorem}[Validity of two-step SN method]\label{thm:SNresult2S}
	Assume that $H_0$ holds, i.e., $\theta^{*} \in \Theta(\eeta)$, $\gamma_n \in (0,\alpha/4)$, and that there are constants $0<c_1<1/2$ and $C_1 > 0$ such that
	\begin{align*}
		\max_{j\in[p]}\Ep[|\sigma_{\theta^{*},j}^{-1}\{m_{j}(W, \theta^{*})- \Ep[m_{j}(W,\theta^{*})]\}|^3] \log^{3/2}(p/\alpha) \leq C_1 n^{1/2-c_1} \\
		\{\Ep[\max_{j\in[p]}|\sigma_{\theta^{*},j}^{-1}\{m_{j}(W, \theta^{*})- \Ep[m_{j}(W,\theta^{*})]\}|^4]\}^{1/2}\log^{2}(p/\gamma_n) \leq C_1 n^{1/2-c_1}.
	\end{align*}
Then,
$$ \P \left( T_n(\bar h) > \max_{\theta\in \widehat\Theta_n^{SN}} c_{n}^{SN,2S}(\theta,\alpha) \right) \leq \alpha + Cn^{-c}.$$
where $c$ and $C$  are constants that depend only on other constants of the problem.
\end{theorem}

Theorem \ref{thm:SNresult2S} shows that the two-step approach provides asymptotic size control under mild conditions compared to the bootstrap-based methods. For example, we are not making polynomial minorant assumption as in Condition MB(iii). The need for taking the maximum over $\widehat \Theta_n^{SN}$ arises from the fact that $H_0$ does not completely determine the true parameter value. As in Theorem \ref{thm:SNresult}, note that the moment assumptions of Theorem \ref{thm:SNresult2S} are only required to hold at the true parameter value $\theta^*$. The proof builds upon Theorem 4.2 in \citet{chernozhukov2013testing} but the additional selection over $\widehat \Theta_n^{SN}$ requires controlling additional approximation errors.

\section{Anti-concentration and Data-driven estimates}\label{SEC:ANTICONCENTRATION}


In the high-dimensional settings the statistics of interest might have no limiting distribution and the use of simulation procedures (e.g. bootstrap) are useful to  approximate the distribution of the statistics of interest, except for some special pivotal statistics (e.g. \citet{BC-SparseQR} for quantile regression). However, even if the approximation has a vanishing error, such error can distort coverage. The anti-concentration property ensures that the random variable we aim to simulate does not concentrate too much at relevant points of the distribution. Unless the approximation errors introduced by the coupling procedures are of a smaller order of the rate in which the random variable concentrate, we cannot reliably use the simulated quantiles to approximate the quantile of the original statistics. The anti-concentration properties for the maximum of (correlated and non-centered) Gaussian random variables have been understood recently, see \citet{chernozhukov2012comparison,chernozhukov2015noncenteredprocesses}. This is a prime example which covers several applications of interest including the construction of simultaneous confidence bands in high-dimensional regression settings, e.g., \citet{BCK-LAD,belloni2015uniformly}.

\begin{remark}[Anti-concentration for the Max]\label{Remark:Max}
Consider a $p$-dimensional Gaussian random vector $W\sim N(0,\Sigma)$, $\Sigma_{jj}= 1$, $j\in[p]$, $p\geq 2$. It is known that $Z=\max_{j\in[p]}W_j$ concentrates around $\Ep[Z]$. However, although it concentrates it does not concentrate too fast. Indeed, the probability density function $f_Z $of $Z=\max_{j\in[p]}W_j$ is bounded by $C\sqrt{\log p}$, see \citet{chernozhukov2012comparison}, so that
$$\sup_{t\in\mathbb{R}}\P( |Z-t|\leq \epsilon ) \leq \epsilon \max_{t'\in\mathbb{R}}f_Z(t') \leq \epsilon C\sqrt{\log p}. $$ Such feature provides a bound on the impact of the approximation error of the coupling on the estimation of the probability distribution of the original process via the bootstrap process. In the case we couple a process for the maximum of $p$ functions, under typical assumptions, the approximation errors from the coupling techniques are of the order $n^{-1/6}\log^{2/3}p$, so that we can reliably use the multiplier bootstrap provided $ \{n^{-1/6}\log^{2/3}p\} \{ \log^{1/2}p \} = n^{-1/6}\log^{7/6}p = o(1)$.
\end{remark}

In this section we are concerned with the anti-concentration property associated with the MinMax statistic that we propose for the problem of subvector inference in partially identified models with many moment inequalities. In contrast to the Max operator, the MinMax operation is not a convex function of its arguments, and this has substantial implications for our derivations. The anti-concentration bounds based on the maximum density function for the MinMax are qualitatively different than the case of the max discussed in Remark \ref{Remark:Max}. The following lemma shows this in the case of the statistic $Z = \min_{k\in [N]}\max_{j\in [p]} W_{kj}$ where $W \in \mathbb{R}^{N\times p}$ is a random matrix with i.i.d.\ $N(0,1)$ entries.

\begin{proposition}\label{lemma:AnticoncentrationNormal}
For $W_{kj}\sim N(0,1)$, i.i.d., $k\in[N]$, $j\in[p]$, define the random variable $Z = \min_{k\in [N]}\max_{j\in [p]} W_{kj}$. Provided that $p/\sqrt{2\pi} > \log(Np) \geq 2$, we have that the probability density function $f_Z$ satisfies:
$$\begin{array}{lrl}
&\displaystyle   (i) \max_{t'\in \mathbb{R}} f_Z(t') & \leq 2\{\sqrt{2}+2\}\log^{3/2}(Np) \\
 & &  \\
&\displaystyle (ii) \max_{t'\in \mathbb{R}} f_Z(t') & \geq  \left\{\sqrt{2}\log^{1/2}\left(\frac{{\displaystyle p}}{\sqrt{2\pi}\log N}\right) - 2\right\} {\displaystyle \frac{\log(N)}{e}}.
  \end{array}$$
\end{proposition}

Proposition \ref{lemma:AnticoncentrationNormal} shows that the concentration properties of the MinMax case is different than the concentration properties of the Max case even for i.i.d. standard normal random variables. In particular, Proposition \ref{lemma:AnticoncentrationNormal}(ii) shows that in the high-dimensional case with $p=N$, the  probability density function has values of the order $\log^{3/2} p$ which contrasts with the max operator. Indeed, the random variable  $\max_{k,j\in[p]} W_{kj}$ has the maximum of its probability density function to be at most of the order of $\log^{1/2} p$. Nonetheless, Proposition \ref{lemma:AnticoncentrationNormal}(i) suggests it is possible for the anti-concentration parameter to increase logarithmically with the number of moment inequalities $p$.

We propose a data-driven procedure to estimate the anti-concentration property associated with the statistic of interest. This is applicable not only to the MinMax functional but to other functionals.  This is particularly relevant in settings where it is hard to derive analytical bounds or when such bounds might be conservative. Indeed in the case of standardized by non-centered processes it is likely that the statistics of interest (and their anti-concentration property) will depend on a subset of the process associated with points that are more likely to determine the statistics. In that case, even if theoretical bounds for how fast some statistics concentrate are available, they can be conservative. Therefore, the use of an adaptive tool can be desired.


We now note several features of the anti-concentration parameter $\mathcal{A}_{n}$ associated to a bootstrap statistic $R_n^*$, as defined in (\ref{eq:AC}). First, note that the anti-concentration quantity defined in (\ref{eq:AC}) trivially satisfies $$\mathcal{A}_n \leq \max_{t'\in \mathbb{R}} f_{R^*_n}(t')$$ but it could be much smaller. Second, it allows for the distribution of $R^*_n$ to have a point mass and still be nicely bounded as we restrict the size of the smallest set.\footnote{This could be of interest to consider other bootstrap procedures other than the Gaussian multiplier bootstrap considered in this paper.} Third, we can construct upper bounds for $\mathcal{A}_{n}$ and $\mathcal{A}_{n}(W)$ by using a lower bound on $\delta_n$. In most cases of interest, $n^{-1/2}$ constitutes such a lower bound which makes upper bounding $\mathcal{A}_{n}(W)$ feasible.

A sufficient condition for the concentration of the statistics not to impact the size is $\mathcal{A}_{n} \delta_n \to 0$.  This ensures that the process $R_n^*$ has suitable anti-concentration not to introduce additional size distortions because of the coupling approximation errors. The next proposition connects the control of distortions via the quantity $\mathcal{A}_{n}(W)$.

\begin{theorem}\label{thm:CondAntiConcentration}
Assume that the conditions in Theorems \ref{thm:MSB:inferenceSubvector} and \ref{thm:PB:inferenceSubvector} hold.\\
(i) Then, under $H_0$, with probability $1-C\{\gamma_n+n^{-1}\}$ we have $$\begin{array}{rl}
\P(T_n(\eeta) > t ) & \leq \P(R^{DR*}_n > t - C\delta_{n}^{DR} \mid (W_i)_{i=1}^n ) + C\{\gamma_n+n^{-1}\}\\
\\
\P(T_n(\eeta) > c_n(\eeta,\alpha) ) & \leq \alpha + C\mathcal{A}_n^{DR}(W)\delta_{n}^{DR}+ C\{\gamma_n+n^{-1}\},\end{array}$$
\noindent (ii) Then, under $H_0$, with probability $1-C\{\gamma_n+n^{-1}\}$ we have $$\begin{array}{rl}
\P(T_n(\eeta) > t ) & \leq \P(R^{PR*}_n > t - C\delta_{n}^{PR} \mid (W_i)_{i=1}^n ) + C\{\gamma_n+n^{-1}\}\\
\\
\P(T_n(\eeta) > c_n(\eeta,\alpha) ) & \leq \alpha + C\mathcal{A}_n^{PR}(W)\delta_{n}^{PR}+ C\{\gamma_n+n^{-1}\},\end{array}$$
\end{theorem}

Theorem \ref{thm:CondAntiConcentration} accounts for the fact that the selection of inequalities and centering are data-driven and therefore it does not follow from a direct application of Markov inequality and Theorems \ref{thm:MSB:inferenceSubvector} and \ref{thm:PB:inferenceSubvector}.

The next result shows how coupling results can be used to provide upper bounds on the unconditional anti-concentration quantity $\mathcal{A}_n$ based on the conditional anti-concentration quantity $\mathcal{A}_n(W)$ with high probability.

\begin{theorem}\label{thm:ANTI-consistency}
Suppose that with probability $1-\gamma_n$ we have
$$ \P( |R_n^* - c_n(\eeta,\alpha) | \leq \epsilon ) \leq \P( |R_n^* - c_n(\eeta,\alpha) | \leq \epsilon + \tilde \delta_n \mid (W_i)_{i=1}^n ) + \gamma_n$$
Then, with probability $1-\gamma_n$ we have
$$  \mathcal{A}_n \leq \mathcal{A}_n(W) (1+\tilde\delta_n/\delta_n)+\gamma_n/\delta_n$$

\end{theorem}

\begin{remark} We note that $\delta_n$ in the definition of the anti-concentration is based on the coupling of the original non-Gaussian process while $\tilde \delta_n$ is based on the coupling of the unconditional and conditional multiplier bootstrap in (\ref{def:CouplingANTI}). In particular, note that $R_n^*$ is the MinMax of symmetric random variables which allows to improve the coupling bounds which leads $\tilde \delta_n$ to be of smaller order of magnitude than $\delta_n$, i.e. $\tilde\delta_n = o(\delta_n)$. Therefore, provided we set $\gamma_n = o(\delta_n)$, with high probability we obtain a meaningful bound to $\mathcal{A}_n$ based on $\mathcal{A}_n(W)$.
\end{remark}

Although this procedure provides a way to conduct inference, theoretical bounds for the anti-concentration play an important role to inform of the requirements on the sample size that are needed for the overall validity of the procedure. Nonetheless, it is conceivable that in more complex applications the anti-concentration can vary substantially with the data generating process. To provide such additional adaptivity is precisely the goal of the proposed data-driven procedure.

\begin{remark}[Global anti-concentration] Note that the anti-concentration (\ref{eq:AC}) is a local measure of the anti-concentration around the point $c_n(\eeta,\alpha)$. Indeed, one could change our definition for a more global definition of anti-concentration such as $\sup_{t,\epsilon \geq n^{-1/2}} \frac{1}{2\epsilon}\P( |R^*_n - t | \leq \epsilon \mid (W_i)_{i=1}^n )$. However, we choose the local definition of anti-concentration as the global version could be unnecessarily conservative.
\end{remark}

\begin{remark}[Alternative definitions of the data-driven estimator]
	Rather than (\ref{eq:AC}), we could propose alternative data-driven estimators of the anti-concentration. For example, we could use $   \bar l_{n} \equiv \inf\{ l : \P( |R^*_n - c_n(\eeta,\alpha) | \leq l \mid (W_i)_{i=1}^n ) \leq m_n\} $ where $m_n = o(1)$ is a pre-specified sequence (e.g.\ $\log^{-1} n$). Provided that $\delta_n = o(\bar l_n)$, no additional asymptotic size distortions are introduced.
\end{remark}

\section{Coupling for MinMax Functional}\label{SEC:MinMaxCoupling}

In this section we state results for coupling the MinMax  of the sum of independent empirical processes with the MinMax of Gaussian processes. Such results are crucial in our analysis and may be of independent interest. They build upon and complement recent results on the literature for the maximum of processes due to \citet{chernozhukov2012Gaussian,chernozhukov2014clt,chernozhukov2014honestbands,chernozhukov2012comparison,chernozhukov2015noncenteredprocesses}.

\subsection{MinMax of Empirical Processes}\label{SEC:PROCESSES}

Let $B:\mathcal{F}\to \mathbb{R}$ be a given functional. For $\eta >0$, let $N_B(\eta)$ denote the cardinality of a minimum $\eta$-cover of $\mathcal{F}$, i.e., $N_B(\eta)$ is the minimal integer $N$ such that there exist $f_1,\ldots,f_N \in \mathcal{F}$ with the property that for every $f\in\mathcal{F}$, $|B(f) - B(f_j)|< \eta$ for some $j\leq N$.
In what follows, we let $(X_i)_{i=1}^{n}$ be i.i.d.\ random processes mapping $\mathcal{F}\to \mathbb{R}$, and we define $\Gn(f) \equiv n^{-1/2}\sum_{i=1}^n \{X_i(f)-\Ep[X_i(f)]\}$. Also, for i.i.d.\ random variables $\xi=(\xi_i)_{i=1}^n$ independent from $(X_i)_{i=1}^{n}$, we let $\Gn^\xi(f) \equiv n^{-1/2}\sum_{i=1}^n \xi_i\{X_i(f)-\En[X_i(f)]\}$. We say a class of functions is VC type with envelope $F$ and constants $\bar A$ and $v$ if for each $\epsilon>0$, its covering number satisfies $N(\mathcal{F},\|\cdot\|_Q,\epsilon\|F\|_Q)\leq (\bar A/\epsilon)^v$, see \cite{Dudley99} for definitions.


 ~\\

\noindent {\bf Condition A.} {\it
(i) There exists a countable subset $\mathcal{G}$ of $\mathcal{F}$ such that for every $f\in \mathcal{F}$ there is a sequence $g_m\in\mathcal{G}$ with $g_m\to f$ pointwise and $B(g_m) \to B(f)$. (ii) The class of functions $\mathcal{F}$ is VC type with measurable envelope $F$ and constants $\bar A\geq e$ and $v\geq 1$. (iii) There exists constants $b\geq \sigma>0$ and $q\in [4,\infty)$ such that $\sup_{f \in \mathcal{F}} \Ep[ |f|^k ] \leq \sigma^2 b^{k-2}$, $k=2,3,4$, and $\|F\|_{P,q}\leq b$.
}

~\\

Conditions A(i)-(iii) correspond to the conditions in \citet{chernozhukov2015noncenteredprocesses}. These assumptions imply the existence of a centered Gaussian process $G_P$ indexed by $\mathcal{F}$ with uniformly continuous sample paths with respect to the $\mathcal{L}^2(P)$-seminorm $d(f,g)=\{\Ep[(f-g)^2]\}^{1/2}$ and covariance operator given $\Ep[G_P(f)G_P(g)]={\rm cov}(f(X),g(X))$.

Next we turn to conditions that are specific to the setting we consider.


~\\
\noindent {\bf Condition B.} {\it (i) The class of functions $\mathcal{F}$ can be written as $\mathcal{F} = \{ f \in \mathcal{F}_\theta, \ \mbox{for some} \ \  \theta \in \Lambda \}$ where $\Lambda \subset \mathbb{R}^{d_\theta}$ satisfies ${\rm diam}(\Lambda) \leq n$. (ii) There is some ${\bar \chi} > 0$ such that for $\gamma_n \in (0,1)$, there are constants $C_{\mathcal{F},\gamma_n}$ and $\eta_{\mathcal{F},\gamma_n}$ such that for any $\epsilon >0$  we have {\small $$P\left( \sup_{\theta,\tilde\theta\in \Lambda, \|\theta - \tilde \theta\|\leq \epsilon} \left| \sup_{f\in \mathcal{F}_\theta} B(f)+ \G(f) - \sup_{\tilde f\in \mathcal{F}_{\tilde \theta}} B(\tilde f)+ \G(\tilde f)  \right| > C_{\mathcal{F},\gamma_n}\epsilon^{\bar \chi} + \eta_{\mathcal{F},\gamma_n} \right) \leq \gamma_n, $$} for $\G = G_P, \Gn, \Gn^\xi$, and with $\xi=(\xi_i)_{i=1}^n$ are i.i.d.\ $N(0,1)$.}

~\\

Condition B(i) postulates the function class is parameterized by two indices. Condition B(ii) is a high level condition that allows us to construct a net for the functionals $\{ \sup_{f\in \mathcal{F}_\theta} B(f)+ \Gn(f) : \theta \in \Lambda\}$ and $\{ \sup_{f\in \mathcal{F}_\theta} B(f)+ G_P(f) : \theta \in \Lambda\}$. It is implied by typical equicontinuity assumptions imposed on the whole process $\{ B(f)+ \Gn(f) : \theta \in \Lambda, f \in \mathcal{F}_\theta\}$ (see \citet{vdV-W} for general definitions and for moment inequalities models see \citet{andrews2014nonparametric,andrews2013inference,bugni2017inference}). However, Condition B(ii) allows for non-Donsker classes as the constant $C_{\mathcal{F},\gamma_n}$ can increase with the sequence of the data generating process. The case ${\bar \chi}=1$ and $\eta_{\mathcal{F},\gamma_n}=0$ covers the important case in which the functions in $\mathcal{F}_\theta$ and the operator $B$ are Lipschitz functions of $\theta$ whose constant can increase with the dimension $d_\theta$. The case ${\bar \chi}=1$ and $\eta_{\mathcal{F},\gamma_n}=o(1)$ allows for discontinuous functions. In Section \ref{SEC:INFERENCESUBVECTORPI} below we provide simple conditions that imply Condition B(ii).


Recently new central limit theorems for high-dimensional vectors have been derived in a sequence of papers by \citet{chernozhukov2012Gaussian, chernozhukov2014clt, chernozhukov2014honestbands, chernozhukov2012comparison, chernozhukov2015noncenteredprocesses}. They have been used to construct Gaussian approximations for the distribution of the maximum of the components as in the case of many moment studied in \citet{chernozhukov2013testing}. In this work, we build upon these tools and ideas to develop new Gaussian approximation for the MinMax statistic.

In what follows it will be convenient to define \begin{equation}\label{def:Kn}
K_n = \log N_B(\eta) + v\{\log n \vee \log( \bar Ab/\sigma) \} + (d_\theta/{\bar \chi})\log( nC_{\mathcal{F},\gamma_n}b/\sigma )
\end{equation}
which is an entropy measure associated with the class of function $\mathcal{F}$. The following theorem provides a key result for our analysis.

\begin{theorem}\label{thm:clt:minmax}
Suppose that Conditions A and B are satisfied and $K_n^3 \leq n$. Define the random variable $T=\inf_{\theta \in \Lambda}\sup_{f \in \mathcal{F}_\theta} B(f) + \Gn(f)$. Then, for every $\gamma_n \in (0,1)$, there exists a random variable $\widetilde T =_d \inf_{\theta \in \Lambda}\max_{f \in \mathcal{F}_\theta} B(f) + G_P(f)$ such that
$$\P( |T - \widetilde T| > C  \delta_{n,\eta,\gamma_n}  ) \leq  C'\{ \gamma_n + n^{-1}\},$$
where $C, C'$ are constants that depend only on $q$, $$\delta_{n,\eta,\gamma_n}= \frac{(b\sigma^2 K_n^2)^{1/3}}{\gamma_n^{1/3}n^{1/6}} + \frac{bK_n}{\gamma_n^{1/q} n^{1/2-1/q}} + \eta_{\mathcal{F},\gamma_n}+\eta,$$
and ${=}_{d}$ means equality in distribution.
\end{theorem}

Theorem \ref{thm:clt:minmax} above establishes that we can approximate the value of the MinMax statistic with the MinMax of Gaussian processes. A crucial step in the proof of Theorem \ref{thm:clt:minmax} was the development of a new smooth approximation function for the MinMax operator. Finally, in most applications the parameters $\eta_{\mathcal{F},\gamma_n}$ and $\eta$ are dominated by the other terms in $\delta_{n,\eta,\gamma_n}$ (e.g., see Section \ref{SEC:INFERENCESUBVECTORPI}).


Our next result shows that the distribution of the MinMax of the Gaussian process is close to the distribution of the MinMax of our bootstrap approximation.


\begin{theorem}\label{thm:clt:minmax:Gaussian}
	Suppose that Conditions A and B are satisfied and $K_n \leq n$ and let $S=\inf_{\theta \in \Lambda}\sup_{f \in \mathcal{F}_\theta} B(f) + \Gn^\xi(f)$. Then, for every $\gamma_n \in (0,1)$, there exists a random variable $\widetilde S {=}_{{{d\mid (X_i)_{i=1}^{n}}}} \inf_{\theta \in \Lambda}\sup_{f \in \mathcal{F}_\theta} B(f) + G_P(f)$ such that
$$\P( |S - \widetilde S| > C  \bar \delta_{n,\eta,\gamma_n}  ) \leq  C'\{ \gamma_n + n^{-1}\}$$
where $C, C'$ are  universal constants that depend on $q$, $$\bar\delta_{n,\eta,\gamma_n}= \frac{(b\sigma K_n^{3/2})^{1/2}}{\gamma_n^{1+1/q}n^{1/4}} + \frac{bK_n}{\gamma_n^{1+1/q} n^{1/2-1/q}}+ \eta_{\mathcal{F},\gamma_n}+\eta,$$
and ${=}_{{{d\mid (X_i)_{i=1}^{n}}}}$ means equality in conditional distribution given $(X_i)_{i=1}^n$.
\end{theorem}

The combination of Theorems \ref{thm:clt:minmax} and \ref{thm:clt:minmax:Gaussian} provide a constructive way to approximate the distribution of $T$ by simulating $S$ which is associated with the process induced by the multiplier bootstrap procedure.

\begin{remark}[Other Bootstrap Procedures]
It seems it possible to apply other bootstrap procedures. Of particular interest is the use of the  empirical bootstrap which allows to match higher order moments relative to the Gaussian multiplier bootstrap (which can lead to weaker requirements when compared to Theorem \ref{thm:clt:minmax}). Indeed, the main novelty of the proofs in Theorems \ref{thm:clt:minmax} and \ref{thm:clt:minmax:Gaussian} is the use of a new approximating function to handle the MinMax structure we are concerned here. Any improvements on coupling results for smooth functions can be incorporated.
\end{remark}

\subsection{MinMax of High-dimensional Matrices}\label{SEC:DISCRETIZED}

In this section we collect key results for coupling the MinMax of the components of  high-dimensional matrices. They are of independent interest and apply for non-i.i.d.\ matrices.
We first state a coupling result between (arbitrary) random matrices and Gaussian random matrices for the MinMax operator. In what follows let $\Var{X_i}=\Ep[(X_i-\Ep[X_i])(X_i-\Ep[X_i])']$.

\begin{theorem}\label{thm:clt:minmax:discrete}
Let $(X_i)_{i=1}^n$ be independent random matrices in $\mathbb{R}^{N\times p}$ (with $Np\geq 2$) with finite absolute third moments componentwise, and $(Y_i)_{i=1}^n$ be independent random matrices in $\mathbb{R}^{N\times p}$ with $Y_i \sim N(\Ep[X_i], \Var{X_i})$. Set $$Z=\min_{k\in [N]}\max_{j\in[p]}\frac{1}{\sqrt{n}}\sum_{i=1}^nX_{ikj} \ \ \mbox{and} \ \ \widetilde Z = \min_{k\in [N]}\max_{j\in[p]}\frac{1}{\sqrt{n}}\sum_{i=1}^nY_{ikj}.$$ Then for any $\delta>0$ and Borel subset $A$ of $\mathbb{R}$,
$$\P( Z \in A) \leq \P ( \widetilde Z \in A^{C\delta}) + \frac{C\log^2(Np)}{\delta^3n^{1/2}}\{ L_n + M_{n,X}(\delta)+M_{n,Y}(\delta)\},$$
where $C>0$ is a universal constant, and
$$\hspace{-0.2cm}\begin{array}{l}
\displaystyle L_n =  \max_{k\in [N], j\in [p]}\frac{1}{n}\sum_{i=1}^n\Ep[|\widetilde X_{ikj}|^3],\\
\displaystyle M_{n,X}(\delta) = \frac{1}{n}\sum_{i=1}^n\Ep\left[\max_{k\in [N], j\in [p]}|\widetilde X_{ikj}|^3 1\left\{\max_{k\in [N], j\in [p]}|\widetilde X_{ikj}|>\delta\sqrt{n}/\log(Np)\right\}\right],\\
\displaystyle M_{n,Y}(\delta) = \frac{1}{n}\sum_{i=1}^n\Ep\left[\max_{k\in [N], j\in [p]}|\widetilde Y_{ikj}|^3 1\left\{\max_{k\in [N], j\in [p]}|\widetilde Y_{ikj}|>\delta\sqrt{n}/\log (Np)\right\}\right],
\end{array}
$$  for $\widetilde X_i=X_{i}-\Ep[X_{i}]$ and $\widetilde Y_i=Y_{i}-\Ep[Y_{i}]$.
\end{theorem}

The following theorem establishing a coupling result regarding two Gaussian random matrices for the MinMax operator. This result is important to establish the validity of the Gaussian multiplier bootstrap which relies on an estimate of the covariance matrix of the original process.

\begin{theorem}\label{thm:clt:minmax:Gaussian:discrete}
Let $X$ and $Y$ be random matrices in $\mathbb{R}^{N\times p}$ ($Np\geq 2$) with $X\sim N(\mu,\Sigma^X)$ and $Y \sim N(\mu,\Sigma^Y)$. Define $T = \min_{k\in [N]}\max_{j\in[p]} X_{kj}$ and $\widetilde T = \min_{k\in [N]}\max_{j\in[p]} Y_{kj}$.  Then for any $\delta>0$ and Borel subset $A$ of $\mathbb{R}$,
$$\P( T \in A) \leq  \P( \widetilde T \in A^{5\delta}) + C\delta^{-2}\|\Sigma^X-\Sigma^Y\|_\infty \log(Np),$$
where  $\|\Sigma^Y-\Sigma^X\|_\infty= \max_{(k,j)\in [N]^2\times[p]^2} |\Sigma^Y_{(k_1j_1,k_2j_2)}-\Sigma^X_{(k_1j_1,k_2j_2)}|$ and $C>0$ is a universal constant.
\end{theorem}

Theorems \ref{thm:clt:minmax:discrete} and \ref{thm:clt:minmax:Gaussian:discrete} parallels the results for non-central random vectors for the max operator obtained in \citet{chernozhukov2015noncenteredprocesses}. Their proofs rely on the use of a novel smooth approximation of the MinMax that is a suitable composition of the logarithm of the sum of the exponentials.

\begin{appendix}

\section{Example}

\begin{example}\label{Ex:Fail}

Consider $d_\theta = 2$, $\Theta =[-1,1]^2$. Let $p=2$ and the following moment inequalities
 $$
\begin{array}{rl}
 \Ep[m_1(W_i,\theta)] & = \Ep[\theta_1 + \theta_2 - W_{i,1}] \leq 0\\
 \Ep[m_2(W_i,\theta)] & = \Ep[W_{i,2} - \theta_1 - \theta_2] \leq 0\\
 \end{array}
 $$ where $W_i \in \mathbb{R}^p$, $W_i\sim N(0,I)$ and we are interest on testing
 $$ H_0: \theta_1 = 0 \ \ vs. \ \ H_1:\theta_1 \neq 0 \ \ \ \   \mbox{so that} \ \ $$
  $$\Theta(h_0) = \{ \theta \in \Theta : \theta_1 = 0\} \ \ \mbox{and} \ \ \Theta_I = \{ \theta \in \Theta: \theta_1+\theta_2 = 0\}$$
It follows that for $(Z_1,Z_2)\sim N(0,I)$ we have
$$ T_n(0) = \min_{-1\leq \theta_2\leq 1} \max\left\{ \frac{\sqrt{n} \theta_2 - \bar W_1}{\hat\sigma_1}, \frac{\bar W_2 - \sqrt{n}\theta_2}{\hat\sigma_2} \right\} \to_d \frac{Z_2-Z_1}{2} \sim N(0,1/2)  $$

In contrast, setting $R_n^*(0) = \min_{\theta \in \Theta(0)} \max_{j\in[p]}\hat v_{\theta,j}$ we have
$$R_n^*(0)\mid (W_i)_{i=1}^n 
\to_d \min\{ -Z_1, Z_2 \} $$
Therefore, critical values based on the conditional quantiles of $R_n^*(0)$ given the data fail to control size. For example, for $\alpha = 0.1$, it follows $c_n(0,\alpha)\approx 0.5$ and $P(T_n(0)>c_n(0,\alpha))\approx 0.24$,
 while the $(1-\alpha)$-quantile of $T_n(0)$ is $\approx 0.86$.\end{example}

\section{Proofs of Section \ref{SEC:INFERENCESUBVECTORPI}}

Before proceeding with the proofs, we first introduce the following notation.

\begin{align}
	v_{\theta,j}~\equiv ~ \frac{1}{\sqrt{n}}\sum_{i=1}^n\{m_j(W_i,\theta)-\Ep[m_j(W_i,\theta)]\}/\sigma_{\theta,j}.
\end{align}
For a (non-random) sequence $\Delta_n$, define $\mu_n$ as follows
\begin{align}
	\mu_n ~\equiv ~(1-\gamma_n)\mbox{-quantile of} \ \sup_{  \theta, \tilde \theta \in \Theta(\eeta), \| \theta - \tilde \theta \|\leq \Delta_n,j\in[p]} \left|  v_{\theta,j}- v_{\tilde\theta,j}\right|.
	\label{eq:mun}
\end{align}
Finally, we define the following objects.
\begin{align*}
	\ell_{\min}&\equiv \inf_{\theta\in\Theta(\eeta)}\max_{j\in[p]}\sqrt{n}\sigma_{\theta,j}^{-1}\Ep[m_{j}(W,\theta)]\\
	\Psi_\theta &\equiv \left\{ j \in [p] : \sqrt{n}\Ep[m_{j}(W,\theta)]/\sigma_{\theta,j}  \geq \max_{\tilde j \in [p]} \sqrt{n}\sigma_{\theta,\tilde j}^{-1}\Ep[m_{\tilde j}(W,\theta)] - \frac{2\bar w_n}{1-t^\sigma_n}\right\}\\
	\Theta_n^{\psi_{n}} &\equiv \{  \theta \in \Theta(\eeta) : \max_{j\in[p]} \sqrt{n}\sigma_{\theta,j}^{-1}\Ep[m_{j}(W,\theta)] \leq \psi_n\sqrt{n} \}\\
	\psi_n &~\equiv ~ 2(\bar w_n/\sqrt{n}) ({1+t_n^\sigma}/{1-t_n^\sigma}) .
\end{align*}

\begin{proof}[Proof of Theorem \ref{thm:MSB:inferenceSubvector}] We divide the argument into cases.


Case 1: $\ell_{\min} \leq -2\bar w_n/(1-t^\sigma_n).$ Then, we have that with probability exceeding $1-C\gamma_n-Cn^{-1}$,
\begin{align*}
	T_n(\eeta) & = \inf_{\theta \in \Theta(\eeta)}\max_{j\in[p]} \hat v_{\theta,j} +\sqrt{n}\Ep[m_{j}(W,\theta)]/\hat\sigma_{\theta,j}\\
	 & \leq_{(1)} \sup_{\theta \in \Theta(\eeta),j\in[p]} |\hat v_{\theta,j}| + \inf_{\theta \in \Theta(\eeta)}\max_{j\in[p]}\sqrt{n}\Ep[m_{j}(W,\theta)]/\hat\sigma_{\theta,j}\\
	 & \leq_{(2)} \bar w_n /(1-t^\sigma_n) - 2\bar w_n/(1-t^\sigma_n) \\
	 & \leq_{(3)} - \sup_{\theta \in \Theta(\eeta),j\in[p]} |\hat v_{\theta,j}|\\
	 & \leq_{(4)} - \sup_{\theta \in \Theta(\eeta),j\in[p]} |\hat v^*_{\theta,j}|+\delta_4\\
	 & \leq_{(5)} \inf_{\theta \in \widehat \Theta_n} \max_{j\in\widehat \Psi_\theta} \hat v^*_{\theta,j}+\delta_4,
\end{align*}
where (1) holds by $\inf_k \max_m a_{km} + b_{km} \leq \max_{k,m}|a_{km}| + \inf_k \max_m b_{km}$, (2) by $\sup_{\theta \in \Theta(\eeta),j\in[p]} |\hat v_{\theta,j}| \leq \bar w_n/(1-t^\sigma_n)$ holding with probability exceeding $1-2\gamma_n$ and the assumption on $\ell_{\min}$, (3) by the same event, (4) holds by Theorem \ref{thm:clt:minmax} and \ref{thm:clt:minmax:Gaussian} with $\delta_4=\bar \delta_1+\bar \delta_2$ with probability exceeding $1-C(\gamma_n+n^{-1})$ and (5) holds since $-\sup_{k \in \bar K,m \in \bar M}|a_{km}|\leq \inf_{k \in K}\max_{m\in M} a_{km}$ for any sets $K\subset \bar K$ and $M \subset \bar M$. The result then follows in this case.

~\\

Case 2: $\ell_{\min} > -2\bar w_n/(1-t^\sigma_n).$
Since $\Theta_I(\eeta) = \Theta_I\cap \Theta(\eeta) \neq \emptyset$, $\ell_{\min}\leq 0$.
Moreover, let {\small $$\check \Theta_n = \{\Theta_I(\eeta) \cap \widehat \Theta_n\} \cup \left\{ \theta \in \Theta_I(\eeta) : \exists \hat\theta \in \widehat \Theta_n\setminus \Theta_I, \begin{array}{c}
 \|\theta-\hat\theta\| \leq \psi_n/\vartheta_n, \\ \max_{j\in[p]}\Ep[m_j(W,\theta)]=0\end{array}\right\}.$$} Note that by Lemma \ref{lemma:selectionExtraTheta} we have $\widehat \Theta_n \subseteq \Theta_n^{\psi_n}$ with probability exceeding $1-2\gamma_n$. Therefore, with probability exceeding $1-2\gamma_n$, by Condition MB we have that for every $\hat\theta \in \widehat \Theta_n \subset \Theta_n^{\psi_n}$ we can find $\theta_n^* \in \Theta_I(\eeta)$ such that
  $$\vartheta_n \| \hat\theta - \theta_n^*\| \leq \max_{j\in[p]}\Ep[\sigma_{\theta,j}^{-1}m_j(W,\hat\theta)]\leq \psi_n.  $$



The main argument follows the following sequence of inequalities.
\begin{align*}
	& \P( T_n(\eeta) \geq t )  = \P( \inf_{\theta \in \Theta(\eeta)}\max_{j\in[p]} \hat v_{\theta,j} +\sqrt{n}\Ep[m_{j}(W,\theta)]/\hat\sigma_{\theta,j} \geq t)\\
	& \leq_{(1)} \P( \inf_{\theta \in \Theta_I(\eeta)}\max_{j\in[p]} v_{\theta,j} +\sqrt{n}\Ep[m_{j}(W,\theta)]/\hat \sigma_{\theta,j} \geq t - \delta_1)+2\gamma_n\\
	& \leq_{(2)} \P( \inf_{\theta \in \Theta_I(\eeta)}\max_{j\in \Psi_\theta} v_{\theta,j} +\sqrt{n}\Ep[m_{j}(W,\theta)]/\hat \sigma_{\theta,j} \geq t - \delta_2)+2\gamma_n\\
	& \leq_{(3)} \P( \inf_{\theta \in \Theta_I(\eeta)}\max_{j\in \Psi_\theta} v_{\theta,j} +\sqrt{n}\Ep[m_{j}(W,\theta)]/\sigma_{\theta,j} \geq t - \delta_3)+3\gamma_n\\
	& \leq_{(4)} \P( \inf_{\theta \in \Theta_I(\eeta)}\max_{j\in \Psi_\theta} v_{\theta,j}^* +\sqrt{n}\Ep[m_{j}(W,\theta)]/\sigma_{\theta,j} \geq t - \delta_4)+ r_4\\
	& \leq_{(5)} \P( \inf_{\theta \in \check\Theta_n}\max_{j\in \Psi_\theta} v_{\theta,j}^* +\sqrt{n}\Ep[m_{j}(W,\theta)]/\sigma_{\theta,j} \geq t - \delta_5)+ r_5\\
	& \leq_{(6)} \P( \inf_{\theta \in \check\Theta_n}\max_{j\in \Psi_\theta} v_{\theta,j}^* \geq t - \delta_6)+ r_6\\
	& \leq_{(7)} \P( \inf_{\theta \in \widehat\Theta_n}\max_{j\in \hat\Psi_\theta} v_{\theta,j}^* \geq t - \delta_7)+ r_7\\
	& \leq_{(8)} \P( \inf_{\theta \in \widehat\Theta_n}\max_{j\in \hat \Psi_\theta} \hat v_{\theta,j}^* \geq t - \delta_8)+ r_8,
\end{align*}
where (1) holds since $\Theta_I(\eeta) \subset \Theta(\eeta)$ and $\delta_1=\bar w_n t^\sigma_n$,  (2) holds with $\delta_2 = \delta_1$ as $\Psi_\theta$ contains all the maximizers for each $\theta \in \Theta_I(\eeta)$ with probability exceeding $1-2\gamma_n$ by Lemma \ref{lemma:selectionExtra0}, and (3) holds because for $\theta \in \Theta_I(\eeta)$ and $j\in \Psi_\theta$ we have $\ell_{\min} - 2\bar w_n/(1-t_n^\sigma) \leq \sqrt{n}\sigma_{\theta,j}^{-1}\Ep[m_j(W,\theta)] \leq 0$ so that we can take $$\begin{array}{rl}
\delta_3 & = \delta_2+ \left|\ell_{\min}- 2\bar w_n/(1-t^\sigma_n)\right|t^\sigma_n \leq \delta_2+ \frac{4\bar w_n t_n^\sigma }{1-t^\sigma_n}.\end{array}$$
By Theorem \ref{thm:clt:minmax} and \ref{thm:clt:minmax:Gaussian} (with $\eta = n^{-1/2}$) we have that (4) holds with $\delta_4 = \delta_3 + \delta_{n,\eta,\gamma_n} + \bar\delta_{n,\eta,\gamma_n}$  and $r_4 = 4\gamma_n + Cn^{-1}$. Relation (5) holds with $\delta_5=\delta_4$ and $r_5 = r_4$ since  $\check\Theta_n \subseteq \Theta_I(\eeta)$ by construction. Next note that by definition of $\check\Theta_n$, (6) holds with $\delta_6 = \delta_5$ and $r_6 = r_5$ since $\check\Theta_n \subseteq \Theta_I(\eeta)$ implies $\Ep[m_j(W,\theta)] \leq 0$ for all $\theta \in \check\Theta_n$. Relation (7) holds with $r_7 = r_6+2\gamma_n+\gamma_n$ and
$$\begin{array}{rl}
\delta_7 & = \delta_6 + \inf_{\theta \in \check\Theta_n}\max_{j\in \Psi_\theta} v_{\theta,j}^*-\inf_{\theta \in \widehat\Theta_n}\max_{j\in \hat\Psi_\theta} v_{\theta,j}^*\\
 & \leq \delta_6 + \sup_{\|\hat \theta-\theta\|\leq \psi_n/\vartheta_n, j\in [p]}|v_{\hat\theta,j}^*-v_{\theta,j}^*| \\
 & \leq \delta_6 + \mu_n, \end{array}
$$  since with probability exceeding $1-2\gamma_n$ we have $\widehat \Theta_n \subseteq \Theta_n^{\psi_n}$ and by definition of $\mu_n$. Indeed we have that for every $\theta \in \widehat \Theta_n$, there is $\check\theta \in \check \Theta_n\subset \Theta_I(\eeta)$ such that $\|\theta - \check\theta\| \leq  \frac{\psi_n}{\vartheta_n}$. Then by Lemma \ref{lemma:selectionExtra}, $\Psi_{\check\theta} \subseteq \hat\Psi_{\theta}$ with probability exceeding $1-2\gamma_n$ provided that
 $$ \begin{array}{rl}
 \frac{\psi_n}{\vartheta_n} & \leq \frac{M_n+\max_{j\in[p]}\sqrt{n}\sigma_{\check\theta,j}^{-1}\Ep[m_{j}(W,\check\theta)]- (1-t_n^\sigma)\max_{\tilde j\in[p]} \sqrt{n}  \sigma_{\theta,\tilde j}^{-1}\Ep[m_j(W,\theta)]-4\bar w_n a_n }{L_G\sqrt{n}(1+t_n^\sigma)},
 \end{array}$$
where $a_n\equiv [(1+t_n^\sigma)^2/(1-t_n^\sigma) + \frac{4}{3}t_n^\sigma]/(1-t_n^\sigma) \to 1$ as $t_n^\sigma =o(1)$. Under our conditions, $$\max_{j\in[p]}\sqrt{n}\sigma_{\check\theta,j}^{-1}\Ep[m_{j}(W,\check\theta)] \geq -\frac{2\bar w_n}{1-t_n^\sigma} \ \mbox{and} \  \max_{\tilde j\in[p]} \sqrt{n}  \sigma_{\theta,\tilde j}^{-1}\Ep[m_j(W,\theta)]\leq \sqrt{n}\psi_n,$$ where $\psi_n = 2\frac{(1+t_n^\sigma)}{(1-t_n^\sigma)}\bar w_n/n^{1/2}$, and therefore it suffices that
$$ M_n \geq \frac{L_G(1+t_n^\sigma)}{\vartheta_n}2\frac{(1+t_n^\sigma)}{(1-t_n^\sigma)}\bar w_n + 2\bar w_n / (1-t_n^\sigma) + 2\frac{(1+t_n^\sigma)}{(1-t_n^\sigma)}\bar w_n +4\bar w_n a_n.$$
 Finally, (8) holds with $\delta_8 = \delta_7 + \bar w_n t^\sigma_n $ and $r_8 = r_7+ \gamma_n$.

Under Condition MB, since $b$ and $\sigma>0$ are constants, $q>6$, and $K_n^3\leq n$, the following relations hold:\\
(i) $(b/\sigma)K_n^{1/2}/n^{1/2-2/q} \leq C$, \\
(ii) $K_n^{1/4}b^{1/2}\sigma^{3/2}\leq n^{1/4}$,\\
(iii) $(b/\sigma)^{1/6}K_n^{1/12}/\gamma_n^{2/3+1/q} \leq n^{1/12}$, and \\
(iv) $K_n^{1/3}b^{1/3}\sigma^{-2/3}/\gamma_n^{2/3+1/q}\leq n^{1/3-1/q}$.

 Therefore, we have $r_8 \leq C\{\gamma_n + n^{-1}\}$ and
$$\begin{array}{rl}
\delta_8  &  \leq  t^\sigma_n\{\sqrt{n}\psi_n+{1+t_n^\sigma}{1-t_n^\sigma}|\ell_{\min}|+2\bar w_n/(1-t^\sigma_n)\} + \bar \delta_{n,\eta,\gamma_n} + \bar \delta_{n,\eta,\gamma_n} + \mu_n\\
& \leq  \frac{Cb\sigma^2 K_n}{\gamma_n^{3/q}n^{1/2}} + \frac{C(b\sigma^2 K_n^2)^{1/3}}{\gamma_n^{1/3}n^{1/6}} +  C L_C^{1/2}\left(\frac{C\sigma K_n^{1/2}}{\gamma_n^{1/q}n^{1/2}\vartheta_n}\right)^{\chi/2} \frac{K_n^{1/2}}{\gamma_n^{1/q}} + \frac{CbK_n}{\gamma_n^{1/q}n^{1/2-1/q}},
\end{array}$$
where we used Lemma \ref{lemma:BoundsOn3} with $\Delta_n = \psi_n/\vartheta_n$, $(1+t^\sigma_n)/(1-t_n^\sigma) \leq 2$, the definitions of $\delta_{n,\eta,\gamma_n}$ and $\bar \delta_{n,\eta,\gamma_n}$ in Theorems \ref{thm:clt:minmax} and \ref{thm:clt:minmax:Gaussian} with $\eta = n^{-1/2}$ and $\eta_{\mathcal{F},\gamma_n}$ as in Lemma \ref{lemma:MimpliesAB}, and the definition of $ \psi_n$.

The second statement follows from the definition of the anti-concentration (\ref{eq:AC}), and the triangle inequality. (Note that the proof allows for $t$ to be random and data dependent, i.e. $t=t(W)$.)

\end{proof}

\begin{lemma}\label{lemma:selectionExtraTheta}
	Suppose $\Theta_I(\eeta)=\Theta_I\cap \Theta(\eeta)\neq \emptyset$. Then, with probability exceeding $1-2\gamma_n$, $\widehat \Theta_n \subseteq \Theta_n^{\psi_n}$.
\end{lemma}
\begin{proof}
Take an $\varepsilon$-minimizer of $T_n(\eeta)$, say $\hat\theta_n$. We can assume that $\hat\theta_n \not\in \Theta_I$ (otherwise we are done). Therefore we have with probability exceeding $1-2\gamma_n$ that
\begin{eqnarray*}
\max_{j\in[p]}\sqrt{n%
}\bar{m}_{j}( \hat \theta_n ) /\hat{\sigma}_{\hat\theta_n,j} &\leq  & \varepsilon + \inf_{\theta \in \Theta(\eeta)  }\max_{j\in[p]}\sqrt{n
}\bar{m}_{j} ( \theta ) /\hat{\sigma}_{\theta,j}  \\
& \leq & \varepsilon + \inf_{\theta \in \Theta_I(\eeta)   }\max_{j\in[p]} v_{\theta,j} {\sigma}_{\theta,j}/\hat{\sigma}_{\theta,j}  \\
& \leq & \varepsilon + (1+t_n^\sigma)\sup_{\theta \in \Theta_I(\eeta)   }\max_{j\in[p]} |v_{\theta,j}|  \\ & \leq & \varepsilon + (1+t_n^\sigma)\bar w_n .
\end{eqnarray*}%
Moreover we have with the same probability that
\begin{eqnarray*} \max_{j\in[p]}\frac{\sqrt{n%
}\bar{m}_{j}( \hat \theta_n )}{\hat{\sigma}_{\hat\theta_n,j}}  &\geq & \max_{j\in[p]} \frac{\sqrt{n}\Ep[m_j(W,\hat\theta_n)]}{\hat{\sigma}_{\hat\theta_n,j}} - \sup_{\theta\in \Theta_h, j\in[p]} |v_{\theta,j}| \frac{{\sigma}_{\theta,j}}{\hat{\sigma}_{\theta,j}} \\
 &\geq & (1-t_n^\sigma)\max_{j\in[p]} \frac{\sqrt{n}\Ep[m_j(W,\hat\theta_n)]}{{\sigma}_{\hat\theta_n,j}} - (1+t_n^\sigma)\sup_{\substack{\theta\in \Theta_h\\ j\in[p]}} |v_{\theta,j}|\\
 &\geq & (1-t_n^\sigma)\max_{j\in[p]} \frac{\sqrt{n}\Ep[m_j(W,\hat\theta_n)]}{{\sigma}_{\hat\theta_n,j}} - (1+t_n^\sigma) \bar w_n .\\
 \end{eqnarray*}%
Combining these relations we have that for any $\varepsilon$-minimizer we have
$$ \max_{j\in[p]} \frac{\sqrt{n}\Ep[m_j(W,\hat\theta_n)]}{{\sigma}_{\hat\theta_n,j}} \leq \frac{2(1+t^\sigma_n)\bar w_n+\varepsilon}{1-t_n^\sigma} $$
The result follows by taking $\varepsilon = 0$.
\end{proof}

The following results use the following sets of indices in $[p]$ parametrized by $\theta\in \Theta(\eeta)$.
{\small\begin{align*}
	\hat \Psi_\theta & \equiv \{ j\in [p] : \sqrt{n} \hat \sigma_{\theta,j}^{-1}\bar m_{\theta,j} \geq \max_{\tilde j\in[p]} \sqrt{n} \hat \sigma_{\theta,\tilde j}^{-1}\bar m_{\theta,\tilde j} - M_n \},\\
	\Psi_\theta & \equiv \left\{ j \in [p] : \sqrt{n} \sigma_{\theta,j}^{-1}\Ep[m_{j}(W,\theta)]\geq {\displaystyle \max_{\tilde j\in[p]}}\sqrt{n} \sigma_{\theta,\tilde j}^{-1}\Ep[m_{\tilde j}(W,\theta)]   -\frac{2\bar w_n}{1-t_n^\sigma}\frac{1+t_n^\sigma}{1-t_n^\sigma}\right\}.
\end{align*}}
\begin{lemma}\label{lemma:selectionExtra0}
Suppose that $\ell_{\min} \geq -2\bar w_n/(1-t^\sigma_n)$. Then, with probability exceeding $1-2\gamma_n$, $\theta \in \Theta_I(\eeta)$ implies that
$$\max_{j\in[p]} v_{\theta,j} +\sqrt{n}\Ep[m_{j}(W,\theta)]/\hat \sigma_{\theta,j} = \max_{j\in \Psi_\theta} v_{\theta,j} +\sqrt{n}\Ep[m_{j}(W,\theta)]/\hat \sigma_{\theta,j}.$$
\end{lemma}
\begin{proof}
Suppose the contrary. Then there exists $\theta \in \Theta_I(\eeta)\equiv \Theta_I\cap \Theta(\eeta)$ and $j^*\in [p]\setminus \Psi_{\theta}$ such that
$$  v_{\theta,j^*} +\sqrt{n}\Ep[m_{j^*}(W,\theta)]/\hat \sigma_{\theta,j^*} > \max_{j\in \Psi_\theta} v_{\theta,j} +\sqrt{n}\Ep[m_{j}(W,\theta)]/\hat \sigma_{\theta,j}.$$
Note that $\theta \in \Theta_I(\eeta)$ implies $\Ep[m_{j}(W,\theta)] \leq 0$.
Since $|v_{\theta,j}|\leq \bar w_n$ and $|\hat\sigma_{\theta,j}/\sigma_{\theta,j}-1|\leq t^\sigma_n$ with probability exceeding $1-2\gamma_n$ then, with the same probability,
$$\begin{array}{rl}
 (1-t_n^\sigma) \sqrt{n}\Ep[m_{j^*}(W,\theta)]/ \sigma_{\theta,j^*} & \geq \sqrt{n}\Ep[m_{j^*}(W,\theta)]/\hat \sigma_{\theta,j^*} \\
& > \max_{j\in[p]}\sqrt{n}\Ep[m_{j}(W,\theta)]/\hat \sigma_{\theta,j} - 2\bar w_n\\
& \geq (1+t_n^\sigma){\displaystyle \max_{j\in[p]}}\sqrt{n}\Ep[m_{j}(W,\theta)]/ \sigma_{\theta,j} - 2\bar w_n.\\
\end{array}$$
Therefore, since $\ell_{\min} \geq -2\bar w_n/(1-t^\sigma_n)$ and $\Ep[m_{j}(W,\theta)] \leq 0$,
$$\begin{array}{rl}
\sqrt{n} \Ep[m_{j^*}(W,\theta)]/\sigma_{\theta,j^*}& \geq \frac{1+t_n^\sigma}{1-t_n^\sigma}\max_{j\in[p]}\sqrt{n}\Ep[m_{j}(W,\theta)]/ \sigma_{\theta,j} - \frac{2\bar w_n}{1-t_n^\sigma} \\
& \geq \max_{j\in[p]}\sqrt{n}\Ep[m_{j}(W,\theta)]/ \sigma_{\theta,j} - \frac{2\bar w_n}{1-t_n^\sigma} \\
&  + \left(\frac{1+t_n^\sigma}{1-t_n^\sigma}-1\right) {\displaystyle \inf_{\theta\in\Theta(\eeta)}\max_{j\in[p]}}\sqrt{n}\Ep[m_{j}(W,\theta)]/\sigma_{\theta,j}\\
& \geq \max_{j\in[p]}\sqrt{n}\Ep[m_{j}(W,\theta)]/ \sigma_{\theta,j} - \frac{2\bar w_n}{1-t_n^\sigma}\frac{1+t_n^\sigma}{1-t_n^\sigma}.
\end{array}$$
This implies that $j^* \in  \Psi_\theta$ which yields a contradiction.
\end{proof}

\begin{lemma}\label{lemma:selectionExtra}
Suppose that $\ell_{\min} \geq -2\bar w_n/(1-t^\sigma_n)$.
Then, with probability exceeding $1-2\gamma_n$, the fact that $\tilde \theta, \theta \in \Theta(\eeta)$ with
\begin{align*}
	&\|\tilde \theta - \theta\| \{L_G\sqrt{n}(1+t_n^\sigma)\}  \leq\\
	&
	\left\{\begin{array}{l}
		M_n + \max_{j\in[p]}\sqrt{n}\sigma_{\tilde\theta,j}^{-1}\Ep[m_{j}(W,\tilde\theta)] \\ - (1-t_n^\sigma)\max_{\tilde j\in[p]} \sqrt{n} \sigma_{\theta,\tilde j}^{-1}\Ep[m_{j}(W,\theta)]
		 - \frac{4\bar w_n}{1-t_n^\sigma} \left\{\frac{(1+t_n^\sigma)^2}{1-t_n^\sigma} + \frac{4}{3}t_n^\sigma\right\}
	\end{array}\right\}
\end{align*}
 %
  implies that $\Psi_{\tilde \theta} \subseteq \hat \Psi_{\theta}$.
\end{lemma}
\begin{proof}
Take $j \in \Psi_{\tilde \theta}$ so that $$\sqrt{n} \sigma_{\tilde \theta,j}^{-1}\Ep[m_{j}(W,\tilde \theta)]\geq \max_{\tilde j\in[p]}\sqrt{n} \sigma_{\tilde\theta,\tilde j}^{-1}\Ep[m_{\tilde j}(W,\tilde \theta)] -\frac{2\bar w_n}{1-t^\sigma_n}\frac{1+t_n^\sigma}{1-t_n^\sigma}.$$

Then with probability at least  $1-2\gamma_n$ we have $$|v_{\theta,j}|\leq \bar w_n \ \ \mbox{and} \ \ |\hat \sigma_{\theta,j}/\sigma_{\theta,j}-1|\leq t_n^\sigma.$$ Thus for any $\theta \in \Theta(\eeta)$ we have
$$
\begin{array}{rl}
\sqrt{n}\bar m_{\theta,j}/\hat\sigma_{\theta,j} &  = \{ v_{\theta,j} + \sqrt{n}\sigma_{\theta,j}^{-1}\Ep[m_j(W,\theta)]\} \sigma_{\theta,j}/\hat\sigma_{\theta,j}\\
& \geq \{ v_{\theta,j} -L_G\sqrt{n}\|\theta-\tilde\theta\| + \sqrt{n}\sigma_{\tilde \theta,j}^{-1}\Ep[m_j(W,\tilde\theta)]\} \sigma_{\theta,j}/\hat\sigma_{\theta,j}\\
& \geq \{ v_{\theta,j} -L_G\sqrt{n}\|\theta-\tilde\theta\| + \max_{\tilde j\in[p]}\sqrt{n} \sigma_{\tilde\theta,\tilde j}^{-1}\Ep[m_{\tilde j}(W,\tilde \theta)] \\
 & -\frac{2\bar w_n}{1-t^\sigma_n}\frac{1+t_n^\sigma}{1-t_n^\sigma}\} \sigma_{\theta,j}/\hat\sigma_{\theta,j}\\
& \geq -\{L_G\sqrt{n}\|\theta-\tilde\theta\|  + \frac{3\bar w_n}{1-t^\sigma_n}\frac{1+t_n^\sigma}{1-t_n^\sigma}\} (1+t_n^\sigma) \\
 & +  \max_{\tilde j\in[p]}\sqrt{n} \sigma_{\tilde\theta,\tilde j}^{-1}\Ep[m_{\tilde j}(W,\tilde \theta)] \sigma_{\theta,j}/\hat\sigma_{\theta,j}.
\end{array}
$$
Next note that
$$
\begin{array}{rl}
&\max_{\tilde j\in[p]}\sqrt{n} \sigma_{\tilde\theta,\tilde j}^{-1}\Ep[m_{\tilde j}(W,\tilde \theta)] \sigma_{\theta,j}/\hat\sigma_{\theta,j} \\
& =\max_{\tilde j\in[p]} \left\{ \sqrt{n} \sigma_{\tilde\theta,\tilde j}^{-1}\Ep[m_{\tilde j}(W,\tilde \theta)] + \frac{2\bar w_n}{1-t_n^\sigma}- \frac{2\bar w_n}{1-t_n^\sigma} \right\}\sigma_{\theta,j}/\hat\sigma_{\theta,j}\\
& \geq   \max_{\tilde j\in[p]}  \left\{ \sqrt{n} \sigma_{\tilde\theta,\tilde j}^{-1}\Ep[m_{\tilde j}(W,\tilde \theta)] + \frac{2\bar w_n}{1-t_n^\sigma} \right\}\sigma_{\theta,j}/\hat\sigma_{\theta,j} - 2\bar w_n\frac{1+t_n^\sigma}{1-t_n^\sigma}\\
& \geq \left\{ \max_{\tilde j\in[p]}\sqrt{n} \sigma_{\tilde\theta,\tilde j}^{-1}\Ep[m_{\tilde j}(W,\tilde \theta)] + \frac{2\bar w_n}{1-t_n^\sigma} \right\}(1-t_n^\sigma) - 2\bar w_n\frac{1+t_n^\sigma}{1-t_n^\sigma}\\
& =\max_{\tilde j\in[p]}\sqrt{n} \sigma_{\tilde\theta,\tilde j}^{-1}\Ep[m_{\tilde j}(W,\tilde \theta)]  - 2\bar w_n\frac{2t_n^\sigma}{1-t_n^\sigma}.
\end{array}
$$

Therefore, $j \in \widehat \Psi_\theta$ occurs provided that
$$
\begin{array}{rl} & {\displaystyle \max_{\tilde j\in[p]}}\sqrt{n} \sigma_{\tilde\theta,\tilde j}^{-1}\Ep[m_{\tilde j}(W,\tilde \theta)] -L_G\sqrt{n}\|\theta-\tilde\theta\|(1+t_n^\sigma)  - 3\bar w_n\frac{(1+t_n^\sigma)^2}{(1-t_n^\sigma)^2}   - \frac{4\bar w_n t_n^\sigma}{1-t_n^\sigma}\\ &  \geq {\displaystyle \max_{\tilde j\in[p]}} \sqrt{n} \hat \sigma_{\theta,\tilde j}^{-1}\bar m_{\theta,\tilde j} - M_n.\end{array}$$

Therefore it suffices that
 $$\begin{array}{rl} \|\tilde \theta - \theta\|\{L_G\sqrt{n}(1+t_n^\sigma)\} & \leq  M_n +\max_{j\in[p]}\sqrt{n}\sigma_{\tilde\theta,j}^{-1}\Ep[m_{j}(W,\tilde\theta)]\\ &- \max_{\tilde j\in[p]} \sqrt{n} \hat \sigma_{\theta,\tilde j}^{-1}\bar m_{\theta,\tilde j}
   -\frac{3\bar w_n}{1-t_n^\sigma} \left\{\frac{(1+t_n^\sigma)^2}{1-t_n^\sigma} + \frac{4}{3}t_n^\sigma\right\} .\end{array}$$
To obtain the statement of the lemma, note that
$$
\begin{array}{rl}
& \max_{\tilde j\in[p]} \sqrt{n} \hat \sigma_{\theta,\tilde j}^{-1}\bar m_{\theta,\tilde j} = \max_{\tilde j\in[p]} \{ ( v_{\theta,\tilde j} + \sqrt{n}\sigma_{\theta,\tilde j}^{-1}\Ep[m_{\tilde j}(W,\theta)] ) \sigma_{\theta,\tilde j}/\hat \sigma_{\theta,\tilde j} \}\\
& \geq - \bar w_n (1+t_n^\sigma) + \max_{\tilde j\in[p]} \{  \sqrt{n}\sigma_{\theta,\tilde j}^{-1}\Ep[m_{\tilde j}(W,\theta)]  \sigma_{\theta,\tilde j}/\hat \sigma_{\theta,\tilde j} \}\\
& = - \bar w_n (1+t_n^\sigma) + \max_{\tilde j\in[p]} \{  ( \sqrt{n}\sigma_{\theta,\tilde j}^{-1}\Ep[m_{\tilde j}(W,\theta)] + \frac{2\bar w_n}{1-t_n^\sigma} - \frac{2\bar w_n}{1-t_n^\sigma}) \frac{\sigma_{\theta,\tilde j}}{\hat \sigma_{\theta,\tilde j}} \} \\
& \geq - \bar w_n (1+t_n^\sigma) -2\bar w_n\frac{1+t_n^\sigma}{1-t_n^\sigma}+ \max_{\tilde j\in[p]} \{  ( \sqrt{n}\sigma_{\theta,\tilde j}^{-1}\Ep[m_{\tilde j}(W,\theta)] + \frac{2\bar w_n}{1-t_n^\sigma}) \frac{\sigma_{\theta,\tilde j}}{\hat \sigma_{\theta,\tilde j}} \} \\
& \geq - \bar w_n (1+t_n^\sigma) -2\bar w_n\frac{2t_n^\sigma}{1-t_n^\sigma}+ (1-t_n^\sigma)\max_{\tilde j\in[p]} \sqrt{n}\sigma_{\theta,\tilde j}^{-1}\Ep[m_{\tilde j}(W,\theta)],
\end{array}
$$
since $\max_{\tilde j\in[p]} \{  ( \sqrt{n}\sigma_{\theta,\tilde j}^{-1}\Ep[m_{\tilde j}(W,\theta)] + \frac{2\bar w_n}{1-t_n^\sigma}) \sigma_{\theta,\tilde j}/\hat \sigma_{\theta,\tilde j} \} \geq 0$ by our assumption.

\end{proof}

\begin{proof}[Proof of Theorem \ref{thm:PB:inferenceSubvector}]
Let $\varepsilon = 1/\sqrt{n}$ and $\gamma_n \to 0$. Let $\bar w_n$, $t_n^{\sigma}$, and $\mu_n$ be defined as in (\ref{eq:wn}), (\ref{eq:tn}), and (\ref{eq:mun}) with $\Delta_n = \frac{\varepsilon + 2\bar w_n}{\kappa_n^{-1}\sqrt{n}\vartheta_n}$, respectively. Lemma \ref{lemma:BoundsOn3} provides upper bounds for each of these quantities. Under our conditions, $t_n^\sigma = o(1)$, $\mu_n=o(1)$ and $\bar w_n\geq 1$ can grow with $n$.
Moreover, by Lemma \ref{lemma:MimpliesAB}, Condition MB implies Conditions A and B (with suitable parameters), and so Theorems \ref{thm:clt:minmax} and \ref{thm:clt:minmax:Gaussian} apply.

We divide the rest of the argument into cases.

Case 1: $\ell_{\min} < - 8\bar w_n$. Then, with probability exceeding $1-2\gamma_n$ we have that $\sup_{\theta\in\Theta(\eeta),j\in[p]}|v_{\theta,j}|\vee |\hat v^*_{\theta,j}|\leq \bar w_n$ and the conditions of Lemma \ref{lemma:Ineq00} are satisfied (given that $(1+t_n^\sigma)/(1-t_n^\sigma) < 2$ and $\kappa_n \geq 2$). Therefore we have that
$$ \P(T_n(\eeta) \geq t ) \leq \P( R_n^{PR*} \geq t ) + 2\gamma_n,$$
and the result follows.

Case 2: $\ell_{\min} \geq - 8\bar w_n$. Define the following auxiliary statistics:
$$
\begin{array}{rl}
S_n^\kappa & \equiv \inf_{\theta \in \Theta (\eeta) }\max_{j\in [p]}\left\{
v_{\theta,j} +\kappa_n^{-1}\sqrt{n}\Ep\left[ m_{j}\left( W,\theta \right) %
\right] /\sigma_{\theta,j}\right\} \\
M_n & \equiv \inf_{\theta \in \Theta (\eeta) }\max_{j\leq
p} M(\theta,j)\\
R_{n}^{\ast \ast \ast } & \equiv \inf_{\theta \in \Theta (\eeta)
}\max_{j\in [p]}\left\{ \hat v_{\theta,j}^{\ast }\hat\sigma_{\theta,j}/\sigma_{\theta,j} +\kappa _{n}^{-1}%
\sqrt{n}\Ep\left[ m_{j}\left( W,\theta \right) \right] /\sigma _{\theta,j}\right\} \\
R_{n}^{\ast \ast } & \equiv \inf_{\theta \in \Theta (\eeta)
}\max_{j\in [p]}\left\{ \hat v_{\theta,j}^{\ast } +\kappa _{n}^{-1}%
\sqrt{n}\Ep\left[ m_{j}\left( W,\theta \right) \right] /\sigma _{\theta,j}\right\},
\end{array}
$$ where $M:\Theta(\eeta)\times[p] \to \mathbb{R}$ denotes a Gaussian process with $\Ep[M(\theta,j)]=\kappa^{-1}_n\sqrt{n}\Ep[m_j(\theta)]/\sigma_{\theta,j}$ and covariance operator given by ${\rm Cov}_M((\theta,j),(\theta',j'))=\Ep[\sigma_{\theta,j}^{-1}(m_j(W,\theta)-\Ep[m_j(W,\theta)])\sigma_{\theta',j'}^{-1}(m_{j'}(W,\theta')-\Ep[m_{j'}(W,\theta')])]$.

We have that
$$
\begin{array}{rl}
\P( T_n(\eeta) \geq t ) & \leq_{(1)} \P( S_n^\kappa \geq t - \delta_1' ) + r_1' \\
 & \leq_{(2)} \P( M_n \geq t - \delta_2' ) + r_2'\\
  & \leq_{(3)} \P( R_n^{\ast\ast\ast} \geq t - \delta_3' ) + r_3'\\
    & \leq_{(4)} \P( R_n^{\ast\ast} \geq t - \delta_4' ) + r_4'\\
& \leq_{(5)} \P( R_n^{PR\ast} \geq t - \delta_5' ) + r_5',
\end{array}
$$
where (1) holds by Lemma \ref{thm:Ineq01} with $\delta_1' = \mu_n + Ct_n^\sigma \bar w_n+L_G\kappa_n (\bar w_n/\vartheta_n)^2/\sqrt{n}$ and $r_1' = 3\gamma_n$.\footnote{Indeed, because $(1+t^\sigma_n)/(1-t^\sigma_n)\leq 2$ and $\varepsilon \leq t_n^\sigma \leq 1  \leq \bar w_n$, with probability $1-3\gamma_n$ we have $\mathcal{W}_n\leq \bar w_n$, $\mathcal{T}_n\leq t_n^\sigma$, and $\mathcal{M}_n \leq \mu_n$ and Lemma \ref{thm:Ineq01} implies that $T_n(\eeta) \leq S_{n}^\kappa + \varepsilon + t_n^\sigma(\varepsilon+3\bar w_n)+L_G\frac{\kappa_n}{\sqrt{n}}(1+t_n^\sigma)(\frac{\varepsilon+2\bar w_n}{\vartheta_n})^2+(1+t_n^\sigma)\mu_n
 \leq \mu_n + Ct_n^\sigma \bar w_n+CL_G\kappa_n (\bar w_n/\vartheta_n)^2/\sqrt{n}$.} Relation (2) holds by Theorem \ref{thm:clt:minmax} with $\delta_2' = \delta_1' + \delta_{n,\eta,\gamma_n}$ and $r_2' = r_1' + C\{ \gamma_n + n^{-1} \}$. Relation (3) follows by Theorem \ref{thm:clt:minmax:Gaussian} with $\delta_3' = \delta_2' + \bar\delta_{n,\eta,\gamma_n}$ and $r_3'= r_2' + C\{ \gamma_n + n^{-1} \}$. Relation (4) follows by $|R_n^{\ast\ast\ast}-R_n^{\ast\ast}|\leq 2\bar w_n t_n^\sigma$ with probability exceeding $1-2\gamma_n$ (so we can take $\delta_4' = \delta_3' + 2\bar w_n t_n^\sigma$ and $r_4' = r_3' + 2\gamma_n$).

Finally, to show relation (5) consider the values $\theta\in \Theta(\eeta)$ that are candidates to be $\varepsilon$ away from the the infimum. In particular, we have for any such $\varepsilon$-minimizer $\theta$ associated with $R_n^{**}$ it holds that
$$ \max_{j\in [p]}\kappa_n^{-1}\sqrt{n}\Ep[m_j(W,\theta)]/\sigma_{\theta,j} \leq \varepsilon + 2\sup_{\theta\in\Theta(\eeta),j\in[p]}|\hat v_{\theta,j}^*| $$
and for  $\varepsilon$-minimizer $\theta$ associated with $R_n^{PR*}$ it holds that
$$ \max_{j\in [p]}\kappa_n^{-1}\sqrt{n}\bar m_{\theta,j}/\hat\sigma_{\theta,j} \leq \varepsilon + 2\sup_{\theta\in\Theta(\eeta),j\in[p]}|\hat v_{\theta,j}^*|. $$
which implies that any $\varepsilon$-minimizer $\theta$ associated with $R_n^{PR*}$ satisfies
$$ \kappa_n^{-1}\sqrt{n}\Ep[m_j(W,\theta)]/\hat\sigma_{\theta,j} \leq \varepsilon + 2\sup_{\theta\in\Theta(\eeta),j\in[p]}|\hat v_{\theta,j}^*| + \kappa_n^{-1}\frac{\sqrt{n}(\Ep[m_j(W,\theta)]-\bar m_{\theta,j})}{\hat\sigma_{\theta,j}}. $$
Further, with probability exceeding $1-2\gamma_n$ we have that any $\varepsilon$-minimizer $\theta$ associated with $R_n^{PR*}$ satisfies
{\small $$\begin{array}{rl}
\displaystyle  \max_{j\in [p]}\kappa_n^{-1}\sqrt{n}\Ep[m_j(W,\theta)]/\sigma_{\theta,j} & \displaystyle \leq \max_{j\in [p]}\left\{ \frac{\hat\sigma_{\theta,j}}{\sigma_{\theta,j}}\left(\varepsilon + 2\sup_{\theta\in\Theta(\eeta),j\in[p]}|\hat v_{\theta,j}^*|\right) + \left|\frac{v_{\theta,j}}{\kappa_n}\right|\right\} \\
 & \leq \frac{1+t_n^\sigma}{1-t_n^\sigma}(\varepsilon+3\bar w_n).
 \end{array} $$}

Defining the following set of pairs $(\theta,j)$ which are candidates to define $R_n^{PR*}$ and $R^{**}_n${\small $$\bar \Psi \equiv \left\{ (\theta,j) \in \Theta(\eeta)\times[p] : \begin{array}{c}\max_{\tilde j\in [p]}\kappa_n^{-1}\sqrt{n}\Ep[m_{\tilde j}(W,\theta)]/\sigma_{\theta,\tilde j} \leq \frac{1+t_n^\sigma}{1-t_n^\sigma}(\varepsilon + 3 \bar w_n )\\
\kappa_n^{-1}\sqrt{n}\frac{\Ep[m_{j}(W,\theta)]}{\sigma_{\theta,j}} \geq \kappa_n^{-1}\sqrt{n} {\displaystyle \max_{\tilde j\in [p]}}\frac{\Ep[m_{\tilde j}(W,\theta)]}{\sigma_{\theta,\tilde j}} - 12\bar w_n\end{array}\right\}.$$}
With probability exceeding $1-2\gamma_n$, the $\varepsilon$-minimizers and the corresponding ``binding'' inequalities that define $R_n^{PR*}$ and $R^{**}_n$ are in $\bar \Psi$. Note that if
$\bar j$ is such that $\kappa_n^{-1}\sqrt{n}\frac{\Ep[m_{\bar j}(W,\theta)]}{\sigma_{\theta,\bar j}} < \kappa_n^{-1}\sqrt{n} {\displaystyle \max_{\tilde j\in [p]}}\frac{\Ep[m_{\tilde j}(W,\theta)]}{\sigma_{\theta,\tilde j}}- 12\bar w_n$ we have that
$$\begin{array}{rl} v_{\bar j}^*(\theta)+\kappa_n^{-1}\sqrt{n}\frac{\Ep[m_{\bar j}(W,\theta)]}{\sigma_{\theta,\bar j}} & \leq \bar w_n + \kappa_n^{-1}\sqrt{n}\frac{\Ep[m_{\bar j}(W,\theta)]}{\sigma_{\theta,\bar j}} \\
& \leq -11\bar w_n + \max_{\tilde j\in[p]}\frac{\Ep[m_{\tilde j}(W,\theta)]}{\sigma_{\theta,\tilde j}}\\
& \leq -10\bar w_n + \max_{\tilde j\in[p]}v_{\tilde j}^*(\theta)+\frac{\Ep[m_{\tilde j}(W,\theta)]}{\sigma_{\theta,\tilde j}},
\end{array}
$$ so $(\theta,j)$ cannot be an $\varepsilon$-minimizer of $R_n^{**}$ for $\varepsilon < \bar w_n$. Similarly, for $R^*_n$,
$$\begin{array}{rl} & v_{\bar j}^*(\theta)+\kappa_n^{-1}\sqrt{n}\frac{\bar m_{\theta,j}}{\hat\sigma_{\theta,\bar j}}  \leq \bar w_n + \kappa_n^{-1}\sqrt{n}\frac{\bar m_{\theta,j}}{\hat\sigma_{\theta,\bar j}} \\
& \leq \bar w_n + \frac{1+t^\sigma_n}{1-t_n^\sigma}\kappa_n^{-1}\sqrt{n}\frac{|\bar m_{\theta,\bar j}-\Ep[m_{\bar j}(W,\theta)]|}{\sigma_{\theta,\bar j}}+\frac{1+t^\sigma_n}{1-t_n^\sigma}\kappa_n^{-1}\sqrt{n}\frac{\Ep[m_{\bar j}(W,\theta)]}{\sigma_{\theta,\bar j}} \\
& \leq -(11-\kappa_n^{-1})\frac{1+t^\sigma_n}{1-t_n^\sigma}\bar w_n +\frac{1+t^\sigma_n}{1-t_n^\sigma}\max_{\tilde j\in[p]}\frac{\kappa_n^{-1}\sqrt{n}\Ep[m_{\tilde j}(W,\theta)]}{\sigma_{\theta,\tilde j}}\\
& \leq -(10-\kappa_n^{-1})\frac{1+t^\sigma_n}{1-t_n^\sigma}\bar w_n +\frac{1+t^\sigma_n}{1-t_n^\sigma}\max_{\tilde j\in[p]}v^*_{j}(\theta)+\frac{\kappa_n^{-1}\sqrt{n}\Ep[m_{\tilde j}(W,\theta)]}{\sigma_{\theta,\tilde j}}\\
& \leq -(10-\kappa_n^{-1})\frac{1+t^\sigma_n}{1-t_n^\sigma}\bar w_n +\frac{2t^\sigma_n}{1-t_n^\sigma}10\bar w_n + \max_{\tilde j\in[p]}v^*_{j}(\theta)+\frac{\kappa_n^{-1}\sqrt{n}\Ep[m_{\tilde j}(W,\theta)]}{\sigma_{\theta,\tilde j}}.
\end{array}
$$

Then, with the same probability
$$\begin{array}{rl}
 |R_n^{PR*} - R^{**}_n| & \leq \kappa_n^{-1}\sup_{(\theta,j)\in \bar\Psi} \sqrt{n}\left|\frac{\Ep[m_j(W,\theta)]}{\sigma_{\theta,j}} -\frac{\bar m_{\theta,j}}{\hat\sigma_{\theta,j}}\right|\\
 & \leq \kappa_n^{-1}\sqrt{n}\sup_{(\theta,j)\in \bar\Psi} \left|\frac{\Ep[m_j(W,\theta)]-\bar m_{\theta,j}}{\sigma_{\theta,j}}\right|+\left| \frac{\bar m_{\theta,j}}{\hat\sigma_{\theta,j}}-\frac{\bar m_{\theta,j}}{\sigma_{\theta,j}}\right|\\
 & \leq \kappa_n^{-1}\sup_{(\theta,j)\in \bar\Psi}|v_{\theta,j}|\left(1+|\hat\sigma_{\theta,j}^{-1}-\sigma_{\theta,j}^{-1}|\right)\\
 &  +  \kappa_n^{-1}\sup_{(\theta,j)\in \bar\Psi} \left|\frac{\sqrt{n}\Ep[m_j(W,\theta)]}{\sigma_{\theta,j}}\right|\left|\frac{\sigma_{\theta,j}}{\hat \sigma_{\theta,j}}-1\right|.
\end{array}$$
Since $\ell_{min}\geq -8\bar w_n$, for any $(\theta,j) \in \bar\Psi$ it follows that
$$ -8\kappa_n^{-1}\bar w_n -12\bar w_n \leq \kappa_n^{-1}\frac{\sqrt{n}\Ep[m_j(W,\theta)]}{\sigma_{\theta,j}} \leq 2(\varepsilon + 3\bar w_n).$$ In turn we have we have
$$ |R_n^{PR*} - R^{**}_n|  \leq C\kappa_n^{-1} \bar w_n + Ct_n^\sigma \bar w_n.$$
and (5) holds with $r_5' = r_4' + 2\gamma_n$ and $\delta_5' = \delta_4' + C\kappa_n^{-1} \bar w_n + C\bar w_n t_n^\sigma$.

Using the definitions of $r_\ell$'s and $\delta_\ell$'s, $\ell=1,2,3,4,5$, we have established that
$$ \P(T_n(\eeta) \geq t ) \leq P\left( R_n^{PR*} \geq t -
\delta_5'\right) + C\{ \gamma_n + n^{-1} \}, $$ where $\delta_5' = C\{t_n^\sigma \bar w_n+L_G\frac{\kappa_n \bar w_n^2}{\sqrt{n}\vartheta_n^2} + \delta_{n,\eta,\gamma_n} + \bar\delta_{n,\eta,\gamma_n}+ \frac{\bar w_n}{\kappa_n}+  \mu_n\}$. Lemma \ref{lemma:BoundsOn3} below provides bounds on $t_n^\sigma$, $\bar w_n$ and $\mu_n$.

Under our conditions $(b/\sigma)K_n^{1/2}/n^{1/2-2/q} \leq C$, $K_n^{1/4}b^{1/2}\sigma^{3/2}\leq n^{1/4}$, $b^{1/6}\sigma^{-1/6}K_n^{1/12}/\gamma_n^{2/3+1/q} \leq n^{1/12}$, and  $K_n^{1/3}b^{1/3}\sigma^{-2/3}/\gamma_n^{2/3+1/q}\leq n^{1/3-1/q}$. Therefore, we have
$$\begin{array}{rl}
\delta_5'  & = CL_G\frac{\kappa_n \sigma^2K_n}{\gamma_n^{2/q}n^{1/2}\vartheta_n^2} + C\frac{(b\sigma^2 K_n^2)^{1/3}}{\gamma_n^{1/3}n^{1/6}} +C\frac{b  K_n}{\gamma_n^{1/q}n^{1/2-1/q}}  \\
&+ \frac{\bar w_n}{\kappa_n} + C L_C^{1/2}\left(\frac{\kappa_n\sigma K_n^{1/2}}{n^{1/2}\vartheta_n\gamma_n^{1/q}}\right)^{\chi/2} \frac{K_n^{1/2}}{\gamma_n^{1/q}},
\end{array}$$
and the result follows.

The second statement follows from the definition of the anti-concentration (\ref{eq:AC}), and the triangle inequality. (Note that the proof allows for $t$ to be random and data dependent, i.e. $t=t(W)$.)


%

%
\end{proof}

\begin{lemma}\label{lemma:BoundsOn3}
Assume Condition A. For $\widetilde K_n =  v\{\log n \vee \log(Ab/\sigma)\}$,
\begin{align*}
	t_n^\sigma &\leq  \frac{Cb\sigma \widetilde K_n^{1/2}}{\gamma_n^{2/q}n^{1/2}}+ \frac{Cb^2\widetilde K_n}{\gamma_n^{2/q}n^{1-2/q}}\\
	\bar w_n &\leq \frac{C\sigma \widetilde K_n^{1/2}}{\gamma_n^{1/q}}+\frac{Cb\widetilde K_n}{\gamma_n^{1/q}n^{1/2-1/q}}\\
	\mu_n &\leq C L_C^{1/2}\Delta_n^{\chi/2} \widetilde K_n^{1/2}/\gamma_n^{1/q} + Cb\widetilde K_n/(\gamma_n^{1/q}n^{1/2-1/q}).
\end{align*}
\end{lemma}
\begin{proof}
	To bound $t_n^\sigma$ note that
$$ \left|\frac{\hat \sigma_{\theta,j}}{\sigma_{\theta,j}}-1\right| = \left|\frac{\hat\sigma_{\theta,j}-\sigma_{\theta,j}}{\sigma_{\theta,j}}\right| =  \left| \frac{\hat \sigma_{\theta,j}^2-\sigma_{\theta,j}^2}{\sigma_{\theta,j}(\hat \sigma_{\theta,j}+\sigma_{\theta,j})}\right| \leq \frac{|\sigma_{\theta,j}^2-\hat\sigma_{\theta,j}^2|}{\sigma_{\theta,j}^2}.  $$
By the same argument as in (\ref{eq:defE}) and (\ref{eq:ControlDelta}), we have with probability exceeding $1-\gamma_n$,
$$\sup_{\theta\in\Theta(\eeta),j\in[p]} \frac{|\sigma_{\theta,j}^2-\hat\sigma_{\theta,j}^2|}{\sigma_{\theta,j}^2} \leq  \frac{Cb\sigma \widetilde K_n^{1/2}}{\gamma_n^{2/q}n^{1/2}}+ \frac{Cb^2\widetilde K_n}{\gamma_n^{2/q}n^{1-2/q}}.$$
Since $|1-a|\leq \epsilon<1$ implies $|1-a^{-1}|\leq \epsilon/(1-\epsilon)$, two times the right hand side of the equation above a valid (upper) bound on $t_n^\sigma$ provided the right hand side is less than 1/2.

Note that with probability exceeding $1-\gamma_n$, we have $$|\hat v_{\theta,j}^*|\leq (1+t_n^\sigma)|\hat v_{\theta,j}^*|\hat\sigma_{\theta,j}/\sigma_{\theta,j}.$$ By similar arguments we have that with probability exceeding $1-\gamma_n$
$$\sup_{\theta\in\Theta(\eeta),j\in[p]} |v_{\theta,j}|\vee |\hat v_{\theta,j}^*|\leq \frac{C\sigma \widetilde K_n^{1/2}}{\gamma_n^{1/q}}+\frac{Cb\widetilde K_n}{\gamma_n^{1/q}n^{1/2-1/q}}.$$
which makes the right hand side of the equation above a valid (upper) bound on $\bar w_n$.

Finally, we bound $\mu_n$. Note that

$$
\begin{array}{rl}
 (I) \equiv & \sup_{  \theta, \tilde \theta \in \Theta(\eeta), j \in [p], \| \theta - \tilde \theta \|\leq \Delta_n} \left| v_{\theta,j}- v_{\tilde\theta,j}\right|\\
& \leq \sup_{ \theta, \tilde \theta \in \Theta(\eeta),  \| \theta - \tilde \theta \|\leq \Delta_n, j\in [p]} |\Gn(\sigma_{\theta,j}^{-1}m_j(W,\theta)-\sigma_{\tilde\theta,j}^{-1}m_j(W,\tilde\theta))|\\
\end{array}
$$
The class of functions $\mathcal{F}_{\mu_n} \equiv \{ m_j(W,\theta)/\sigma_{\theta,j}-m_j(W,\tilde\theta)/\sigma_{\tilde\theta,j} :  \theta, \tilde \theta \in \Theta(\eeta), \| \theta - \tilde \theta \|\leq \frac{\varepsilon + 2\bar w_n}{\kappa_n^{-1}\sqrt{n}\vartheta_n}, j\in [p] \}$ has an envelope $F_{\mu_n}$ bounded by $CF$ and entropy number bounded by a constant times the entropy number $\widetilde K_n$ of $\mathcal{F}$ as we have $\mathcal{F}_{\mu_n}\subset \mathcal{F}-\mathcal{F}$. Therefore, since $\Ep[m_j(W,\theta)/\{\sigma_{\theta,j}-m_j(W,\tilde\theta)/\sigma_{\tilde\theta,j}\}^2] \leq L_C\|\theta-\tilde\theta\|^\chi$,  by Lemmas 6.1 and 6.2 in \citet{chernozhukov2015noncenteredprocesses} with $t=\gamma_n^{-2/q}$ and $\alpha = \gamma_n^{1/q}$, we have that with probability exceeding $1-\gamma_n$,
\begin{equation}\label{bound:eqI}
	(I) ~\leq~ C L_C^{1/2}\Delta_n^{\chi/2} \widetilde K_n^{1/2}/\gamma_n^{1/q} + Cb\widetilde K_n/(\gamma_n^{1/q}n^{1/2-1/q}),
\end{equation}
which makes the right hand side a valid choice of $\mu_n$.
\end{proof}

\begin{lemma}\label{lemma:MimpliesAB}
	Assume that Condition MB holds and $d_\theta< n$. Then, Conditions A and B hold with $\Lambda = \Theta(\eeta)$,
$\bar \chi=1\wedge \chi/2$, $\gamma_n = C\{\tilde \gamma_n + n^{-1}\}$, and
\begin{align*}
	C_{\mathcal{F},\tilde \gamma_n} & = C\left\{\sqrt{n}L_G + L^{1/2}_C\widetilde K_n^{1/2}\tilde \gamma_n^{-1/q}\right\}\\
	\eta_{\mathcal{F},\tilde \gamma_n} & =  C\left\{\frac{(b\sigma)^{1/2} \widetilde K_n^{3/4}}{\tilde \gamma_n^{1/q}n^{1/4}}+\frac{b\widetilde K_n}{\tilde \gamma_n^{1/q}n^{1/2-1/q}} + \frac{ \sigma \widetilde K_n^{1/2}}{n^{1/2}}\right\},
\end{align*}
where $\widetilde K_n \equiv v\{\log n \vee \log(Ab/\sigma)\}$.
\end{lemma}
\begin{proof}
Condition MB(i) implies Condition A for the class of functions  $\mathcal{F}_\theta = \{ m_j(\cdot,\theta)/\sigma_{\theta,j} : \theta \in \Theta(\eeta), j\in [p]\}$. Condition B(i) holds by definition of $\mathcal{F}_\theta$ and that $\Lambda = \Theta(\eeta)$ satisfies ${\rm diam}(\Lambda) \leq n$ as $\sup_{\theta \in \Theta(\eeta)}\|\theta\| \leq \sqrt{d_\theta}\sup_{\theta \in \Theta(\eeta)}\|\theta\|_\infty \leq \sqrt{d_\theta n} < n$. To complete the verification, note that for any $\theta,\tilde \theta \in \Theta(\eeta)$, $\|\theta-\tilde\theta\|\leq \epsilon$,
$$\begin{array}{l}
\left| \sup_{f\in \mathcal{F}_\theta} B(f)+ \G(f) - \sup_{\tilde f\in \mathcal{F}_{\tilde \theta}} B(\tilde f)+ \G(\tilde f)  \right|\\
\leq \left| \sup_{f\in \mathcal{F}_\theta, \tilde f\in \mathcal{F}_{\tilde \theta}} B(f)+ \G(f) - \{B(\tilde f)+ \G(\tilde f)\}  \right|\\
\leq \sup_{\|\theta-\tilde \theta\| \leq \epsilon, j\in[p]} |B(\{\theta,j\})-B(\{\tilde\theta,j\})|\\
+\sup_{\|\theta-\tilde \theta\| \leq \epsilon,j\in[p]}|\G(\{\theta,j\})-\G(\{\tilde \theta,j)\}|
\end{array}, $$ since $\mathcal{F}_\theta = \{m_j(\cdot,\theta)/\sigma_{\theta,j} : j\in [p]\}$ for all $\theta\in\Theta(\eeta)$ and we can associate the function $m_j(\cdot,\theta)/\sigma_{\theta,j}$ with the index $f=\{\theta,j\}$.

Condition MB implies that
\begin{equation}\label{Bpart0}
	\begin{array}{rl}
		|B(\{\theta,j\})-B(\{\tilde\theta,j\})| & = \sqrt{n}|\Ep[\sigma_{\theta,j}^{-1}m_j(W,\theta)-\Ep[\sigma_{\tilde\theta,j}^{-1}m_j(W,\tilde\theta)]| \\
 &\leq \sqrt{n}L_G\|\theta-\tilde\theta\| \leq \sqrt{n}L_G\epsilon.
\end{array}.
\end{equation}

Next, for $\G=G_P, \Gn,$ and $\Gn^\xi$, we now provide an upper bound for
\begin{equation}\label{toshowB}
\sup_{\|\theta-\tilde \theta\| \leq \epsilon,j\in[p]}|\G(\{\theta,j\})-\G(\{\tilde \theta,j)\}|.
\end{equation}


For $\G=\Gn$, the same calculation as in (\ref{bound:eqI}) yields
\begin{equation}\label{Bpart1}
\begin{array}{rl}
	\sup_{\|\theta-\tilde \theta\| \leq \epsilon, j\in[p]} |\Gn(m_j(W,\theta)/\sigma_{\theta,j}-m_j(W,\tilde\theta)/\sigma_{\tilde\theta,j})| \\
\leq \frac{C L_C^{1/2}\epsilon^{\chi/2} \widetilde K_n^{1/2}}{\gamma_n^{1/q}} + \frac{C_qb\widetilde K_n}{\gamma_n^{1/q}n^{1/2-1/q}}.
\end{array}
\end{equation}
Next we bound (\ref{toshowB}) for $\G=G_P$. Since $\Ep[\{\G(\{\theta,j\})-\G(\{\tilde\theta,j\}) \}^2]=\Ep[\{m_j(W,\theta)/\sigma_{\theta,j}-m_j(W,\tilde\theta)/\sigma_{\tilde\theta,j}\}^2],$ we have $\sup_{\|\theta-\tilde \theta\| \leq \epsilon, j\in[p]} \Ep[\{\G(\{\theta,j\})-\G(\{\tilde\theta,j\}) \}^2] \leq  L_C\epsilon^{\chi}$. Then we apply the Borell-Sudakov-Tsirel'son inequality combined with Dudley's maximal inequality, as in Step 2 in the proof of Theorem 2.1 of \citet{chernozhukov2015noncenteredprocesses}. Thus, with probability exceeding $1-2n^{-1}$,
\begin{equation}\label{Bpart2}
 \sup_{\|\theta-\tilde \theta\| \leq \epsilon, j\in[p]} |\G(\{\theta,j\})-\G(\{\tilde\theta,j\})| < CL_C^{1/2}\epsilon^{\chi/2}\widetilde K_n^{1/2}+L_C^{1/2}\epsilon^{\chi/2}\sqrt{2\log n}.\end{equation}
To bound (\ref{toshowB}) for $\G=\Gn^\xi$, we consider the event $E$ defined in (\ref{eq:defE}), and define the following:
\begin{align*}
	Z_\epsilon &\equiv \sup_{\|\theta-\tilde \theta\| \leq \epsilon, j\in[p]} |\Gn^\xi(m_j(W,\theta)/\sigma_{\theta,j})-\Gn^\xi(m_j(W,\tilde\theta)/\sigma_{\tilde\theta,j})|\\
	\sigma_n^2 &\equiv \sup_{\|\theta-\tilde\theta\|\leq \epsilon, j\in[p]} \En[\{\sigma_{\theta,j}^{-1}m_j(W,\theta)-\sigma_{\tilde\theta,j}^{-1}m_j(W,\tilde\theta)\}^2]\\
	\mathcal{F}_\epsilon &\equiv\{ m_j(W,\theta)/\sigma_{\theta,j}-m_j(W,\tilde\theta)/\sigma_{\tilde \theta,j} : \|\theta-\tilde \theta\| \leq \epsilon, j\in[p]\}.
\end{align*}
Step 0 in the proof of Theorem 2.2 in \citet{chernozhukov2015noncenteredprocesses} established that $\P(E) \geq 1-\gamma_n-n^{-1}$. Since $\Gn^\xi$ is a centered Gaussian process conditional on $(W_i)_{i=1}^{n}$, the Borell-Sudakov-Tsirel'son inequality implies that
\begin{align*}
	\left.P\left( Z_\epsilon > \Ep[Z_\epsilon| (W_i)_{i=1}^n] + \sigma_n\sqrt{2\log n}~\right|~ (W_i)_{i=1}^n \right) \leq 2n^{-1},
\end{align*}
where $|Z_\epsilon| \leq \sup_{f\in\mathcal{F}_\epsilon} \left|\frac{1}{\sqrt{n}}\sum_{i=1}^n\xi_if(W_i)\right|+ \sup_{f\in\mathcal{F}_\epsilon} \left|\frac{1}{\sqrt{n}}\sum_{i=1}^n\xi_i\En[f(W_i)]\right|$.
As in Step 2 in the proof of Theorem 2.2 in \citet{chernozhukov2015noncenteredprocesses},
$$ \Ep[Z_\epsilon\mid (W_i)_{i=1}^n] \leq C(\sigma_n \vee (\sigma n^{-1/2}))\widetilde  K_n^{1/2}.$$
Since $E$ implies
$$ \sigma_n^2 \leq \sup_{f\in\mathcal{F}_\epsilon}\Ep[f^2(W)] + n^{-1/2}\sup_{f\in\mathcal{F}_\epsilon}|\Gn(f^2)| \leq L_C\epsilon^\chi + \frac{b\sigma \widetilde K_n^{1/2}}{\gamma_n^{2/q}n^{1/2}}+\frac{b^2\widetilde K_n}{\gamma_n^{2/q}n^{1-2/q}},$$
it follows that with probability exceeding $1-\gamma_n-3n^{-1}$ that
\begin{equation}\label{Bpart3}
	\begin{array}{rl}
 |Z_\epsilon| & \leq \Ep[Z_\epsilon\mid (X_i)_{i=1}^n] + \sigma_n\sqrt{2\log n}   \\
 & \lesssim \sigma_n \widetilde K_n^{1/2} + n^{-1/2}\sigma \widetilde K_n^{1/2} + \sigma_n\sqrt{2\log n}\\
 & \lesssim L_C^{1/2}\epsilon^{\chi/2} \widetilde K_n^{1/2} + \frac{(b\sigma)^{1/2} \widetilde K_n^{3/4}}{\gamma_n^{1/q}n^{1/4}}+\frac{b\widetilde K_n}{\gamma_n^{1/q}n^{1/2-1/q}} + \frac{ \sigma \widetilde K_n^{1/2}}{n^{1/2}},
\end{array}
\end{equation}
where we used that $\log n \leq C\widetilde K_n$.

Combining (\ref{Bpart0}), (\ref{Bpart1}), (\ref{Bpart2}), and (\ref{Bpart3}), Condition B holds with the parameters specified in the statement.
\end{proof}


\begin{lemma}\label{lemma:Ineq00}
Assume that
$\ell_{min}< - 8\sup_{\theta \in \Theta(\eeta), j\in [p]} |v_{\theta,j}|\vee |\hat v_{\theta,j}^*|$ and
\begin{equation}\label{eq:Ineq00}
	\frac{7}{8}\geq \frac{1}{8}\left(1 + \kappa_n^{-1}\sup_{\theta\in\Theta(\eeta),j\in[p]}\frac{\sigma_{\theta,j}}{\hat\sigma_{\theta,j}} \right) \sup_{\theta \in  \Theta(\eeta), j\in[p]} \frac{\hat\sigma_{\theta,j}}{\sigma_{\theta,j}} + \kappa_n^{-1}.
\end{equation}
Then, $T_n(\eeta) \leq R_n^{PR*}$.
\end{lemma}
\begin{proof}
First, consider the following derivation.
{\small \begin{align*}
R_n^{PR*} & = \inf_{\theta\in \Theta(\eeta)}\max_{j\in [p]}\left\{ \hat v_{\theta,j}^* +\kappa_n^{-1}\sqrt{n}\bar m_{\theta,j} /\hat \sigma _{\theta,j} \right\}   \\
 & \geq  \inf_{\theta\in \Theta(\eeta)}\max_{j\in [p]}\left\{ \hat v_{\theta,j}^* +\kappa_n^{-1}\sqrt{n}\Ep\left[ m_{j}\left( W,\theta\right) \right] /\hat \sigma _{\theta,j} \right\} - \kappa^{-1}_n \sup_{\theta\in\Theta(\eeta),j\in[p]}|v_{\theta,j}|\frac{\sigma_{\theta,j}}{\hat\sigma_{\theta,j}}  \\
& \geq  - \sup_{\theta \in \Theta(\eeta), j\in [p]} |\hat v_{\theta,j}^*| + \inf_{\theta\in \Theta(\eeta)}\max_{j\in [p]}\kappa_n^{-1}\sqrt{n}\Ep\left[ m_{j}\left( W,\theta\right) \right]/\hat\sigma _{\theta,j}  \\
 & - \kappa^{-1}_n \sup_{\theta\in\Theta(\eeta),j\in[p]}|v_{\theta,j}|\frac{\sigma_{\theta,j}}{\hat\sigma_{\theta,j}}   \\
& \geq  \frac{1}{8}\left(1 + \kappa_n^{-1}\sup_{\theta\in\Theta(\eeta),j\in[p]}\frac{\sigma_{\theta,j}}{\hat\sigma_{\theta,j}} \right)\inf_{\theta \in  \Theta(\eeta)}\max_{j\in[p]} \sqrt{n}\Ep[m_j(W,\theta)]/\sigma_{\theta,j} \\
 & + \inf_{\theta\in \Theta(\eeta)}\max_{j\in [p]}\kappa_n^{-1}\frac{\sqrt{n}\Ep\left[ m_{j}\left( W,\theta\right) \right]}{\hat\sigma _{\theta,j}}    \\
& \geq   \left\{\frac{1}{8}\left(1 + \kappa_n^{-1}\sup_{\theta\in\Theta(\eeta),j\in[p]}\frac{\sigma_{\theta,j}}{\hat\sigma_{\theta,j}} \right) \inf_{\theta \in  \Theta(\eeta), j\in[p]} \frac{\hat\sigma_{\theta,j}}{\sigma_{\theta,j}} + \kappa_n^{-1} \right\} \cdot\\
& \cdot \inf_{\theta \in  \Theta(\eeta)}\max_{j\in[p]} \frac{\sqrt{n}\Ep[m_j(W,\theta)]}{\hat \sigma_{\theta,j}}.
\end{align*}}
Second, consider the following derivation.
\begin{align*}
T_n(\eeta) & = \inf_{\theta\in \Theta(\eeta)}\max_{j\in [p]}\left\{ v_{\theta,j} +\sqrt{n}\Ep\left[ m_{j}\left( W,\theta\right) \right] /\sigma _{\theta,j} \right\} \frac{\sigma _{\theta,j}}{\hat \sigma _{\theta,j}}   \\
& \leq \inf_{\theta\in \Theta(\eeta)} \sup_{\theta \in \Theta(\eeta), j\in [p]} \left\{ |v_{\theta,j}|  + \max_{j\in [p]}\frac{\sqrt{n}\Ep\left[ m_{j}\left( W,\theta\right) \right]}{\sigma _{\theta,j}}\right\} \frac{\sigma _{\theta,j}}{\hat \sigma _{\theta,j}}    \\
& \leq  \inf_{\theta\in \Theta(\eeta)} \frac{7}{8} \max_{j\in [p]}\frac{\sqrt{n}\Ep\left[ m_{j}\left( W,\theta\right) \right]}{\sigma _{\theta,j}} \frac{\sigma _{\theta,j}}{\hat \sigma _{\theta,j}}    \\
& =  \frac{7}{8} \inf_{\theta\in \Theta(\eeta)}  \max_{j\in [p]}\frac{\sqrt{n}\Ep\left[ m_{j}\left( W,\theta\right) \right]}{\hat\sigma _{\theta,j}}.
\end{align*}

The result then follows from combining the previous two derivations, (\ref{eq:Ineq00}), and the fact that $\inf_{\theta\in \Theta(\eeta)}  \max_{j\in [p]}{\sqrt{n}\Ep\left[ m_{j}\left( W,\theta\right) \right]}/{\hat\sigma _{\theta,j}} < 0$.
\end{proof}

\begin{lemma}
\label{thm:Ineq01}Assume Condition MB holds, $\Theta(\eeta) \cap \Theta _{I}\left(
F\right) \not=\emptyset $ and let
$$\begin{array}{rl}
 \mathcal{W}_n & \equiv \sup_{\theta\in\Theta(\eeta),j\in[p]}|v_{n,j}(\theta)|, \\
 \mathcal{T}_n^\sigma & \equiv \sup_{\theta\in\Theta(\eeta),j\in[p]}\left|\frac{\sigma_{\theta,j}}{\hat\sigma_{\theta,j}}-1\right|, \ \ \mbox{and} \ \ \\
 \mathcal{M}_n &\equiv  \sup_{\|\theta-\tilde\theta\|\leq \frac{\varepsilon+2\mathcal{W}_n}{\kappa_n^{-1}\sqrt{n}}, j\in[p]}|v_{\theta,j}-v_{\tilde\theta,j}|.\end{array}$$
Then for any $\varepsilon >0$,%
$$
R_{n}\leq S_{n}^\kappa + \varepsilon + t_n^\sigma(\varepsilon+3\mathcal{W}_n)+L_G\frac{\kappa_n}{\sqrt{n}}(1+\mathcal{T}_n^\sigma)\{(\varepsilon+2\mathcal{W}_n)/\vartheta_n\}^2+(1+t_n^\sigma)\mathcal{M}_n.
$$ where $S_n^\kappa \equiv  \inf_{\theta \in \Theta(\eeta) }\max_{j\in [p]}\left\{ v_{\theta,j} +\sqrt{n}\kappa_n^{-1}\Ep\left[ m_{j}\left( W,\theta \right) \right] /\sigma_{\theta,j}\right\}$.

\end{lemma}
\begin{proof}[Proof of Lemma \ref{thm:Ineq01}]
By definition of $R_n$ we have
\begin{eqnarray*}
R_{n} &\equiv &\inf_{\theta \in \Theta(\eeta)}\max_{j\in[p]}\sqrt{n%
}\bar{m}_{\theta,j} /\hat{\sigma}_{\theta,j}  \\
&=&\inf_{\theta \in \Theta(\eeta) }\max_{j\in[p]}\left\{
v_{\theta,j} \sigma _{\theta,j} /\hat{\sigma}_{\theta,j} +\sqrt{n}\Ep[ m_j( W,\theta)]/\hat{\sigma}_{\theta,j} \right\}  \\
&=&\inf_{\theta \in \Theta(\eeta) }\max_{j\in[p]} \left\{v_{\theta,j} +\sqrt{n}\Ep[ m_{j}( W,\theta)] /\sigma _{\theta,j} \right\}\sigma_{\theta,j}/\hat\sigma_{\theta,j}
\end{eqnarray*}%
and $ S_n^\kappa \equiv \inf_{\theta \in \Theta(\eeta) }\max_{j\in[p]}\left\{
v_{\theta,j}+\kappa_n^{-1}\sqrt{n}\Ep[m_j( W,\theta)] /\sigma_{\theta,j} \right\}$. Since $\Theta(\eeta)\cap\Theta_{I}\not=\emptyset $, we have
%
%

\begin{eqnarray}\label{eq:52}
S_{n}^\kappa &\leq &\max_{j\in [p]}\left\{ v_{\tilde{\theta}_{n},j} +%
\kappa_n^{-1}\sqrt{n}\Ep [ m_{j}(W,\tilde{\theta}_{n}) ] /\sigma
_{\tilde{\theta}_{n},j} \right\} \notag \\
&\leq &\max_{j\in [p]}  v_{\tilde{\theta}_{n},j}
  \notag \\
&\leq &\sup_{\theta \in \Theta(\eeta)}\max_{j\in[p]}| v_{\theta ,j}|.  \label{eq:UpperB2}
\end{eqnarray}

By definition of infimum, $\exists \theta_{n}\in \Theta(\eeta)$
s.t.
\begin{eqnarray}
S_{n}^\kappa & \geq & \max_{j\in [p]}\left\{ v_{\theta_{n},j} +\kappa_n^{-1}\sqrt{n}\Ep[ m_{j}( W,\theta_{n})] /\sigma_{\theta_n,j} \right\} -\varepsilon \notag  \\
& \geq & -\sup_{\theta\in\Theta(\eeta)}\max_{j\in[p]}|v_{\theta,j}|+\max_{j\in[p]}\kappa_n^{-1}\sqrt{n}\Ep[m_j(W,\theta_n)]/\sigma_{\theta_n,j} - \varepsilon \notag\\
& \geq & -\sup_{\theta\in\Theta(\eeta)}\max_{j\in[p]}|v_{\theta,j}|+\kappa_n^{-1}\sqrt{n}\Ep[m_j(W,\theta_n)]/\sigma_{\theta_n,j} - \varepsilon \label{eq:53}
\end{eqnarray}%
Thus by (\ref{eq:52}) and (\ref{eq:53}) we have
\begin{equation}\label{eq:up01}
l_{n,j} := \frac{\kappa_n^{-1}\sqrt{n}\Ep%
\left[ m_{j}( W,\theta_n) \right]}{\sigma _{\theta_n,j}}  \leq \varepsilon  + 2 \sup_{\theta \in \Theta(\eeta)}\max_{j\in[p]}|v_{\theta,j}|, \ \ j\in [p].\end{equation}

First assume that $\max_{j\in [p]} l_{n,j} \geq 0$. By Lemma  \ref{lem:OtherThetaConstruction}, there exist $\breve{\theta}_n \in \Theta(\eeta)$ such that
\[ \|\theta _{n}-\breve{\theta}_{n}\|
\leq \frac{{\displaystyle \max_{j\in[p]}}  \frac{l_{n,j}}{\vartheta_n}}{\kappa _{n}^{-1}n^{1/2}} \ \ \ \mbox{and} \ \ \ \frac{\sqrt{n}\Ep[ m_{j}( W,\breve{\theta}_{n})]}{\sigma
_{\breve{\theta}_{n},j}} \leq l_{n,j}+ L_G  \frac{\kappa _{n}}{n^{1/2}} \{\max_{j\in [p]}l_{n,j}/\vartheta_n\}^2 .
\]
Therefore we have
\begin{eqnarray}
S_{n}^\kappa & \geq & \max_{j\in[p]}\left\{ v_{\theta _{n},j} +\kappa_n^{-1}\sqrt{n}\Ep[ m_{j}( W,\theta_{n})] /\sigma_{\theta_n,j} \right\} -\varepsilon \notag  \\
& = & \max_{j\in [p]}\left\{ v_{\theta_{n},j} +\kappa_n^{-1}\sqrt{n}\Ep
[ m_{j}( W,\theta _{n})] /\sigma_{\theta_n,j} \right\}\frac{\sigma_{\breve\theta_n,j}}{\hat\sigma _{\breve\theta_n,j}} -\varepsilon + a_{1n} \notag  \\
& \geq & \max_{j\in[p]}\left\{ v_{\theta _{n},j} +\sqrt{n}\Ep[ m_{j}( W,\breve\theta _{n}) ] /\sigma _{\breve\theta_n,j} \right\}\frac{\sigma _{\breve\theta_n,j}}{\hat\sigma _{\breve\theta_n,j}} \notag  \\
& & -\varepsilon + a_{1n} - L_G\frac{\kappa_n}{\sqrt{n}} \{\max_{j\in [p]} l_{nj}/\vartheta_n\}^2 \frac{\sigma _{\breve\theta_n,j}}{\hat\sigma _{\breve\theta_n,j}} \notag  \\
& \geq & \max_{j\in [p]}\left\{ v_{\breve\theta_{n},j} +\sqrt{n}\Ep[ m_{j}(W,\breve\theta_{n})]/\sigma_{\breve\theta_n,j} \right\}\frac{\sigma_{\breve\theta_n,j}}{\hat\sigma _{\breve\theta_n,j}} \notag  \\
 & & -\varepsilon + a_{1n} - L_G\frac{\kappa_n}{\sqrt{n}} \{\max_{j\in [p]} l_{nj}/\vartheta_n\}^2 \frac{\sigma _{\breve\theta_n,j}}{\hat\sigma _{\breve\theta_n,j}} +\frac{\sigma _{\breve\theta_n,j}}{\hat\sigma _{\breve\theta_n,j}} \{ \max_{j\in[p]} v_{\theta _{n},j}- \max_{j\in [p]} v_{\breve\theta _{n},j}\}\notag  \\
& \geq & R_n -\varepsilon + a_{1n} - L_G\frac{\kappa_n}{\sqrt{n}} \{\max_{j\in [p]} l_{nj}/\vartheta_n\}^2 \frac{\sigma _{\breve\theta_n,j}}{\hat\sigma _{\breve\theta_n,j}} - \frac{\sigma _{\breve\theta_n,j}}{\hat\sigma _{\breve\theta_n,j}} \max_{j\in [p]} |v_{\theta _{n},j} - v_{\breve\theta _{n},j}|\notag  \\
 \notag
\end{eqnarray}%
where\footnote{Indeed by definition
$$ \begin{array}{rl}a_{1n}
& = \max_{j\in[p]}\left\{ v_{\theta_{n},j} +\kappa_n^{-1}\sqrt{n}\Ep[ m_{j}( W,\theta_{n})] /\sigma_{\theta_n,j} \right\}\\
 & - \max_{j\in[p]}\left\{ v_{\theta_{n},j} +\kappa_n^{-1}\sqrt{n}\Ep[m_{j}( W,\theta_{n})] /\sigma_{\theta_n,j} \right\}\frac{\sigma_{\breve\theta_n,j}}{\hat\sigma _{\breve\theta_n,j}}  \end{array}$$} $|a_{1n}| \leq \max_{j\in[p]}\{\left| v_{\theta _{n},j} +\kappa_n^{-1}\sqrt{n}\Ep[ m_{j}(W,\theta_{n})] /\sigma _{\theta_n,j} \right| \ |1-\frac{\sigma_{\breve\theta_n,j}}{\hat\sigma_{\breve\theta_n,j}}|\}$.

Because $\max_{j\in [p]} l_{n,j} \geq 0$ we have that
$$\begin{array}{rl}
|a_{1n}| &\displaystyle \leq \left\{\sup_{\theta \in \Theta(\eeta)} \max_{j\in[p]}|v_{\theta,j}| +  \max_{j\in[p]} l_{n,j}\right\}\sup_{\theta\in\Theta(\eeta),j\in[p]}\left|1-\frac{\sigma_{\theta,j}}{\hat\sigma_{\theta,j}}\right|\\
& \leq \left\{\varepsilon + 3\sup_{\theta \in \Theta(\eeta)} \max_{j\in[p]}|v_{\theta,j}| \right\}\sup_{\theta\in\Theta(\eeta),j\in[p]}\left|1-\frac{\sigma_{\theta,j}}{\hat\sigma_{\theta,j}}\right| \\
\end{array}
$$
where the second step follows from (\ref{eq:up01}).

In the other case, $\max_{j\in[p]}l_{n,j} < 0$, we have that
\begin{eqnarray}
S_{n}^\kappa & \geq & \max_{j\in [p]}\left\{ v_{\theta _{n},j} +\kappa_n^{-1}\sqrt{n}\Ep[ m_{j}( W,\theta_{n})] /\sigma_{\theta_n,j} \right\} -\varepsilon \notag  \\
& \geq & \max_{j\in[p]}\left\{ v_{\theta _{n},j} + \sqrt{n}\Ep[m_{j}(W,\theta_{n}) ] /\sigma _{\theta_n,j} \right\}\frac{\sigma_{\theta_n,j}}{\hat\sigma_{\theta_n,j}} -\varepsilon \notag  \\
& & -\left|1-\frac{\sigma _{\theta_n,j}}{\hat\sigma _{\theta_n,j}}\right| \max_{j\in[p]}|v_{\theta_n,j}| \notag\\
& \geq & R_n - \varepsilon    -\left|1-\frac{\sigma_{\theta_n,j}}{\hat\sigma _{\theta_n,j}}\right| \max_{j\in[p]}|v_{\theta_n,j}|\notag
\end{eqnarray}
where the second inequality holds provided that $\frac{\sigma _{\theta_n,j}}{\hat\sigma _{\theta_n,j}} \geq \kappa_n^{-1}$ because we have $\max_{j\in[p]}l_{n,j} < 0$.

\end{proof}

\begin{lemma}
\label{lem:OtherThetaConstruction} Assume Condition MB, $\Theta (\eeta) \cap
\Theta_I \not=\emptyset $ and  $\kappa_n\geq 1$. Consider an arbitrary (possibly
random) $\theta _{n}\in \Theta \left(\eeta\right) $ and define the vector $$l_{n}\equiv
\left\{ \kappa _{n}^{-1}\sqrt{n}\Ep\left[ m_{j}\left( W,\theta _{n}\right) %
\right] /\sigma _{\theta _{n},j} :j\in [p]\right\} \in \mathbb{R}%
^{p}.$$ Then, there exists a (possibly random) $\breve{\theta}_{n}\in \Theta
(\eeta) $ s.t.%
{\small \[ \left\Vert \theta _{n}-\breve{\theta}_{n}\right\Vert
\leq \frac{0\vee \max_{j\in [p]}  l_{n,j}/\vartheta_n}{\kappa _{n}^{-1}\sqrt{n}} \ \ \ \mbox{and} \ \ \ \frac{\sqrt{n}\Ep[ m_{j}( W,\breve{\theta}_{n})]}{\sigma
_{\breve{\theta}_{n},j}} \leq l_{n,j}+ \frac{L_G\kappa _{n}}{\sqrt{n}} \left\{\max_{j\leq
p}\frac{l_{n,j}}{\vartheta_n}\right\}^2
\]}
for all $j\in[p]$.

\end{lemma}
\begin{proof}
First note that if $\theta_n \in \Theta (\eeta) \cap
\Theta_I$, we can trivially take $\breve \theta_n=\theta_n$. Indeed the first relation is trivial and the second holds because $l_{n,j}\leq 0$, $\kappa_n \geq 1$ and $\sqrt{n}\Ep[ m_j( W,\breve{\theta}_{n})]/\sigma_{\breve{\theta}_{n},j} = \kappa_n l_{n,j}$.

Next we consider the case that $\theta_n \not\in \Theta ( h) \cap
\Theta_I$. Condition MB implies that
\[
\max_{j\in [p]}  \Ep[
m_{j}( W,\theta _{n})] /\sigma _{\theta
_{n},j}   \geq \vartheta_n \min \{ \delta ,\inf_{\tilde{\theta}\in
\Theta(h) \cap \Theta_I }\| \theta
_{n}-\tilde{\theta}\| \},
\]%
and therefore:%
\[\begin{array}{rl}
{\displaystyle \max_{j\in[p]}} \frac{\kappa_{n}^{-1}\sqrt{n}\Ep[ m_{j}(W,\theta_{n})] }{\sigma_{\theta_n, j}} & =\max_{j\in[p]}  l_{n,j}  \\
&  \geq \kappa_{n}^{-1}\sqrt{n} \vartheta_n \min \left\{ \delta
,\inf_{\tilde{\theta}\in \Theta (\eeta) \cap \Theta_I }\| \theta_{n}-\tilde{\theta}\| \right\} .
\end{array}\]%
This implies that we can find $\tilde{\theta}_{n}\in \Theta(h)
\cap \Theta_I $ s.t.
\begin{equation}\label{eq:rate:interm}
\| \theta _{n}-\tilde{\theta}_{n}\| \kappa _{n}^{-1}\sqrt{n}\vartheta_n %
\leq \max_{j\in [p]}  l_{n,j}  .
\end{equation}

By the intermediate value theorem, for each $j\in[p]$, we can then find $\theta _{n}^{\ast,j }\in
\Theta ( h) $ between $\theta _{n}$ and $\tilde{\theta}_{n}$,
i.e., $\theta _{n}^{\ast,j}=\alpha _{nj}\theta _{n}+( 1-\alpha
_{nj}) \tilde{\theta}_{n}$ for some $\alpha_{nj}\in [0,1] $%
, s.t.%
{\small \begin{equation}\label{eq:test}
\frac{\sqrt{n}\Ep[ m_{j}( W,\theta _{n})]
}{\kappa _{n}\sigma_{\theta_{n},j}} =\frac{\sqrt{n}\Ep[
m_{j}( W,\tilde{\theta}_{n})]}{\kappa _{n}\sigma _{\tilde{%
\theta}_{n},j}} +\kappa _{n}^{-1}\sqrt{n}G_j( \theta _{n}^{\ast,j
}) ( \theta _{n}-\tilde{\theta}_{n}) .
\end{equation}}%
Define $\breve{\theta}_{n}=( 1-\kappa _{n}^{-1}) \tilde{\theta}%
_{n}+\kappa _{n}^{-1}\theta _{n}$ or, equivalently, $( \breve{\theta}%
_{n}-\tilde{\theta}_{n}) =\kappa _{n}^{-1}( \theta _{n}-\tilde{%
\theta}_{n}) $, which can be done simultaneously for all $j\in [p]$. Therefore, we can rewrite (\ref{eq:test}) as
\begin{equation}
G_j( \theta_{n}^{\ast,j }) \sqrt{n}( \breve{\theta}_{n}-\tilde{%
\theta}_{n}) =\frac{\sqrt{n}\Ep[ m_{j}( W,\theta
_{n}) ]}{\kappa _{n}\sigma _{\theta _{n},j}} -
\frac{\sqrt{n}\Ep[ m_{j}( W,\tilde{\theta}_{n})]}{\kappa _{n}\sigma_{\tilde{\theta}_{n},j}} .  \label{eq:IVT1}
\end{equation}

By convexity, $\breve{\theta}_{n}\in \Theta (\eeta) $. Moreover, by definition of $\breve\theta_n$ and (\ref{eq:rate:interm}) we have
\[
\| \breve{\theta}_{n}-\tilde{\theta}_{n}\| \sqrt{n}\vartheta_n\leq
\max_{j\in [p]} l_{n,j} .
\]

By another application of the intermediate value theorem, we can find $%
\theta _{n}^{\ast \ast,j }=\beta _{nj}\breve{\theta}_{n}+\left(
1-\beta _{nj}\right) \tilde{\theta}_{n}$ for some $\beta _{nj}\in \left[ 0,1%
\right] $, s.t.%
{\small \begin{eqnarray}
\sqrt{n}\Ep[ m_{j}( W,\breve{\theta}_{n})] /\sigma_{\breve{\theta}_{n},j}  &=&\sqrt{n}\Ep[m_{j}( W,\tilde{\theta}_{n})] /\sigma _{\tilde{\theta}_{n},j} +G_j( \theta _{n}^{\ast \ast,j }) \sqrt{n}( \breve{%
\theta}_{n}-\tilde{\theta}_{n})   \notag \\
&=&\sqrt{n}\Ep[ m_{j}( W,\tilde{\theta}_{n})] /\sigma
_{\tilde{\theta}_{n},j} +G_j( \theta _{n}^{\ast,j })
\sqrt{n} ( \breve{\theta}_{n}-\tilde{\theta}_{n}) +\epsilon_{1,j,n},\label{eq:IVT2}
\end{eqnarray}}%
where%
$$
\epsilon _{1,j,n}\equiv \{ G_j\left( \theta _{n}^{\ast \ast,j }\right)
-G_j( \theta_{n}^{\ast,j }) \} \sqrt{n}( \breve{\theta}_{n}-\tilde{\theta}_{n}) .
$$
Combining \eqref{eq:IVT1} and \eqref{eq:IVT2}, we get:%
\begin{eqnarray*}
\sqrt{n}\Ep[ m_{j}( W,\breve{\theta}_{n})] /\sigma
_{\breve{\theta}_{n},j}  &=&\kappa _{n}^{-1}\sqrt{n}\Ep[
m_{j}(W,\theta_{n})] /\sigma_{\theta_{n},j} \\
&& +( 1-\kappa _{n}^{-1}) \sqrt{n}\Ep[ m_{j}( W,\tilde{\theta}_{n})]/\sigma_{\tilde{\theta}_{n},j} +\epsilon_{1,j,n} \\
&=&l_{n,j}+\epsilon _{2,n}+\epsilon_{1,j,n},
\end{eqnarray*}%
where%
\[
\epsilon _{2,j,n}\equiv ( 1-\kappa _{n}^{-1}) \sqrt{n}\Ep[
m_{j}( W,\tilde{\theta}_{n}) ] /\sigma_{\tilde{\theta}_{n},j} .
\]

From $\tilde{\theta}_{n}\in \Theta_I$ and $\kappa_{n} \geq 1$, we conclude that $\epsilon_{2,j,n}\leq 0$ for all $j\in[p]$. Moreover, it follows that \begin{eqnarray*}
|\epsilon _{1,j,n}|  &\leq &\left\Vert G_j( \theta
_{n}^{\ast \ast,j }) -G_j( \theta _{n}^{\ast,j }) \right\Vert
\| \sqrt{n}( \breve{\theta}_{n}-\tilde{\theta}_{n})\| \\
&\leq &\left\Vert G_j(\theta_{n}^{\ast \ast,j }) -G_j( \theta_{n}^{\ast,j }) \right\Vert \| \sqrt{n}( \breve{\theta}_{n}- \tilde{\theta}_{n}) \|  \\
&\leq &L_G \left\Vert \theta _{n}^{\ast \ast,j }-\theta _{n}^{\ast,j
}\right\Vert  \max_{j\in [p]}\left\{ l_{n,j}\right\}/\vartheta_n  \\
&=&L_G  \left\Vert \beta _{nj}\breve{\theta}_{n}+\left( 1-\beta
_{nj}\right) \tilde{\theta}_{n}-\alpha _{nj}\theta _{n}-\left( 1-\alpha
_{nj}\right) \tilde{\theta}_{n}\right\Vert  \max_{j\in [p]}
l_{n,j} /\vartheta_n  \\
&=&L_G \left\Vert \left( \alpha _{nj}-\beta_{nj}\kappa _{n}^{-1}\right)
\tilde{\theta}_{n}+\left( \beta _{nj}\kappa_{n}^{-1}-\alpha _{nj}\right)
\theta _{n}\right\Vert  \max_{j\in [p]} l_{n,j} /\vartheta_n  \\
&=&L_G|\alpha_n-\beta_n\kappa_n^{-1}| \ \|\tilde{\theta}_{n}-\theta_{n}\|
\max_{j\in [p]} l_{n,j}/\vartheta_n  \\
&\leq& L_G  \left\{\max_{j\in [p]}
l_{n,j}/\vartheta_n\right\}^{2} \kappa_{n}/\sqrt{n}.
\end{eqnarray*}%
Then, for all $j\in [p]$ we have%
\[
\frac{\sqrt{n}\Ep[ m_{j}( W,\breve{\theta}_{n})]}{\sigma_{\breve{\theta}_{n},j}} \leq \frac{\kappa _{n}^{-1}\sqrt{n}\Ep\left[
m_{j}\left( W,\theta _{n}\right) \right]}{\sigma _{\theta_n,j}} +L_G  \left\{\max_{j\in [p]}
l_{n,j}/\vartheta_n\right\}^{2} \frac{\kappa _{n}}{\sqrt{n}}
\]%
completing the proof of the second statement.

To show the first statement when $\theta_n \not\in\Theta(\eeta)\cap \Theta_I$ (i.e. $\max_{j\in [p]}  l_{n,j}>0$), recall that $\breve\theta_n = (1-\kappa_n^{-1})\tilde\theta_n + \kappa_n^{-1}\theta_n$. Therefore by (\ref{eq:rate:interm}), we have
{\small \[ \| \theta _{n}-\breve{\theta}_{n}\|
= (1-\kappa_n^{-1})\|\tilde\theta_n - \theta_n \| \leq \frac{\max_{j\in [p]}  l_{n,j}/\vartheta_n}{\kappa _{n}^{-1}\sqrt{n}} \]}
since $\kappa_{n} \geq 1$.
\end{proof}

\begin{proof}[Proof of Corollary \ref{cor:mainresult1}]
First we verify Condition MB. By assumption we have $\Theta(\eeta)$ convex and $\ell_\infty$-diameter uniformly bounded. Since the componentwise derivatives of $m_j$ are bounded by $C_1$, we can take $L_G = C_1\sqrt{d_\theta}$, $L_C = L_G^2$ and $\chi=2$. Condition MB(iii) holds by assumption.

Next we verify Condition A, because of the Lipchitz condition, we have that $\mathcal{F}$ has $\varepsilon$-covering number bounded by $p(6 C_1 d_\theta/\varepsilon)^{d_\theta}$ so it is VC type with $\bar A = 6C_1 d_\theta p^{1/d_\theta}$ and $v = d_\theta$. Moreover, since $m_j$'s are uniformly bounded above and the variances $\sigma_{\theta,j}$ are bounded away from zero, we can take $b\leq C_1$ and $\Ep[|f|^k] \leq \sigma^2 b^{k-2}$ trivially holds and we can take $q=16$. The set of functions $\mathcal{B}=\{B(f)=\sqrt{n}\Ep[m_j(W,\theta)] : j\in[p], \theta \in \Theta\}$ has covering number satisfying $N_B(\eta) \leq p(6C_1 d_\theta n/\eta)^{d_\theta}$. Thus we have $K_n \leq C\log p + C'd_\theta \log n$.

Thus the conditions to apply Theorem \ref{thm:PB:inferenceSubvector} hold with   $$\delta_{n,\gamma_n}^{PR} \lesssim  \frac{1}{\gamma_n}\left\{\frac{K_n^{2/3}}{n^{1/6}}+ \kappa_n\frac{d_\theta^{1/2}K_n}{n^{1/2}} + \frac{K_n^{1/2}}{\kappa_n} \right\}\lesssim \frac{1}{\gamma_n}\frac{C_1 n^{-c_1}}{\mathcal{A}_{n}^{PR}}.$$ Then by choosing $\gamma_n = C^2n^{-2c}$ for some $C$ sufficiently large and $c>0$ sufficiently small, the last result in Theorem \ref{thm:PB:inferenceSubvector} yields  the result.
\end{proof}


\begin{proof}[Proof of Theorem \ref{thm:MT:inferenceSubvector}]
The proof builds upon the proofs of Theorems \ref{thm:MSB:inferenceSubvector} and \ref{thm:PB:inferenceSubvector}. We now divide the argument into cases.

Case 1: $\ell_{\min} \geq -2 \bar w_n/(1-t_n^\sigma)$.
Let $$\begin{array}{rl}
S_n(\kappa) & = \inf_{\theta \in \Theta(\eeta) }\max_{j\in [p]}
v_{\theta,j} +\kappa_n^{-1}\sqrt{n}\Ep[ m_{j}( W,\theta ) %
] /\sigma_{\theta,j} \\
S_n^*(\kappa) & = \inf_{\theta \in \Theta(\eeta) }\max_{j\in [p]}
v_{\theta,j}^* +\kappa_n^{-1}\sqrt{n}\Ep[ m_{j}( W,\theta)] /\sigma_{\theta,j}
\end{array}$$
and $$ \begin{array}{rl}
 S_n(\psi_n,\Psi) & = \inf_{\theta \in \Theta_n^{\psi_n}}\max_{j\in \Psi_\theta} v_{\theta,j} +\sqrt{n}\Ep[m_{j}(W,\theta)]/\sigma_{\theta,j}\\
 S_n^*(\psi_n,\Psi) & = \inf_{\theta \in \Theta_n^{\psi_n}}\max_{j\in \Psi_\theta} v_{\theta,j}^* +\sqrt{n}\Ep[m_{j}(W,\theta)]/\sigma_{\theta,j}.
\end{array}$$

We have the following sequence of inequalities
$$
\begin{array}{rl}
 \P( T_n(\eeta) \geq t ) &  = \P( \min\{ T_n(\eeta), T_n(\eeta) \} \geq t ) \\
& \leq_{(1)} \P( \min\{ S_n(\psi_n,\Psi), S_n(\kappa)\} \geq t - \bar \delta_1 ) + \bar r_1 \\
& \leq_{(2)} \P( \min\{ S_n^*(\psi_n,\Psi), S_n^*(\kappa)\} \geq t - \bar \delta_2 ) + \bar r_2 \\
& \leq_{(3)} \P( \min\{ R_n^{DR*}, R_n^{PR*}\} \geq t - \bar \delta_3 ) + \bar r_3,
\end{array}
$$ where (1) holds with $\bar\delta_1 = \delta_3 \vee \delta_1'$ and $\bar r_1 := r_3 + r_1'$ where $\delta_3$ and $r_3$ are defined in the proof of Theorem \ref{thm:MSB:inferenceSubvector}, while $\delta_1'$  and $r_1'$ are defined in the proof of Theorem \ref{thm:PB:inferenceSubvector}.
Next, note that $\min\{ S_n(\psi_n,\Psi), S_n(\kappa)\}$ can be directly written as a MinMax statistic, so we can apply Theorems \ref{thm:clt:minmax} and \ref{thm:clt:minmax:Gaussian}. Thus (2) holds with $\bar \delta_2 = \bar\delta_1 + \delta_{n,\eta,\gamma_n} + \bar \delta_{n,\eta,\gamma_n} $ and $\bar r_2 = \bar r_1 +  C\{\gamma_n + n^{-1}\}$. Finally we have that (3) holds with $\bar \delta_3 = \bar \delta_2 + \delta_8 + \delta'_5$ and $\bar r_3 = \bar r_2 + r_8 + r_5'$ by the same arguments in in the proofs of Theorems \ref{thm:MSB:inferenceSubvector} and \ref{thm:PB:inferenceSubvector}.

Case 2: $\ell_{\min} \leq -2 \bar w_n/(1-t_n^\sigma)$. This follows directly from analogous arguments.

%
%
%
\end{proof}

\subsection{Proofs of Section \ref{Sec:SN:cv}}

\begin{proof}[Proof of Theorem \ref{thm:SNresult2S}]
Define $J^{*}_{n} \equiv  \{ j\in[p] : \sqrt{n} \Ep[m_j(W, \theta^*)]/\sigma_{\theta^{*},j} \geq -c^{SN}(p,\gamma_n)\}$. Consider the following derivation.
\begin{align*}
	&\P \Big(  T_n(\bar h) > {\displaystyle \max_{\theta\in \widehat\Theta_n^{SN}}} c^{SN,2N}(\theta,\alpha) \Big)
	\leq \P (  T_n(\bar h) > c^{SN,2N}(\theta^{*},\alpha) ) + \P(\theta^{*} \not\in \widehat\Theta_n^{SN}) \\
	& \leq \P (  \max_{j\in[p]}\sqrt{n}\bar m_{\theta^{*},j}/\hat\sigma_{\theta^{*},j} > c^{SN,2N}(\theta^{*},\alpha) ) + \P(\theta^{*} \not\in \widehat\Theta_n^{SN}) \\
		& \leq \P \Big(  \max_{j\in J^{*}_{n}}\sqrt{n}\bar m_{\theta^{*},j}/\hat\sigma_{\theta^{*},j} > c^{SN,2N}(\theta^{*},\alpha) \Big) + \P( \exists j \not\in J^{*}_{n} : \bar m_{\theta^{*},j} > 0  ) \\
	& ~~~+ \P(\theta^{*} \in \widehat\Theta_n^{SN}) .
\end{align*}
The proof is completed by providing suitable upper bounds on each terms on the right hand side.

For the first term, we use an argument based on Steps 2 and 3 in the proof of Theorem 4.2 in \citet{chernozhukov2013testing}.
\begin{align*}
	&\P \Big(  \max_{j\in J^{*}_{n}}\sqrt{n}\bar m_{\theta^{*},j}/\hat\sigma_{\theta^{*},j} > c^{SN,2N}(\theta^{*},\alpha) \Big)  \\
	&\leq \P (   \max_{j\in J^{*}_{n}}\sqrt{n}\bar m_{\theta^{*},j}/\hat\sigma_{\theta^{*},j} > c^{SN}(|J^{SN}|,\alpha-3\gamma_n) ) + \P(|\hat J_n^{SN}(\theta^{*})|<|J^{*}_{n}|)\\
	&\leq \alpha -3\gamma_n + \gamma_n + Cn^{-c}.
\end{align*}
For the second term, we use an argument based on Step 1 of the proof in Theorem 4.2 in \citet{chernozhukov2013testing}.
\begin{align*}
     \P( \exists j \not\in J^{*}_{n} : \bar m_{\theta^{*},j} > 0  )& \leq\P( \max_{j\in[p]}\sqrt{n} (\bar m_{\theta^{*},j}  -\Ep[m_j(W,\theta^{*})] )/\sigma_{\theta^{*},j} >  c^{SN}(p,\gamma_n)  )\\
    &  \leq \gamma_n + Cn^{-c}.
\end{align*}
For the third term, we note that the conditions we are imposing imply those in Theorem \ref{thm:SNresult}. Thus, consider the following derivation based on that result.
\begin{align*}
	\P(\theta^{*} \not\in \widehat\Theta_n^{SN}) &~=~ \P( \max_{j\in [p]} \sqrt{n}\bar m_{\theta^{*},j}/\hat \sigma_{\theta^{*},j} > c^{SN}(p,\gamma_n) ) \\
	&~\leq~ \P(  T_n(\eeta)  > c^{SN}(p,\gamma_n) )
	\leq \gamma_n + Cn^{-c}.
\end{align*}
The desired then result follows from combining the upper bounds.
 \end{proof}

\section{Proofs of Section \ref{SEC:ANTICONCENTRATION}}

\begin{proof}[Proof of Proposition \ref{lemma:AnticoncentrationNormal}]
By independence, $F_{Z}(t) \equiv \P(Z\leq t)=1-(1-\Phi ^{p}(t))^{N}$, and so
\begin{equation}
f_{Z}(t)=Np(1-\Phi ^{p}(t))^{N-1}\Phi ^{p-1}(t)\phi (t). \label{eq:DensityClosedForm}
\end{equation}

\underline{Upper bound.} Let $t^{\ast }\in \arg \max_{t\in \mathbb{R} }f_{Z}(t)$. It suffices to show that
\begin{equation}
f_{Z}(t^{\ast })=Np(1-\Phi ^{p}(t^{\ast }))^{N-1}\Phi ^{p-1}(t^{\ast })\phi (t^{\ast })\leq 5\sqrt{2}\ln ^{3/2}( Np) . \label{eq:UBinProof}
\end{equation}
We now divide the argument into cases.

Case 1: $\vert t^{\ast }\vert \geq \sqrt{2}\ln ^{1/2}( Np/ [ 2\pi \ln ^{3/2}( Np) ] ) $. \eqref{eq:DensityClosedForm} then implies $f_{Z}(t)\leq Np\phi (t)$ $\forall t\in \mathbb{R} $. Then, $\vert t^{\ast }\vert \geq \sqrt{2}\ln ^{1/2}( Np/ [ 2\pi \ln ^{3/2}( Np) ] ) $ implies that $ f_{Z}(t^{\ast })\leq Np\phi (t^{\ast })\leq \ln ^{3/2}( Np) $. Then, \eqref{eq:UBinProof} follows.

Case 2: $\vert t^{\ast }\vert <\sqrt{2}\ln ^{1/2}( Np/[ 2\pi \ln ^{3/2}( Np) ] ) $ and $t^{\ast }\leq 1/3$. Then, consider the following derivation.
\[
\begin{array}{rl}
f_{Z}(t^{\ast }) & \leq Np\Phi ^{p-1}(t^{\ast })\phi (0)\leq Np\Phi ^{p-1}(1/3)\phi (0)\\
& \leq Np(2/3)^{p-1}=(3/2)^{1+\ln ( Np) /\ln ( 3/2) -p}\leq 3/2,
\end{array}\]
where the first inequality follows from \eqref{eq:DensityClosedForm}, $ (1-\Phi ^{p}(t^{\ast }))^{N-1}\leq 1$, and $\phi (t^{\ast })\leq \phi (0)$, the second inequality follows from $t^{\ast }\leq 1/3$, and the third inequality follows from $p\ln ( 3/2) -\ln ( Np) >0$ which, in turn, follows from $\ln ( Np) /p\leq 1/\sqrt{2\pi }<\ln ( 3/2) $. Then, \eqref{eq:UBinProof} follows.

Case 3: $t^{\ast }\in ( 1/3,\sqrt{2}\ln ^{1/2}( Np/[ 2\pi \ln ^{3/2}( Np) ] ) ) $ and $\Phi ^{p}(t^{\ast })\leq 1/( Np) $. Then,
\[
f_{Z}(t^{\ast })\leq \phi (t^{\ast })/\Phi (t^{\ast })\leq \phi (0)/\Phi (1/3)\leq 2/3,
\]
where the first inequality follows from \eqref{eq:DensityClosedForm}, $ (1-\Phi ^{p}(t^{\ast }))^{N-1}\leq 1$, and $\Phi ^{p}(t^{\ast })<1/( Np) $, and the second inequality follows form $\phi (t^{\ast })\leq \phi (0)$ and $t^{\ast }>1/3$. Then, \eqref{eq:UBinProof} follows.

Case 4: $t^{\ast }\in ( 1/3,\sqrt{2}\ln ^{1/2}( Np/[ 2\pi \ln ^{3/2}( Np) ] ) ) $, $\Phi ^{p}(t^{\ast })>1/( Np) $, and $( Np-1) \phi ( t^{\ast }) \leq t^{\ast }\Phi (t^{\ast })$. Then,
\begin{equation}
\phi ( t^{\ast }) \leq t^{\ast }\Phi (t^{\ast })/( Np-1) \leq \sqrt{2}\ln ^{1/2}( Np/[ 2\pi \ln ^{3/2}( Np) ] ) /( Np-1) , \label{eq:Case4_intermediate}
\end{equation}
where the first inequality uses that $( Np-1) \phi ( t^{\ast }) \leq t^{\ast }\Phi (t^{\ast })$ and the second inequality uses that $t^{\ast }\leq \sqrt{2}\ln ^{1/2}( Np/[ 2\pi \ln ^{3/2}( Np) ] ) $. Then, consider the following derivation.
\[
\begin{array}{rl}
f_{Z}(t^{\ast })& \leq Np\phi (t^{\ast })\leq \sqrt{2}\ln ^{1/2}( Np/[ 2\pi \ln ^{3/2}( Np) ] ) {Np}/{( Np-1) }\\
&  \leq 2\sqrt{2}\ln ^{1/2}( Np) .
\end{array}
\]
where the first inequality follows from \eqref{eq:DensityClosedForm}, $ (1-\Phi ^{p}(t^{\ast }))^{N-1}\leq 1$, and $\Phi ^{p}(t^{\ast })<1/( Np) $, the second inequality follows from \eqref{eq:Case4_intermediate}, and the third inequality follows from $1\leq 2\pi \ln ^{3/2}( Np) $ and $Np/( Np-1) \leq 2$, which holds since $\ln ( Np) \geq 2$. Then, \eqref{eq:UBinProof} follows.

Case 5: $t^{\ast }\in ( 1/3,\sqrt{2}\ln ^{1/2}( Np/[ 2\pi \ln ^{3/2}( Np) ] ) ) $, $\Phi ^{p}(t^{\ast })>1/( Np) $, and $( Np-1) \phi ( t^{\ast }) >t^{\ast }\Phi (t^{\ast })$. Since $\ln f_{Z}(t)$ is a twice continuously differentiable function, any $t^{\ast }\in \arg \max_{t\in \mathbb{R}}f_{Z}(t)=\arg \max_{t\in \mathbb{R}}\ln f_{Z}(t)$ satisfies the first order condition: $ \partial \ln f_{Z}(t^{\ast })/\partial t=0$. This condition yields the following derivation.
\begin{equation}
\Phi ^{p}(t^{\ast })=\frac{( p-1) \phi ( t^{\ast }) -t^{\ast }\Phi ( t^{\ast }) }{( Np-1) \phi ( t^{\ast }) -t^{\ast }\Phi (t^{\ast })}\leq \frac{p-1}{Np-1}\leq \frac{1 }{N}, \label{eq:FOC_consequence}
\end{equation}
where the first inequality follows from first order condition, $( Np-1) \phi ( t^{\ast }) >t^{\ast }\Phi (t^{\ast })$, $Np>1$ , and $t^{\ast }>1/3>0$, and the remaining relationships are elementary.

Then, consider the following derivation.
\begin{align*}
f_{Z}(t^{\ast }) &\leq p\phi (t^{\ast })/\Phi (1/3) \\
&\leq ( 4/5) p( t^{\ast }+\sqrt{( t^{\ast }) ^{2}+4}) ( 1-\Phi (t^{\ast })) \\
&\leq ( 4/5) p( t^{\ast }+\sqrt{( t^{\ast }) ^{2}+4}) ( 1-1/( Np) ^{1/p}) \\
&=( 4/5) p( t^{\ast }+\sqrt{( t^{\ast }) ^{2}+4} ) ( 1-1/\exp ( \ln (Np)/p ) ) \\
&\leq ( 4/5) p( t^{\ast }+\sqrt{( t^{\ast }) ^{2}+4}) ( 1-1/[ 1+2 \ln(Np)/p ] ) \\
&\leq ( 8/5) ( t^{\ast }+\sqrt{( t^{\ast }) ^{2}+4}) \ln ( Np) \\
&\leq ( 16/5) ( t^{\ast }+1) \ln ( Np) \\
&\leq 5\sqrt{2}\ln ^{3/2}( Np) ,
\end{align*}
where the first inequality follows from \eqref{eq:FOC_consequence}, $ (1-\Phi ^{p}(t^{\ast }))^{N-1}$, $t^{\ast }>1/3$, and $1/\Phi (1/3)\leq 8/5$, the second inequality follows from $ 2\phi ( t) /( t+\sqrt{t^{2}+4}) \leq ( 1-\Phi ( t) ) $ for all $t\in \mathbb{R} $ (e.g., \cite{abramowitz/stegun:1964}), the third inequality follows from $\Phi ^{p}(t^{\ast })>1/( Np) $, the fourth inequality follows from the fact that $\exp ( x) \leq 1+2x$ for all $x\in [ 0,1/\sqrt{ 2\pi }] $, and sixth inequality follows from the fact that $ ( t^{\ast }+\sqrt{( t^{\ast }) ^{2}+4}) \leq 2( t^{\ast }+1) $, the seventh inequality follows from $t^{\ast }\in ( 1/3,\sqrt{2}\ln ^{1/2}( Np/[ 2\pi \ln ^{3/2}( Np) ] ) ) $ and the fact that
\begin{equation}
( 16/5) ( \sqrt{2}\ln ^{1/2}( Np/[ 2\pi \ln ^{3/2}( Np) ] ) +1) \leq 5\sqrt{2}\ln ^{1/2}( Np) , \label{eq:Case5_LastStep}
\end{equation}
To conclude the argument, it suffices to show \eqref{eq:Case5_LastStep}. Let $D\equiv [ 2\pi \ln ^{3/2}( Np) ] \geq 2^{5/2}\pi $, where the inequality holds by $\ln ( Np) \geq 2$. Because the relation $\sqrt{2}\ln ^{1/2}( Np/[ 2\pi \ln ^{3/2}( Np) ] ) \geq 1/3>0$ holds, \eqref{eq:Case5_LastStep} is equivalent to
\begin{equation}
( 2^{8}-5^{4}) \ln ( Np) +2^{8}( 2( \ln ( Np) -\ln D) ) ^{1/2}+2^{7}-2^{8}\ln D\leq 0. \label{eq:Case5_LastStep2}
\end{equation}
On the one hand, $2^{7}-2^{8}\ln D\leq 2^{7}-2^{8}\ln ( 2^{5/2}\pi ) \leq 0$. On the other hand, $2^{8}( 2( \ln ( Np) -\ln D) ) ^{1/2}\leq 2^{8}( 2( \ln ( Np) ) ) ^{1/2}\leq 2^{17/2}\ln ( Np) $, and so $ ( 2^{8}-5^{4}) \ln ( Np) +2^{8}( 2( \ln ( Np) -\ln D) ) ^{1/2}\leq ( 2^{8}+2^{17/2}-5^{4}) \ln ( Np) \leq 0$. By combining these inequalities, \eqref{eq:Case5_LastStep2} follows.

\underline{Lower bound.} Let $\bar{t}$ be (uniquely) defined by $\Phi ( \bar{t}) =( 1/N) ^{1/p}$. Note that
\begin{equation}\label{eq:LB_defn_implication}
\begin{array}{rl}
\Phi ( \bar{t}) & =( 1/N) ^{1/p}=\exp ( - (\ln N)/p) \geq \exp ( -(\ln( Np))/{p}) \\
& >\exp ( -1/(\sqrt{2\pi })) >0.5,
\end{array}\end{equation}
where the second inequality follows from $(\ln ( Np)) /p<1/\sqrt{ 2\pi }$, and the remaining relationships are elementary. \eqref{eq:LB_defn_implication} implies that $\bar{t}>0$.

As an intermediate step, we provide an alternative lower bound for $\bar{ t}$. First, note that $\bar{t}>0$ implies that:
\begin{equation}
\bar{t}\geq \sqrt{2\ln ( \bar{t}+\sqrt{\bar{t}^{2}+4}) +( \bar{t}) ^{2}}-2. \label{eq:LB_derivation_1}
\end{equation}
Second, consider the following derivation.
\begin{equation}\label{eq:LB_derivation2}
\begin{array}{rl}
\frac{2\exp ( -( \bar{t}) ^{2}/2) }{\sqrt{2\pi }( \bar{t}+\sqrt{\bar{t}^{2}+4}) }& =\frac{2\phi ( \bar{t}) }{ ( \bar{t}+\sqrt{\bar{t}^{2}+4}) }\\ & \leq 1-\Phi ( \bar{t} ) =1-\exp ( -({\ln N})/{p}) \leq 2({\ln N})/{p},
\end{array}
\end{equation}
where the first equality holds by the definition of $\phi ( t) $, the first inequality follows from the fact that $2\phi ( t) /( t+\sqrt{t^{2}+4}) \leq ( 1-\Phi ( t) ) $ for all $t\in \mathbb{R} $ (e.g., \cite{abramowitz/stegun:1964}), the second equality holds by the definition of $\bar{t}$, and the second inequality follows from the fact that $1-\exp ( -x) \leq 2x$ for all $x\geq 0$. \eqref{eq:LB_derivation2} then implies that:
\begin{equation}
\sqrt{2\ln ( \bar{t}+\sqrt{\bar{t}^{2}+4}) +( \bar{t}) ^{2}}\geq \sqrt{2\ln ( p/(\sqrt{2\pi }( \ln N)) ) }. \label{eq:LB_derivation3}
\end{equation}
By combining \eqref{eq:LB_derivation_1} and \eqref{eq:LB_derivation3}, we conclude that
\begin{equation}
\bar{t}\geq \sqrt{2\ln ( {p/(\sqrt{2\pi }( \ln N))}) }-2. \label{eq:LB_derivation4}
\end{equation}

The lower bound is a consequence of the following derivation.
\begin{eqnarray*}
\max_{t\in \mathbb{R}}f_{Z}(t) &\geq &f_{Z}( \bar{t}) =p(1-1/N)^{N-1}N^{1/p}\phi (\bar{t}) \\
&\geq &p\exp ( -1) \bar{t}( N^{1/p}-1) \\
&=&p\exp ( -1) \bar{t}( \exp ( \frac{1}{p}\ln N) -1) \\
&\geq &\exp ( -1) \bar{t}\ln N \\
&\geq &\exp ( -1) \left( \sqrt{2\ln ( {p/(\sqrt{2\pi }( \ln N))}) }-2\right) \ln N,
\end{eqnarray*}
where the first equality uses \eqref{eq:DensityClosedForm} and $\Phi ( \bar{t}) =( 1/N) ^{1/p}$, the second inequality follows from $\Phi ( \bar{t}) =( 1/N) ^{1/p}$ and the fact that $(1-1/N)^{N-1}\geq \exp ( -1) $, $\phi (t)\geq t( 1-\Phi ( t) ) $ for all $t\in \mathbb{R} $ (e.g., \cite{abramowitz/stegun:1964}), the third inequality follows from the fact that $\exp ( x) -1\geq x$ for all $x\geq 0$, and the fourth inequality follows from \eqref{eq:LB_derivation4}.
\end{proof}

\begin{proof}[Proof of Theorem \ref{thm:CondAntiConcentration}]
To show (i) we proceed similarly as in the proof of Theorem \ref{thm:MT:inferenceSubvector} and obtain
\begin{align*}
	&\P( T_n(\eeta) \geq t ) \\
	& \leq_{(3)} \P( \inf_{\theta \in \Theta_I(\eeta)}\max_{j\in \Psi_\theta} v_{\theta,j} +\sqrt{n}\Ep[m_{j}(W,\theta)]/\sigma_{\theta,j} \geq t - \delta_3)+3\gamma_n\\
& \leq_{(3')} \P( \inf_{\theta \in \Theta_I(\eeta)}\max_{j\in \Psi_\theta} G_{\theta,j} +\sqrt{n}\Ep[m_{j}(W,\theta)]/\sigma_{\theta,j} \geq t - \delta_3')+ r_3'\\
& = \P( \inf_{\theta \in \Theta_I(\eeta)}\max_{j\in \Psi_\theta} \tilde G_{\theta,j} +\sqrt{n}\Ep[m_{j}(W,\theta)]/\sigma_{\theta,j} \geq t - \delta_3' \mid (W_i)_{i=1}^n )+ r_3'\\
\end{align*}
where (3) holds by the proof of Theorem \ref{thm:MSB:inferenceSubvector}, (3') by Theorem \ref{thm:clt:minmax} (with $\eta = n^{-1/2}$)  with $\delta_3' = \delta_3 + \delta_{n,\eta,\gamma_n} $  and $r_3' = 4\gamma_n + Cn^{-1}$. The last step holds by Theorem \ref{thm:clt:minmax:Gaussian} which asserts the existence of a Gaussian process with the same distribution conditional on the data.

Next we condition on the event $E' = E \cap E_1\cap E_2$, with $E$ as defined in (\ref{eq:defE}), $E_1 = \{ \sup_{\theta\in\Theta(\eeta),j\in[p]}|v_{\theta,j}|\leq \bar w\}$ and $E_2 =  \{\sup_{\theta\in\Theta(\eeta),j\in[p]}|\hat\sigma_{\theta,j}/\sigma_{\theta,j}-1|\leq t_n^\sigma\}$ which occurs with probability $1-3\gamma_n-n^{-1}$. Under this event the covariance matrix of the bootstrap process induced by $(v^*_{\theta,j})$ is close to the process induced by $(\tilde G_{\theta,j})$ in the sense of (\ref{bound:CovInfty}). Therefore we have that with probability $1-3\gamma_n-n^{-1}$
\begin{align*}
	&\P( T_n(\eeta) \geq t ) \\
	& \leq_{(4')} \P( \inf_{\theta \in \Theta_I(\eeta)}\max_{j\in \Psi_\theta} v^*_{\theta,j} +\sqrt{n}\Ep[m_{j}(W,\theta)]/\sigma_{\theta,j} \geq t - C\delta_4 \mid (W_i)_{i=1}^n )+ Cr_4\\
& \leq_{(5')} \P( \inf_{\theta \in \widehat\Theta_n}\max_{j\in \hat \Psi_\theta} \hat v_{\theta,j}^* \geq t - C\delta_8 \mid (W_i)_{i=1}^n )+ Cr_8
\end{align*}
where (4') follows from Theorem \ref{thm:clt:minmax:Gaussian} since both $(v^*_{\theta,j})$ and  $(\tilde G_{\theta,j})$ are Gaussian processes ($\delta_4$ and $r_4$ are defined in the proof of Theorem \ref{thm:MSB:inferenceSubvector}), and (5') follows by the similar arguments as in the inequalities (5)-(8) of of the proof of Theorem \ref{thm:MSB:inferenceSubvector}.

To show (ii) we start from the proof of Theorem \ref{thm:PB:inferenceSubvector} which yields the first inequality below
\begin{align*}
	&\P( T_n(\eeta) \geq t ) \\
	& \leq_{(5)} \P( \inf_{\theta \in \Theta(\eeta)}\max_{j\in [p]} \hat v_{\theta,j}^* +\sqrt{n}\kappa_n^{-1}\bar m_{\theta,j}/\hat\sigma_{\theta,j} \geq t - \delta_5')+r_5'\\
& \leq_{(6)} \P( \inf_{\theta \in \Theta(\eeta)}\max_{j\in [p]} \hat v_{\theta,j}^* +\sqrt{n}\kappa_n^{-1}\Ep[m_j(W,\theta)]/\sigma_{\theta,j} \geq t - \delta_6')+r_6'\\
& \leq_{(7)} \P( \inf_{\theta \in \Theta(\eeta)}\max_{j\in [p]} G_{\theta,j} +\sqrt{n}\kappa_n^{-1}\Ep[m_j(W,\theta)]/\sigma_{\theta,j} \geq t - \delta_7')+r_7'\\
& = \P( \inf_{\theta \in \Theta(\eeta)}\max_{j\in [p]} \tilde G_{\theta,j} +\sqrt{n}\kappa_n^{-1}\Ep[m_j(W,\theta)]/\sigma_{\theta,j} \geq t - \delta_7'\mid (W_i)_{i=1}^n)+r_7'
\end{align*}
where (6) holds with $\delta_6' = \delta_5'+ \kappa_n^{-1}\bar w_n (1+t_n^\sigma)/(1-t_n^\sigma)$ and $r_6' = r_5'+\gamma_n$, (7) holds by  Theorem \ref{thm:clt:minmax:Gaussian}. Again, the last step holds by Theorem \ref{thm:clt:minmax:Gaussian} which asserts the existence of a Gaussian process with the same distribution conditional on the data.

Next we condition on the event $E' = E \cap E_1\cap E_2$, with $E$ as defined in (\ref{eq:defE}), $E_1 = \{ \sup_{\theta\in\Theta(\eeta),j\in[p]}|v_{\theta,j}|\leq \bar w\}$ and $E_2 =  \{\sup_{\theta\in\Theta(\eeta),j\in[p]}|\hat\sigma_{\theta,j}/\sigma_{\theta,j}-1|\leq t_n^\sigma\}$ which occurs with probability $1-3\gamma_n-n^{-1}$. Thus with probability $1-3\gamma_n-n^{-1}$ we have
\begin{align*}
	&\P( T_n(\eeta) \geq t ) \\
	& \leq_{(8)} \P( \inf_{\theta \in \Theta(\eeta)}\max_{j\in [p]} \hat v_{\theta,j}^* +\sqrt{n}\kappa_n^{-1}\Ep[m_j(W,\theta)]/\sigma_{\theta,j} \geq t - \delta_8' \mid (W_i)_{i=1}^n)+r_8'\\
	& \leq_{(9)} \P( \inf_{\theta \in \Theta(\eeta)}\max_{j\in [p]} \hat v_{\theta,j}^* +\sqrt{n}\kappa_n^{-1}\bar m_{\theta,j}/\hat\sigma_{\theta,j} \geq t - \delta_9' \mid (W_i)_{i=1}^n)+r_9'\\
\end{align*}
where (8) holds by Theorem \ref{thm:clt:minmax:Gaussian} since the associated covariance matrices satisfy (\ref{bound:CovInfty}) under the event $E$, and (9) holds by events $E_1\cap E_2$ with $\delta_9' = \delta_8'+ \kappa_n^{-1}\bar w_n(1+t_n^\sigma)/(1-t_n^\sigma) + t_n^\sigma \bar w_n$.
\end{proof}

\begin{proof}[Proof of Theorem \ref{thm:ANTI-consistency}]
For every $\epsilon\geq \delta_n$ we have with probability $1-\gamma_n$
\begin{equation}\label{def:CouplingANTI}
 \begin{array}{rl}
 \P( |R_n^* - c_n(\eeta,\alpha) | \leq \epsilon ) & \leq \P( |R_n^* - c_n(\eeta,\alpha) | \leq \epsilon + \tilde \delta_n \mid (W_i)_{i=1}^n ) + \gamma_n \\
 & \leq (\epsilon + \tilde \delta_n)\mathcal{A}_n(W) + \gamma_n
\end{array}
\end{equation}
Therefore, for any $\epsilon > \delta_n$, with the same probability we have
$$ \frac{1}{\epsilon} \P( |R_n^* - c_n(\eeta,\alpha) | \leq \epsilon ) \leq \mathcal{A}_n(W)(1+\tilde\delta_n/\epsilon) + \gamma_n/\epsilon$$
Taking the sup over $\epsilon \geq \delta_n$ we have
$$  \mathcal{A}_n \leq \mathcal{A}_n(W) (1+\tilde\delta_n/\delta_n)+\gamma_n/\delta_n$$
\end{proof}

\begin{lemma}\label{lem:density}
Let $f$ denote the density function of $\min_{k\in[N]}\max_{j\in[p]} W_{kj}$, and $f_k$ denote the density function of $\max_{j\in[p]} W_{kj}$. Provided that $-1<{\rm corr}(W_{kj},W_{k'j'})<1$,
$$\begin{array}{rl}
f(t) & \displaystyle = \phi(t)\sum_{k=1}^N \sum_{j=1}^p \P( \max_{\ell\in[p]} W_{m\ell} \geq t, \ \forall m \neq k, \  \max_{\ell\in[p]}W_{k\ell}\leq t \mid W_{kj} = t)  \\
& \displaystyle = \sum_{k=1}^N \P( \max_{j\in[p]} W_{mj} \geq t, \forall m \mid \max_{j\in[p]}W_{kj}=t) f_k(t) \\
\end{array}$$
where $f_k(t) = \phi(t)\sum_{j=1}^p\P( W_{k\ell} \leq t \mid W_{kj}=t)$.
\end{lemma}

\section{Proofs of Section \ref{SEC:PROCESSES}}

\begin{proof}[Proof of Theorem \ref{thm:clt:minmax}]
The proof proceeds in steps.

{\it Step 1.} (Main Step.) To show the result we will (suitably) discretize the sets $\Lambda$ and each $\mathcal{F}_\theta$ and set $N = |\widehat{\Lambda}|$ and $p = \max_{\theta \in \widehat{\Lambda} } |\widehat{\mathcal{F}}_\theta|$. For $\varepsilon>0$, define $T^\varepsilon = \min_{\theta\in \widehat{\Lambda}} \max_{f \in \widehat{\mathcal{F}}_\theta} B(f)+\Gn(f)$ and $\widetilde T^\varepsilon = \min_{\theta\in \widehat{\Lambda}} \max_{f \in \widehat{\mathcal{F}}_\theta} B(f)+G_P(f)$ where the sets $\widehat{\Lambda}$ and $\widehat{\mathcal{F}}_\theta$ are such that $|T-T^\varepsilon| \leq \varepsilon$ and $ |\widetilde T - \widetilde T^\varepsilon| \leq \bar \varepsilon$ with probability exceeding $1-\rho(\varepsilon)$.

By Step 2 equation (\ref{res:Step2}) below we can take $\bar \varepsilon = \sigma/(bn^{1/2})+\eta_{\mathcal{F},\gamma_n}+\eta+ C\sqrt{\sigma^2 K_n/n} + CbK_n/(\gamma_n^{1/q}n^{1/2-1/q})$ and $\rho(\varepsilon)=2n^{-1}+3\gamma_n$. Therefore, we have
with probability exceeding $1-r_{1n}(\delta)-C\{\gamma_n+n^{-1}\}$,
$$
\begin{array}{rl}
|T-\widetilde T| & \leq |T-T^\varepsilon|+|T^\varepsilon - \widetilde T^\varepsilon| + |\widetilde T^\varepsilon - \widetilde T| \\
& \leq |T-T^\varepsilon|+ |\widetilde T^\varepsilon - \widetilde T| + C\delta\\
& \leq 2\sigma/(bn^{1/2}) +2\eta_{\mathcal{F},\gamma_n} + 2\eta \\
& + C\sqrt{\sigma^2 K_n/n} + CbK_n/(\gamma_n^{1/q}n^{1/2-1/q}) + C\delta,
\end{array}
$$
 where the first line holds by the triangle inequality, the second holds by Step 3 with probability exceeding $1-r_{1n}(\delta)$, the third inequality holds by Step 2 as stated.

To establish the result recall that $\log(Np) \leq C K_n$ (by Step 2) so that setting $$\delta = C' \left\{ \frac{(b\sigma^2 K_n^2)^{1/3}}{\gamma_n^{1/3}n^{1/6}} + \frac{bK_n}{\gamma_n^{1/q} n^{1/2-1/q}}\right\}$$
yields $r_{1n}(\delta) \leq C \gamma_n$ and the result follows.

{\it Step 2.} (Controlling Discretization Error). To discretize the empirical and Gaussian processes we set $\varepsilon = \sigma/(bn^{1/2})$ and $p = 2 \max_{\theta \in \Lambda} N(\mathcal{F}_\theta,d,\varepsilon b) \cdot N_B(\eta)$.
 Note that $\mathcal{F}$ is VC type with constants $\bar A\geq e$ and $v>1$ so that $N(\mathcal{F}_\theta,d,\varepsilon b) \leq N(\mathcal{F},d,\varepsilon b) \leq (4\bar A/\varepsilon)^v$. Moreover, we will choose a $\varepsilon_\Lambda$-cover $\widehat \Lambda$ of $\Lambda$, $\widehat \Lambda \subset \Lambda$, with cardinality bounded by $ N = (12n/\varepsilon_\Lambda)^{d_\theta}$ with $\varepsilon_\Lambda=
 \{\sigma/(C_{\mathcal{F},\gamma_n} bn^{1/2})\}^{1/\bar \chi}$.

Let $K_n = \log N_B(\eta) + v(\log n \vee \log( \bar Ab/\sigma) ) + (d_\theta/\bar \chi)\log( nC_{\mathcal{F},\gamma_n}b/\sigma )$ and $\widetilde K_n =  v\{\log n \vee \log( \bar Ab/\sigma) \}$. Note that $\widetilde K_n \leq K_n$ and  $\log (Np) \leq C K_n$.

Letting $\mathcal{F}^{\varepsilon} = \{ f - g : f,g \in \mathcal{F}, d(f,g) <\varepsilon b\}$,  $\psi(\theta):=\sup_{f\in \mathcal{F}_\theta} (B(f) + \Gn(f) )$ and $\hat\psi(\theta):=\sup_{f\in \widehat{\mathcal{F}}_\theta} (B(f) + \Gn(f) )$  we have that
$$ \begin{array}{rl}
T - T^\varepsilon & = \displaystyle \inf_{\theta \in \Lambda} \psi(\theta) - \min_{\theta \in \widehat\Lambda} \hat \psi(\theta) \\
& \leq  \inf_{\theta \in \widehat \Lambda} \psi(\theta) - \min_{\theta \in \widehat\Lambda} \hat \psi(\theta) \\
& \displaystyle \leq  \max_{\theta \in \widehat \Lambda} \left| \psi(\theta)- \hat \psi(\theta)\right| \\
& \leq \eta + \sup_{ h \in \mathcal{F}^\varepsilon} |\Gn(h)|.
\end{array}
$$
Moreover, we have
$$ \begin{array}{ll}
T - T^\varepsilon & = \displaystyle \inf_{\theta \in \Lambda} \psi(\theta) - \min_{\theta \in \widehat\Lambda} \hat \psi(\theta) \\
& \displaystyle \geq  \inf_{\theta \in \Lambda} \psi(\theta) - \min_{\theta \in \widehat\Lambda} \psi(\theta) -\eta -  \sup_{ h \in \mathcal{F}^\varepsilon} |\Gn(h)| \\
&\displaystyle \geq  -\sup_{\theta,\tilde \theta \in \Lambda, \|\theta-\tilde\theta\|\leq \varepsilon_\Lambda} \left| \psi(\theta) - \psi(\tilde \theta)\right|  -\eta -  \sup_{ h \in \mathcal{F}^\varepsilon} |\Gn(h)| ,
\end{array}
$$
where we used that $\widehat \Lambda$ is an $\varepsilon_\Lambda$-cover of $\Lambda$.
Similarly we have
{\small $$ \begin{array}{ll}
\widetilde T - \widetilde T^\varepsilon & = \displaystyle \inf_{\theta \in \Lambda} \sup_{f\in \mathcal{F}_\theta} (B(f) + G_P(f) ) - \min_{\theta \in \widehat\Lambda} \max_{f\in \widehat{\mathcal{F}}_\theta} (B(f) + G_P(f) )\\
&\displaystyle   \leq \eta +  \sup_{ h \in \mathcal{F}^\varepsilon} |G_P(h)|\\
\widetilde T - \widetilde T^\varepsilon & =\displaystyle \inf_{\theta \in \Lambda} \sup_{f\in \mathcal{F}_\theta} (B(f) + G_P(f) ) - \min_{\theta \in \widehat\Lambda} \max_{f\in \widehat{\mathcal{F}}_\theta} (B(f) + G_P(f) ) \\
& \displaystyle \geq   - \sup_{\theta,\tilde \theta \in \Lambda, \|\theta-\tilde\theta\|\leq \varepsilon_\Lambda} \left| \sup_{f\in \mathcal{F}_\theta} (B(f) + G_P(f) ) - \sup_{f\in {\mathcal{F}}_{\tilde \theta}} (B(f) + G_P(f) )\right|\\
&  -\eta -  \sup_{ h \in \mathcal{F}^\varepsilon} |G_P(h)| .
\end{array}
$$}

By Steps 2 and 3 in the proof of Theorem 2.1 of \citet{chernozhukov2015noncenteredprocesses}, we have
$$\begin{array}{c}
 \P( \sup_{ h \in \mathcal{F}^\varepsilon} |G_P(h)| > C\sqrt{\sigma^2 \widetilde K_n/n} ) \leq 2n^{-1} \ \ \mbox{and} \\ \P( \sup_{ h \in \mathcal{F}^\varepsilon} |\Gn(h)| > Cb\widetilde K_n/(\gamma_n^{1/q}n^{1/2-1/q}) ) \leq \gamma_n. \end{array}$$

Note that by Condition B we have
$$ P\left( \sup_{\theta,\tilde \theta \in \Lambda, \|\theta-\tilde\theta\|\leq \varepsilon_\Lambda} \left| \psi(\theta) -   \psi(\tilde \theta)\right|  > C_{\mathcal{F},\gamma_n} \varepsilon_\Lambda^{\alpha} + \eta_{\mathcal{F},\gamma_n} \right)\leq \gamma_n,$$
where by definition of $\varepsilon_\Lambda$ we have  $C_{\mathcal{F},\gamma_n} \varepsilon_\Lambda^{\bar \chi} \leq \sigma/(bn^{1/2})$.

Similarly, again by Condition B, we have
{\small $$\P\left(  \sup_{\theta,\tilde \theta \in \Lambda, \|\theta-\tilde\theta\|\leq \varepsilon_\Lambda} \left| \sup_{f\in \mathcal{F}_\theta} B(f) + G_P(f)  -   \sup_{f\in \mathcal{F}_{\tilde \theta}} B(f) + G_P(f) \right| > \frac{\sigma}{bn^{1/2}} + \eta_{\mathcal{F},\gamma_n} \right) \leq \gamma_n.$$}
Therefore, we have
{\small \begin{equation}\label{res:Step2} \P\left(|T-T^\varepsilon|\vee |\widetilde T - \widetilde T^\varepsilon | >\frac{\sigma}{bn^{1/2}}+\eta_{\mathcal{F},\gamma_n}+\eta+\frac{C\sigma K_n^{1/2}}{n^{1/2}}+\frac{CbK_n}{\gamma_n^{1/q}n^{1/2-1/q}}\right)\leq \frac{2}{n}+3\gamma_n.\end{equation}}
{\it Step 3.} (CLT for Discretized Process). In this step we show that $$\P(|T^\varepsilon-\widetilde T^\varepsilon| >  C \{ \eta + \delta_{n,\eta,\gamma_n}\} ) \leq C'\{\gamma_n + n^{-1}\}.$$ For that we will apply Theorem \ref{thm:clt:minmax:discrete} for $T^\varepsilon = \min_{\theta \in \hat\Lambda}\max_{f\in \widehat{\mathcal{F}}_\theta} B(f) + \Gn(f)$ and $\widetilde T^\varepsilon = \min_{\theta \in \hat\Lambda}\max_{f\in \widehat{\mathcal{F}}_\theta} B(f) + G_P(f)$ where $N=|\hat\Lambda|=(12n^{\frac{3}{2}}C_{\mathcal{F},\gamma_n} b/\sigma)^{\frac{d_\theta}{\bar \chi}}$ and $p\leq (4\bar A n^{1/2}b/\sigma)^v$. This will show that for every Borel subset $A$ of $\mathbb{R}$  we have
$$ \P(T^\varepsilon \in A) \leq \P( \widetilde T^\varepsilon \in A^{C \{ \eta + \delta_{n,\eta,\gamma_n} \}} ) +  C'\{ \gamma_n + n^{-1}\},$$
and the result follows from Strassen's theorem (see e.g.\ Lemma 4.1 in \citet{chernozhukov2015noncenteredprocesses}).

Under Condition A,
by letting $f_i = X_i(f)$ and $\widetilde X_i = f_i - \Ep[f_i]$, we have
$$L_n=\max_{\theta \in \widehat{\Lambda}, f \in \widehat{\mathcal{F}}_\theta}\Ep[\mbox{$\frac{1}{n}\sum_{i=1}^n$}|\widetilde X_{if}|^3] \leq 8 \max_{\theta \in \widehat{\Lambda}, f \in \widehat{\mathcal{F}}_\theta}\Ep[\mbox{$\frac{1}{n}\sum_{i=1}^n$}|f_i|^3] \leq 8\sigma^2b,$$
{\small $$\begin{array}{rl}
M_{n,\widetilde X}(\delta) &\displaystyle = \frac{1}{n}\sum_{i=1}^n\Ep\left[ \max_{\theta \in \widehat{\Lambda}, f \in \widehat{\mathcal{F}}_\theta}|\widetilde X_{if}|^3 1\left\{ \max_{\theta \in \widehat{\Lambda}, f \in \widehat{\mathcal{F}}_\theta}|\widetilde X_{if}|> \delta \sqrt{n}/\log (Np) \right\} \right]\\
 &\displaystyle \leq \frac{1}{n}\sum_{i=1}^n\Ep[ \max_{\theta \in \widehat{\Lambda}, f \in \widehat{\mathcal{F}}_\theta}|\widetilde X_{if}|^3 \max_{\theta \in \widehat{\Lambda}, f \in \widehat{\mathcal{F}}_\theta}|\widetilde X_{if}|^{q-3}/ \{ \delta \sqrt{n}/\log (Np) \}^{q-3} ]\\
 &\displaystyle \leq \frac{2^{q-1}}{n}\sum_{i=1}^n\Ep[ \max_{\theta \in \widehat{\Lambda}, f \in \widehat{\mathcal{F}}_\theta}|f_i|^{q} ]/ \{ \delta \sqrt{n}/\log(Np)\}^{q-3}\\
 &\displaystyle \leq 2^{q-1}b^q \{\log(Np) / (\delta \sqrt{n})\}^{q-3}.
 \end{array}$$}
Next note that we can assume $\delta^3 \geq C\sigma^2 b n^{-1/2} \log^{2}(Np)$, or equivalently $\delta \geq C'\sigma n^{-1/6}\log^{2/3}(Np)$ (otherwise the result is trivial as $r_{1n}(\delta) \geq 1$).
Then, for $Y_{if} \sim N(\Ep[f_i], {\rm var}(f_i) )$,  $\widetilde Y_{if} = Y_{if} - \Ep[f_i]$,
  by Lemma 6.6 in \citet{chernozhukov2015noncenteredprocesses} we have
{\small $$\begin{array}{rl}
M_{n,\widetilde Y}(\delta) & \displaystyle= \frac{1}{n}\sum_{i=1}^n\Ep[ \max_{\theta \in \widehat{\Lambda}, f \in \widehat{\mathcal{F}}_\theta}|\widetilde Y_{if}|^3 1\{ \max_{\theta \in \widehat{\Lambda}, f \in \widehat{\mathcal{F}}_\theta}|\widetilde Y_{if}|> \delta \sqrt{n}/\log (Np) \} ]\\
 & \displaystyle \leq  12(\delta \sqrt{n}/\log (Np) +  c\sigma\sqrt{\log(Np)})^3\exp(-\delta\sqrt{n}/\{c\sigma\log^{3/2}(Np)\}) \\
 & \displaystyle \leq  C\{\delta \sqrt{n}/\log (Np)\}^3\exp(-\delta\sqrt{n}/\{c\sigma\log^{3/2}(Np)\}) \\
&  \leq C n^{-2}\sigma^2b,
 \end{array}$$}
where the second and third inequalities follow from under $C\log(Np)\leq K_n \leq n^{1/3}$ (and the lower bound on $\delta$).

Therefore, by Theorem \ref{thm:clt:minmax:discrete} and Strassen's theorem we have
$$\P( |T^\varepsilon - \widetilde T^\varepsilon| > C\delta ) \leq r_{1n}(\delta) := C'_q \frac{\log^2(Np)}{\delta^3 n^{1/2}}\left\{\sigma^2 b + \frac{\log^{q-3}(Np) b^q }{\{\delta \sqrt{n}\}^{q-3}}\right\} .$$
\end{proof}

\begin{proof}[Proof of Theorem \ref{thm:clt:minmax:Gaussian}]

Similar to the proof of Theorem \ref{thm:clt:minmax}, we proceed in steps.
By a conditional version of Strassen's theorem (see, e.g., Lemma 4.2 of \citet{chernozhukov2015noncenteredprocesses}), since $\sigma(X_i:i=1,\ldots,n)$ is countably generated, it suffices to show that there is an event $E \in \sigma(X_i:i=1,\ldots,n)$ such that $\P(E) \geq 1 - \gamma_n - n^{-1}$ and on the event $E$,
$$ \P(S \in  A \mid X_1,\ldots,X_n) \leq \P( \widetilde S \in A^{C(\eta+\bar\delta_{n,\eta,\gamma_n})}) + C(\gamma_n+ n^{-1})$$
for every Borel subset $A$ of $\mathbb{R}$. We will specify an event $E$ as the intersection of the following events:
\begin{equation}\label{eq:defE}\begin{array}{l}
(i)   \sup_{f \in \mathcal{F}}|\Gn(f)| \leq C \sigma \widetilde K_n^{1/2} / \gamma_n^{1/q} + Cb \widetilde K_n/(\gamma_n^{1/q}n^{1/2-1/q}) \\
(ii) \sup_{f,g \in \mathcal{F}}|\Gn(fg)| \leq C b\sigma \widetilde K_n^{1/2}/\gamma_n^{2/q} + C b^2 \widetilde  K_n/(\gamma_n^{2/q}n^{1/2-2/q})\\
(iii) \|F\|_{P_n,2} \leq n^{1/2}\|F\|_{P,2},
\end{array}\end{equation}
where $\widetilde K_n = v\{\log n \vee \log (\bar A b /\sigma)\}$. Step 0 in the proof of Theorem 2.2 in \citet{chernozhukov2015noncenteredprocesses} established that $\P(E) \geq 1-\gamma_n-n^{-1}$.

{\it Step 1.} (Main Step.) We discretize the sets $\Lambda$ and each $\mathcal{F}_\theta$ and set $N = |\widehat{\Lambda}|$ and $p = \max_{\theta \in \widehat{\Lambda} } |\widehat{\mathcal{F}}_\theta|$. For $\varepsilon>0$, define $S^\varepsilon = \min_{\theta\in \widehat{\Lambda}} \max_{f \in \widehat{\mathcal{F}}_\theta} B(f)+\Gn^\xi(f)$ and $\widetilde S^\varepsilon = \min_{\theta\in \widehat{\Lambda}} \max_{f \in \widehat{\mathcal{F}}_\theta} B(f)+G_P(f)$ where the sets $\widehat{\Lambda}$ and $\widehat{\mathcal{F}}_\theta$ are such that $|S-S^\varepsilon| \leq \varepsilon$ and $ |\widetilde S - \widetilde S^\varepsilon| \leq \bar \varepsilon$ with probability exceeding $1-\rho(\varepsilon)$.

By Step 2 below we can take $\bar \varepsilon = \sigma/(bn^{1/2})+\eta_{\mathcal{F},\gamma_n}+\eta+ C\sqrt{\sigma^2 K_n/n} +  (b\sigma K_n^{3/2})^{1/2}/(\gamma_n^{1/q}n^{1/4})+ CbK_n/(\gamma_n^{1/q}n^{1/2-1/q})$ and $\rho(\varepsilon)=5n^{-1}+3\gamma_n$. Then we have
$$
\begin{array}{rl}
|S-\widetilde S| & \leq |S-S^\varepsilon|+|S^\varepsilon - \widetilde S^\varepsilon| + |\widetilde S^\varepsilon - \widetilde S| \\
& \leq |S-S^\varepsilon|+ |\widetilde S^\varepsilon - \widetilde S| + C\delta\\
& \leq 2\sigma/(bn^{1/2}) +2\eta_{\mathcal{F},\gamma_n}+ 2\eta+ C\sqrt{\sigma^2 K_n/n}\\
&  + (b\sigma K_n^{3/2})^{1/2}/(\gamma_n^{1/q}n^{1/4})+ CbK_n/(\gamma_n^{1/q}n^{1/2-1/q}) + C\delta,
\end{array}
$$
 where the first line holds by the triangle inequality, the second holds by Step 3 with probability exceeding $1-\bar r_{1n}(\delta)$, the third inequality holds by Step 2 as stated before. Then, setting $ \delta = \bar\delta _{n,\eta,\gamma_n}$ the result follows by noting that $\log(Np) \leq CK_n$ by (\ref{def:Kn}) and $\bar r_{1n}(\bar\delta _{n,\eta,\gamma_n})\leq C\gamma_n$.

{\it Step 2.} (Controlling Discretization Error). Using the same notation as in the proof of Theorem \ref{thm:clt:minmax}, and defining $\psi^\xi(\theta):=\sup_{f\in \mathcal{F}_\theta} (B(f) + \Gn^\xi(f) )$ and $\hat\psi(\theta):=\sup_{f\in \widehat{\mathcal{F}}_\theta} (B(f) + \Gn^\xi(f) )$ we have
$$ \begin{array}{rl}
\displaystyle S - S^\varepsilon & \leq \eta + \sup_{ h \in \mathcal{F}^\varepsilon} |\Gn^\xi(h)| \\
\displaystyle S - S^\varepsilon &
\displaystyle  \geq  -\sup_{\theta,\tilde \theta \in \Lambda, \|\theta-\tilde\theta\|\leq \varepsilon_\Lambda} \left| \psi^\xi(\theta) - \psi^\xi(\tilde \theta)\right|  -\eta -  \sup_{ h \in \mathcal{F}^\varepsilon} |\Gn^\xi(h)|,
\end{array}
$$
where we used that $\widehat \Lambda$ is an $\varepsilon_\Lambda$-cover of $\Lambda$.

Similarly we have
$$ \begin{array}{ll}
\displaystyle \widetilde S - \widetilde S^\varepsilon & \leq \eta +  \sup_{ h \in \mathcal{F}^\varepsilon} |G_P(h)|\\
\displaystyle \widetilde S - \widetilde S^\varepsilon  &   \geq   -\eta -  \sup_{ h \in \mathcal{F}^\varepsilon} |G_P(h)| \\
& \displaystyle  -\sup_{\theta,\tilde \theta \in \Lambda, \|\theta-\tilde\theta\|\leq \varepsilon_\Lambda} \left| \sup_{f\in \mathcal{F}_\theta} (B(f) + G_P(f) ) - \sup_{f\in {\mathcal{F}}_{\tilde \theta}} (B(f) + G_P(f) )\right|.
\end{array}
$$

By Step 2 in the proof of Theorem 2.1 of \citet{chernozhukov2015noncenteredprocesses}, and Step 2 in the proof of Theorem 2.2 of \citet{chernozhukov2015noncenteredprocesses}, we have
$$\begin{array}{c}
 \P( \sup_{ h \in \mathcal{F}^\varepsilon} |G_P(h)| > C\sqrt{\sigma^2 K_n/n} ) \leq 2n^{-1} \ \ \mbox{and} \\
 \P\left( {\displaystyle \sup_{ h \in \mathcal{F}^\varepsilon}} |\Gn^\xi(h)| > C\left\{ \frac{(b\sigma K_n^{3/2})^{1/2}}{\gamma_n^{1/q}n^{1/4}}+\frac{bK_n}{\gamma_n^{1/q}n^{1/2-1/q}} \right\} \mid X_1,\ldots,X_n \right) \leq 2n^{-1}, \end{array}$$
provided that the event $E$ as defined in (\ref{eq:defE}) occurs.

Using Condition B as in Step 2 in the proof of Theorem \ref{thm:clt:minmax} we have
$$\P\left(  \sup_{\theta,\tilde \theta \in \Lambda, \|\theta-\tilde\theta\|\leq \varepsilon_\Lambda} \left| \psi^\xi(\theta) -   \psi^\xi(\tilde \theta)\right| > \sigma / (bn^{1/2}) +\eta_{\mathcal{F},\gamma_n}\right) \leq \gamma_n.$$
{\small $$\P\left(  \sup_{\substack{\theta,\tilde \theta \in \Lambda,\\ \|\theta-\tilde\theta\|\leq \varepsilon_\Lambda}} \left| \sup_{f\in \mathcal{F}_\theta} (B(f) + G_P(f) ) -   \sup_{f\in \mathcal{F}_{\tilde \theta}} (B(f) + G_P(f) )\right| > \frac{\sigma}{bn^{1/2}} +\eta_{\mathcal{F},\gamma_n} \right) \leq \gamma_n.$$}

{\it Step 3.} (CLT for Discretized Gaussian Process)
Next we will establish that in the event $E$ we have
{\small \begin{equation}\label{eq:ttt} \P( S^\varepsilon \in A\mid X_1,\ldots, X_n) \leq \P( \widetilde S^\varepsilon \in A^{5\delta}) + \bar r_{1n}(\delta)\end{equation}}
\noindent for every $\delta>0$ and Borel subset $A$ of $\mathbb{R}$ where $\log(Np) \leq CK_n$ and
 $$\bar r_{1n}(\delta):=\frac{CK_n}{\delta^2 \gamma_n^{2/q}}\left\{ \frac{b\sigma K_n^{1/2}}{n^{1/2}}+\frac{b^2K_n}{n^{1-2/q}}  \right\}. $$
We define
 $$\begin{array}{rl}
 \| \Sigma^{S^\varepsilon} - \Sigma^{\widetilde S^\varepsilon}\|_\infty  := {\displaystyle \max_{(k,j)\in[N]^2\times[p]^2}} & |  \En[f_{k_1j_1}f_{k_2j_2}]-\En[f_{k_1j_1}]\En[f_{k_2j_2}] \\
& - \{\Ep[f_{k_1j_1}f_{k_2j_2}]-\Ep[f_{k_1j_1}]\Ep[f_{k_2j_2}]\}|.
\end{array}$$
It follows that
\begin{equation}\label{eq:ControlDelta} \begin{array}{rl}
| \En[f_{k_1j_1}f_{k_2j_2}] - \Ep[f_{k_1j_1}f_{k_2j_2}] | & \leq n^{-1/2} \sup_{f,g \in \mathcal{F}}|\Gn(fg)|\\
| \En[f_{k_1j_1}]\En[f_{k_2j_2}] - \Ep[f_{k_1j_1}]\Ep[f_{k_2j_2}] | & \leq n^{-1} \sup_{f \in \mathcal{F}}|\Gn(f)|^2 \\
& + \sigma n^{-1/2} \sup_{f \in \mathcal{F}}|\Gn(f)|,
 \end{array}
 \end{equation} and conditional on $E$, we have
\begin{equation}\label{bound:CovInfty} \| \Sigma^{S^\varepsilon} - \Sigma^{\widetilde S^\varepsilon}\|_\infty \leq  \frac{Cb\sigma K_n^{1/2}}{\gamma_n^{2/q}n^{1/2}}+\frac{Cb^2K_n}{\gamma_n^{2/q}n^{1-2/q}}.\end{equation}
 Therefore, by Theorem \ref{thm:clt:minmax:Gaussian:discrete}
 we have that (\ref{eq:ttt}) holds.

\end{proof}

\section{Proofs of Section \ref{SEC:DISCRETIZED}}

\begin{proof}[Proof of Theorem \ref{thm:clt:minmax:discrete}]
The proof follows similar steps to the proof of Lemma 5.1 in \citet{chernozhukov2014clt}. For a Borel set $A \subset \mathbb{R}$, define its $\epsilon$-enlargement as $A^\epsilon=\{ t \in \mathbb{R} : {\rm dist}(t,A) \leq \epsilon\}$.  Define $\bar \mu_{kj} = \frac{1}{\sqrt{n}}\sum_{i=1}^n \Ep[X_{i,kj}]$, $\bar \mu_k = (\bar \mu_{kj})_{j=1}^p$, and for a vector $v \in \mathbb{R}^d$ we let $$F_{\beta,\bar \mu_k}(v)=\beta^{-1}\log\left(\sum_{j=1}^p\exp(\beta \{v_j+\bar \mu_{kj}\})\right).$$ For a $N\times p$ matrix $W=[W_1,\ldots,W_N]'$, let $$F_{\beta,\bar \mu}(W) = (F_{\beta,\bar \mu_1}(W_1),\ldots,F_{\beta,\bar \mu_N}(W_N))'\in \mathbb{R}^N.$$
In order to approximate the $\min\max$ operator we will consider the function $G_\beta: \mathbb{R}^{N\times p}\to \mathbb{R}$ defined as $$ G_\beta(W)=-F_{\beta,0}(-F_{\beta,\bar \mu}(W)).$$

It follows by Lemma \ref{aux:Gbeta} that
$$ -\beta^{-1}\log N \leq G_\beta(W-\bar\mu) - \min_{1\leq k\leq N}\max_{1\leq j \leq p} W_{kj} \leq \beta^{-1} \log p, \ \ \mbox{for all} \ W \in \mathbb{R}^{N\times p}. $$

We choose $\beta$ so that $\delta = \beta^{-1}\log (Np)$. By Lemma 5.1 in \citet{chernozhukov2015noncenteredprocesses}, for each Borel set $A\subset \mathbb{R}$ and $\delta>0$, there exists a function $g \in C^3$, satisfying $\|g'\|_\infty \leq \delta^{-1}$, $\|g''\|_\infty \leq \delta^{-2}K$, $\|g'''\|_\infty\leq \delta^{-3}K$ for a universal constant $K$, such that
$ 1_{A}(t) \leq g(t) \leq 1_{A^{3\delta}}(t)$ for all $t\in \mathbb{R}$.

To proceed define the composition $m=g\circ G_\beta$. Let $$\begin{array}{rl} Z & = {\displaystyle \min_{k\in [N]}\max_{j\in [p]}}\frac{1}{\sqrt{n}}\sum_{i=1}^nX_{i,kj}= {\displaystyle \min_{k\in [N]}\max_{j\in [p]}}\frac{1}{\sqrt{n}}\sum_{i=1}^n\wX_{i,kj}+\bar\mu_{kj},\\
\widetilde{Z} & = {\displaystyle \min_{k\in [N]}\max_{j\in [p]}}\frac{1}{\sqrt{n}}\sum_{i=1}^nY_{i,kj}= {\displaystyle \min_{k\in [N]}\max_{j\in [p]}}\frac{1}{\sqrt{n}}\sum_{i=1}^n\wY_{i,kj}+\bar\mu_{kj}.
\end{array}$$
Since $\delta \geq \beta^{-1}\log(Np)$, Lemma \ref{aux:Gbeta} implies that
$$ \begin{array}{rl}
\P(Z \in A ) & \leq \P( G_\beta(\wX) \in A^{\delta}) \leq \Ep[ m(\wX) ]\\
 &  = \Ep[ m(\wY) ] + \Ep[ m(\wX) ]- \Ep[ m(\wY) ]  \\
 & \leq \P( G_\beta(\wY) \in A^{4\delta}) + \Ep[ m(\wX) ]- \Ep[ m(\wY) ]  \\
 & \leq \P( \widetilde Z \in A^{5\delta}) + \Ep[ m(\wX) ]- \Ep[ m(\wY) ]  \\
  \end{array}$$
We will proceed to bound $\Ep[ m(\wX) ]- \Ep[ m(\wY) ]$.

Let $\wW$ be a copy of $\wY$ and we can assume that $\wX, \wY, \wW$ are independent. It suffices to bound $\Ep[\mathcal{I}_n]$ where for $v \in [0,1]$
$$ \mathcal{I}_n = m(\sqrt{v}S_n^{\wX}+\sqrt{1-v}S_n^{\wY})-m(S_n^{\wW}), \ \ S_n^Q = n^{-1/2}\sum_{i=1}^n Q_i, \ \ Q = \wX,\wY,\wW.$$
Our case corresponds to $v=1$. For a constant $\bar A\geq 6$ (independent of $n$), and for any $w \in \mathbb{R}^{N\times p}$ and $ t>0$, define
$$ h(w,t) \equiv 1\left\{ \min_{k\in [N]}\max_{j\in [p]} (w_{kj}+\bar\mu_{kj}) \in A^{\bar A \delta + t/\beta} \right\}. \ \ $$
For any $t\in(0,1)$, define $\omega(t)\equiv\frac{1}{\sqrt{t}\wedge \sqrt{1-t}}.$.

For any $t\in [0,1]$, Define the Slepian interpolation $Z(t) \equiv \sum_{i=1}^n Z_i(t)$, where
$$ Z_i(t) \equiv \frac{1}{\sqrt{n}}\left\{ \sqrt{t}(\sqrt{v}\wX_i + \sqrt{1-v}\wY_i)+\sqrt{1-t}\wW_i\right\} .$$
so that $Z(1) = \sqrt{v} S_n^{\wX} + \sqrt{1-v}S_n^{\wY}$ and $Z(0)=S_n^{\wW}$. It follows that
$$\mathcal{I}_n = m(Z(1))-m(Z(0)) = \int_0^1 \frac{d}{dt}m(Z(t))dt .$$

Define the Stein leave-one-out term of $Z(t)$ as $Z^{(i)}(t)\equiv Z(t)-Z_i(t)$ and
$$ \dZ_i(t)= \frac{1}{\sqrt{n}}\left\{ \frac{1}{\sqrt{t}}(\sqrt{v}\wX_i+\sqrt{1-v}\wY_i)-\frac{1}{\sqrt{1-t}}\wW_i\right\}/$$

By a Taylor expansion,
$$
\begin{array}{rl}
\Ep[\mathcal{I}_n] & = \Ep[ \int_0^1 \frac{d}{dt}m(Z(t))dt] = \frac{1}{2} \Ep[ \int_0^1 \sum_{k\in[N]}\sum_{j\in[p]} m_{(k,j)}'(Z(t))\dZ_{kj}(t)dt]\\
& =  \frac{1}{2}\sum_{(k,j)\in[N]\times[p]}\sum_{i=1}^n \int_0^1\Ep[  m_{(k,j)}'(Z(t))\dZ_{i,kj}(t)]dt\\
& = I + II + III, \\
\end{array}
$$
where
$$
\begin{array}{rl}
I & \equiv \frac{1}{2}\sum_{(k,j)\in[N]\times[p]}\sum_{i=1}^n \int_0^1\Ep[  m_{(k,j)}'(Z^{(i)}(t))\dZ_{i(kj)}(t)]dt\\
II & \equiv \frac{1}{2}\sum_{(k,j)\in[N]^2\times[p]^2}\sum_{i=1}^n \int_0^1\Ep[  m_{(k,j)}''(Z^{(i)}(t))\dZ_{i(k_1j_1)}(t)Z_{i(k_2j_2)}(t)]dt\\
III & \equiv \frac{1}{2}\sum_{(k,j)\in[N]^3\times[p]^3}\sum_{i=1}^n\int_0^1\int_0^1(1-\tau)\Ep[H_{(k,j)}(i,t,\tau)]dtd\tau
\end{array}
$$ and $H_{(k,j)}(i,t,\tau) \equiv m'''_{(kj)}(Z^{(i)}(t)+\tau Z_i(t))\dZ_{i(k_1j_1)}(t)Z_{i(k_2j_2)}(t)Z_{i(k_3j_3)}(t)$.

By independence between $Z^{(i)}(t)$ and $Z_i(t)$, for $(k,j)\in[N]\times[p]$ we have $$\Ep[  m_{(k,j)}(Z^{(i)}(t))\dZ_{i,kj}(t)] =\Ep[  m_{(k,j)}'(Z^{(i)}(t))]\Ep[\dZ_{i(k_1j_1)}(t)]=0,$$ since $\Ep[\dZ_{i(k_1j_1)}(t)]=0$ which in turn implies $I=0$. Similarly, by independence between $Z^{(i)}(t)$ and $Z_i(t)$, for $(k,j)\in [N]^2\times[p]^2$ we have
$$\begin{array}{c}\Ep[  m_{(k,j)}''(Z^{(i)}(t))\dZ_{i(k_1j_1)}(t)Z_{i(k_2j_2)}(t)]\\
=\Ep[  m_{(k,j)}''(Z^{(i)}(t))]\Ep[\dZ_{i(k_1j_1)}(t)Z_{i(k_2j_2)}(t)]\end{array}$$ and note that $\Ep[\dZ_{i(k_1j_1)}(t)Z_{i(k_2j_2)}(t)]=0$ by expanding and using independence between $\wX$ and $\wY$. Therefore, $II=0$.

To bound $III$ let $\chi_i = 1\{ \max_{k\in [N],j\in [p]} |\wX_{ikj}|\vee |\wY_{ikj}| \vee |\wW_{ikj}| \leq \sqrt{n}/(4\beta)\}$ and $III=\frac{1}{2}III_1 + \frac{1}{2}III_2$ where
$$
\begin{array}{rl}
III_1 & = \sum_{(k,j)\in[N]^3\times[p]^3}\sum_{i=1}^n\int_0^1\int_0^1(1-\tau)\Ep[ \chi_i H_{(k,j)}(i,t,\tau)]d\tau dt\\
III_2 & = \sum_{(k,j)\in[N]^3\times[p]^3}\sum_{i=1}^n\int_0^1\int_0^1(1-\tau)\Ep[ (1-\chi_i) H_{(k,j)}(i,t,\tau)]d\tau dt.
\end{array}
$$

Using $U_{(k,j)}(x)$ as defined in relation (\ref{Def:Ukj}), by Lemma \ref{lemma:Der3}, we have for $(k,j)\in[N]^3\times[p]^3$ that
\begin{equation}\label{eq:rel_m_U}
\begin{array}{c} |m'''_{(k,j)}(x)|\leq U_{(k,j)}(x), \ \ \sum_{(k,j)\in[N]^3\times[p]^3}U_{(k,j)}(x) \lesssim \delta^{-1}\beta^2 \\
 U_{(k,j)}(x)\lesssim U_{(k,j)}(x+\tilde x) \lesssim U_{(k,j)}(x) \end{array}\end{equation}
for all $x, \tilde x \in \mathbb{R}^{N\times p}$ with $\beta\|\tilde x\|_\infty = \beta \max_{1\leq m \leq N,1\leq \ell \leq p}|\tilde x_{m\ell}|\leq 1$. Moreover we have that
$$ \begin{array}{rl}
|\dZ_{i(k_1j_1)}(t)Z_{i(k_2j_2)}(t)Z_{i(k_3j_3)}(t)| & \leq 3\frac{w(t)}{n^{3/2}} \max_{\ell=1,2,3}\{ |\wX_{i,k_\ell,j_\ell}|, \ |\wY_{i,k_\ell,j_\ell}|, \ |\wW_{i,k_\ell,j_\ell}|\} \\
& \leq  3\frac{w(t)}{n^{3/2}}{\displaystyle \max_{m\in [N], \ell \in [p]}}|\wX_{i,m\ell}|^3\vee |\wY_{i,m\ell}|^3\vee |\wW_{i,m\ell}|^3
 \\
& =  \frac{3w(t)}{n^{3/2}} M_i^3
\end{array} $$ where we define $M_i \equiv \max_{m\in [N], \ell \in [p]}|\wX_{i,m\ell}|\vee |\wY_{i,m\ell}|\vee |\wW_{i,m\ell}|$.

Therefore we bound $III_2$ as follows
{\small $$
\begin{array}{rl}
|III_2|& {\displaystyle\leq  \sum_{(k,j)\in[N]^3\times[p]^3} \sum_{i=1}^n\int_0^1\int_0^1}\Ep[(1-\chi_i)U_{(k,j)}(Z^{(i)}(t)+\tau Z_i(t)) \frac{3w(t)}{n^{3/2}} M_i^3  ]d\tau dt\\
& {\displaystyle\leq  \sum_{i=1}^n\int_0^1\int_0^1}\Ep[\frac{3w(t)}{n^{3/2}} M_i^3 (1-\chi_i){\displaystyle\sum_{(k,j)\in[N]^3\times[p]^3}} U_{(k,j)}(Z^{(i)}(t)+\tau Z_i(t))  ]d\tau dt\\
& \lesssim \delta^{-1}\beta^2  {\displaystyle\sum_{i=1}^n\int_0^1\int_0^1}\Ep[\frac{3w(t)}{n^{3/2}} M_i^3 (1-\chi_i) ]d\tau dt.\\
\end{array}
$$}
As in \citet{chernozhukov2014clt} we define $\mathcal{T}=\sqrt{n}/(4\beta)$ and using the union bound we have
$$1-\chi_i \leq 1\{\|\wX_i\|_\infty > \mathcal{T}\}+ 1\{\|\wY_i\|_\infty > \mathcal{T}\}+1\{\|\wW_i\|_\infty > \mathcal{T}\}.
$$ Using a variant of the Chebyshev's association inequality\footnote{This is used to show that $\Ep[ 1\{A>t\}B^3] \leq \Ep[1\{A>t\}A^3]+1\{B>t\}B^3]$ for positive random variables $A,B$.} (see Lemma B.1 in \citet{chernozhukov2014clt}) we have
$$\begin{array}{rl}
 \Ep[ M_i^3 (1-\chi_i)] & \leq 5\Ep[ 1\{\|\wX_i\|_\infty\geq \mathcal{T}\} \|\wX_i\|_\infty^3 ]+10\Ep[ 1\{\|\wY_i\|_\infty\geq \mathcal{T}\} \|\wY_i\|_\infty^3 ].\\
\end{array}$$
where we used that $\wY_i =_d \wW_i$. Therefore, since $\int_0^1w(t)dt =\int_0^1 1/\{\sqrt{t}\wedge\sqrt{1-t}\} dt = 2 \int_0^{1/2} (1/\sqrt{t})dt = 4\sqrt{2}$,
$$
\begin{array}{rl}
|III_2|& \displaystyle \lesssim \frac{\delta^{-1}\beta^2 }{n^{3/2}}\sum_{i=1}^n\Ep[ \|\wX_i\|_\infty^3 1\{\|\wX_i\|_\infty > \mathcal{T} \}] + \Ep[ \|\wY_i\|_\infty^3 1\{\|\wY_i\|_\infty > \mathcal{T} \}] \\
& \displaystyle \lesssim \frac{\delta^{-1}\beta^2 }{n^{1/2}}\{M_{n,\wX}(\delta/4)+M_{n,\wY}(\delta/4)\}
\end{array}$$
where $M_{n,\wX}$ and $M_{n,\wY}$ are defined in the statement of the theorem.

Next we turn to $III_1$. By Step 2 below we have that \begin{equation}\label{eq:help04a} \chi_i|m_{(k,j)}'''(Z^{(i)}(t)+\tau Z_i(t))|= h(Z^{(i)},1)\chi_i|m_{(k,j)}'''(Z^{(i)}(t)+\tau Z_i(t))|.\end{equation}

By definition we have

{\small $$
\begin{array}{ll}
III_1  = \sum_{(k,j)\in[N]^3\times[p]^3}\sum_{i=1}^n\int_0^1\int_0^1(1-\tau)\Ep[ \chi_i H_{(k,j)}(i,t,\tau)]d\tau dt\\
 \leq_{(1)} \sum_{(k,j)\in[N]^3\times[p]^3}\sum_{i=1}^n\int_0^1\int_0^1(1-\tau)\\
 \Ep[ \chi_i |m_{(k,j)}'''(Z^{(i)}(t)+\tau Z_i(t))| |\dZ_{i(k_1,j_1)}(t)Z_{i(k_2,j_2)}(t)Z_{i(k_3,j_3)}(t)|]d\tau dt\\
 \leq_{(2)} \sum_{(k,j)\in[N]^3\times[p]^3}\sum_{i=1}^n\int_0^1\int_0^1(1-\tau)\\
 \Ep[ h(Z^{(i)}(t),1)\chi_i |m_{(k,j)}'''(Z^{(i)}(t)+\tau Z_i(t))| |\dZ_{i(k_1,j_1)}(t)Z_{i(k_2,j_2)}(t)Z_{i(k_3,j_3)}(t)|]d\tau dt\\
 \leq_{(3)} \sum_{(k,j)\in[N]^3\times[p]^3}\sum_{i=1}^n\int_0^1\int_0^1(1-\tau)\\
 \Ep[ \chi_i h(Z^{(i)},1)U_{(k,j)}(Z^{(i)}(t)+\tau Z_i(t))| |\dZ_{i(k_1,j_1)}(t)Z_{i(k_2,j_2)}(t)Z_{i(k_3,j_3)}(t)|]d\tau dt\\
 \lesssim_{(4)} \sum_{(k,j)\in[N]^3\times[p]^3}\sum_{i=1}^n\int_0^1\int_0^1(1-\tau)\\
 \Ep[ \chi_i h(Z^{(i)},1)U_{(k,j)}(Z^{(i)}(t))| |\dZ_{i(k_1,j_1)}(t)Z_{i(k_2,j_2)}(t)Z_{i(k_3,j_3)}(t)|]d\tau dt,
\end{array}
$$} where (1) follows from the definition of $H_{(k,j)}$, (2) follows from (\ref{eq:help04a}), and (3) follows from the definition of $U_{(k,j)}$. Relation (4) follows since whenever $\chi_i=1$ we have $\|Z_i(t)\|_\infty \leq 3/(4\beta)$ we have $\beta\|\tau Z_i(t)\|_\infty \leq 3/4 < 1$ as required in Lemma \ref{lemma:Der3} so that
$U_{(k,j)}(Z^{(i)}(t)+\tau Z_i(t)) \lesssim U_{(k,j)}(Z^{(i)}(t)).$

Therefore
$$
\begin{array}{rl}
III_1 & \lesssim \sum_{(k,j)\in[N]^3\times[p]^3}\sum_{i=1}^n\int_0^1 \Ep[  h(Z^{(i)},1)U_{(k,j)}(Z^{(i)}(t)) ]\\
& \cdot \Ep[ \chi_i  |\dZ_{i(k_1,j_1)}(t)Z_{i(k_2,j_2)}(t)Z_{i(k_3,j_3)}(t)|] dt\\
& = \sum_{(k,j)\in[N]^3\times[p]^3}\sum_{i=1}^n\int_0^1 \Ep[ \{\chi_i+(1-\chi_i)\} h(Z^{(i)},1)U_{(k,j)}(Z^{(i)}(t)) ]\\
& \cdot \Ep[ \chi_i  |\dZ_{i(k_1,j_1)}(t)Z_{i(k_2,j_2)}(t)Z_{i(k_3,j_3)}(t)|] dt\\
& \lesssim \sum_{i=1}^n\int_0^1 \Ep[ (1-\chi_i)]\Ep[\sum_{(k,j)\in[N]^3\times[p]^3} h(Z^{(i)},1)U_{(k,j)}(Z^{(i)}(t)) ]\\
& \cdot \max_{(k,j)\in[N]^3\times[p]^3}\Ep[ \chi_i  |\dZ_{i(k_1,j_1)}(t)Z_{i(k_2,j_2)}(t)Z_{i(k_3,j_3)}(t)|] dt\\
& +\sum_{i=1}^n\int_0^1 \sum_{(k,j)\in[N]^3\times[p]^3} \Ep[ \chi_i h(Z^{(i)},1)U_{(k,j)}(Z^{(i)}(t)) ]\\
& \cdot \Ep[ \chi_i  |\dZ_{i(k_1,j_1)}(t)Z_{i(k_2,j_2)}(t)Z_{i(k_3,j_3)}(t)|] dt\\
&= III_{1a} + III_{1b}.
\end{array}
$$

Since $\sum_{(k,j)\in[N]^3\times[p]^3} U_{(k,j)}(Z^{(i)}(t)) \lesssim \delta^{-1} \beta^2$ and $h(Z^{(i)},1) \in \{0,1\}$ we have
$$ III_{1a} \lesssim \{M_{n,\widetilde X}(\delta/4)\}+M_{n,\widetilde Y }(\delta/4)\}\delta^{-1}\beta^2/n^{1/2}.$$
To bound $III_{1b}$ note that if $h(Z^{(i)}(t)+Z_i(t),2)=0$ and $\chi_i=1$, we have $h(Z^{(i)}(t),1)=0$. Therefore,
\begin{equation}\label{eq:help06} \chi_i h(Z^{(i)}(t),1) = \chi_i h(Z^{(i)}(t),1)h(Z^{(i)}(t)+ Z_i(t),2) \leq h(Z(t),2).\end{equation}

Moreover since $U_{(k,j)}(Z^{(i)}(t)) \lesssim U_{(k,j)}(Z(t))$ (by Lemma \ref{lemma:Der3} if $\chi_i=1$) and (\ref{eq:help06}) hold we have
$$
\begin{array}{rl}
III_{1b} & \lesssim \sum_{i=1}^n\sum_{(k,j)\in[N]^3\times[p]^3} \int_0^1  \Ep[ h(Z(t),2)U_{(k,j)}(Z(t)) ]\\
& \cdot \Ep[ \chi_i  |\dZ_{i(k_1,j_1)}(t)Z_{i(k_2,j_2)}(t)Z_{i(k_3,j_3)}(t)|] dt\\
& = \sum_{(k,j)\in[N]^3\times[p]^3} \int_0^1  \Ep[  h(Z(t),2)U_{(k,j)}(Z(t)) ]\\
& \cdot \sum_{i=1}^n \Ep[ \chi_i  |\dZ_{i(k_1,j_1)}(t)Z_{i(k_2,j_2)}(t)Z_{i(k_3,j_3)}(t)|] dt\\
& \leq  \int_0^1  \Ep[ h(Z(t),2) \sum_{(k,j)\in[N]^3\times[p]^3} U_{(k,j)}(Z(t)) ]\\
& \cdot \max_{(k,j)\in[N]^3\times[p]^3} \sum_{i=1}^n \Ep[ \chi_i  |\dZ_{i(k_1,j_1)}(t)Z_{i(k_2,j_2)}(t)Z_{i(k_3,j_3)}(t)|] dt\\
&\lesssim \delta^{-1} \beta^2 \frac{L_n}{n^{1/2}} \int_0^1 \Ep[ h(Z(t),2)w(t)] dt \\
&\lesssim \delta^{-1} \beta^2 L_n/n^{1/2}.
\end{array}
$$

Thus we obtain the stated bound (by redefining $\delta$ by a multiplicative factor of $4$).

{\it Step 2.} In this step we show that
\begin{equation}\label{eq:help04} \chi_i|m_{(k,j)}'''(Z^{(i)}(t)+\tau Z_i(t))|= h(Z^{(i)},1)\chi_i|m_{(k,j)}'''(Z^{(i)}(t)+\tau Z_i(t))|.\end{equation}
Note that if $\chi_i=0$ or $h(Z^{(i)}(t),1)=1$ the statement is trivial. So we assume that
 $h(Z^{(i)}(t),1)=0$ and $\chi_i=1$ occur.
Recall that  $h(Z^{(i)}(t),1)=0$ is equivalent to 
$$1\left\{ \min_{m\in[N]}\max_{\ell\in [p]} (Z^{(i)}_{m\ell}(t)+\bar\mu_{m\ell}) \in A^{\bar A\delta+1/\beta}\right\} = 0,$$
and $\chi_i=1$ is equivalent to $$1\left\{ \max_{m\in [N],\ell \in [p]} |\wX_{i,m\ell}|\vee |\wY_{i,m\ell}|\vee |\wW_{i,m\ell}| \leq \sqrt{n}/(4\beta)\right\}=1.$$
First we note that by definition, $h(Z^{(i)}(t),1)=0$ implies
\begin{equation}\label{eq:help1}
 \min_{m\in[N]}\max_{\ell\in [p]} (Z^{(i)}_{m\ell}(t)+\bar\mu_{m\ell}) \not\in  A^{\bar A\delta+\beta^{-1}}.
\end{equation}
Moreover, $\chi_i=1$ implies
\begin{equation}\label{eq:help2}
\begin{array}{rl}
|Z_{i,m\ell}(t)| & = \frac{1}{\sqrt{n}}\left|\sqrt{t}(\sqrt{v}\wX_{i,m\ell}+\sqrt{1-v}\wY_{i,m\ell})+\sqrt{1-t}\wW_{,i,m\ell}\right|\\
& \leq \frac{1}{\sqrt{n}}\left\{ |\wX_{i,m\ell}|+|\wY_{i,m\ell}|+|\wW_{i,m\ell}|\right\}\\
& \leq \frac{1}{\sqrt{n}}3\sqrt{n}/(4\beta)= (3/4)\beta^{-1}.
\end{array}
\end{equation}
Therefore, when $h(Z^{(i)}(t),1)=0$ and $\chi_i=1$, (\ref{eq:help1}) and (\ref{eq:help2}) we have
\begin{equation}\label{eq:help3}
\begin{array}{ll}
 \min_{m\in[N]}\max_{\ell\in [p]} (Z^{(i)}_{m\ell}(t)+\tau Z_{i,m\ell}(t)+\bar\mu_{m\ell})
\not\in A^{\bar A \delta + \frac{1}{4}\beta^{-1}} \\
\end{array}\end{equation} for all $\tau \in [0,1]$, since $\max_{m\in [N],j\in[p]} |Z_{i,m\ell}(t)| \leq (3/4)\beta^{-1}$. In turn, by definition (\ref{eq:help3}) implies $h(Z^{(i)}(t),0)=0$.

Moreover, (\ref{eq:help3}) and
$$-\delta \leq G_\beta(Z^{(i)}(t)+\tau Z_i(t)) - \min_{m\in [N]}\max_{j\in [p]} (Z^{(i)}(t)+\tau Z_i(t)+\bar\mu)_{m\ell} \leq \delta$$
implies that $G_\beta(Z^{(i)}(t)+\tau Z_i(t)) \not\in A^{\bar A\delta - \delta}$.

Therefore by definition of $g$, which implies $1_{A^{\delta}}(t) \leq g(t) \leq 1_{A^{4\delta}}(t)$, if $G_\beta(Z^{(i)}(t)+\tau Z_i(t)) \not\in A^{\bar A\delta - \delta}$, it follows that $m(Z^{(i)}(t)+\tau Z_i(t))=(g\circ G_\beta)(Z^{(i)}(t)+\tau Z_i(t))=0$ for $\bar A \geq 6$ so that $m_{(k,j)}'''(Z^{(i)}(t)+\tau Z_i(t))=0$. Therefore, if $h(Z^{(i)}(t),1)=0$ and $\chi_i=1$ we have (\ref{eq:help04}).
 \end{proof}

\begin{proof}[Proof of Theorem \ref{thm:clt:minmax:Gaussian:discrete}]
Similar to the proof of Theorem \ref{thm:clt:minmax:discrete} we define $m=g\circ G_\beta$, $\wX = X - \mu$, and $\wY = Y-\mu$ where we can assume $\wX$ and $\wY$ to be independent. Recall that $|G_\beta(\wX)- T| \leq \beta^{-1}\log(Np)$ by Lemma \ref{lemma:Der1} and we can take $g \in C^3(\mathbb{R})$ such that $1_A(t)\leq g(t) \leq 1_{A^{3\delta}}(t)$. Therefore we have
$$ \begin{array}{rl}
\P(T \in A ) & \leq \P( G_\beta(\wX) \in A^{\beta^{-1}\log(Np)}) \leq \Ep[ m(\wX) ]\\
 &  = \Ep[ m(\wY) ] + \Ep[ m(\wX) ]- \Ep[ m(\wY) ]  \\
 & \leq \P( G_\beta(\wY) \in A^{\beta^{-1}\log(Np)+3\delta}) + \Ep[ m(\wX) ]- \Ep[ m(\wY) ]  \\
 & \leq \P( \widetilde T \in A^{2\beta^{-1}\log(Np)+3\delta}) + \Ep[ m(\wX) ]- \Ep[ m(\wY) ] .
  \end{array}$$

We will proceed to bound $\Ep[ m(\wX) ]- \Ep[ m(\wY) ]$. Defining $Z(t)= \sqrt{t}\wX + \sqrt{1-t}\wY$ we have
$$ \Ep[ m(\wX) ]- \Ep[ m(\wY) ] = \frac{1}{2}\sum_{(k,j)\in [N]^2\times[p]^2} (\Sigma^X_{(k_1j_1,k_2j_2)}-\Sigma^Y_{(k_1j_1,k_2j_2)})\Ep[m_{(k,j)}''(Z(t))],$$
which follows from Stein's identity (see Lemma 2 in \citet{chernozhukov2012comparison}). Therefore, we have
$$\begin{array}{rl}
 |\Ep[ m(\wX) ]- \Ep[ m(\wY) ]| & \leq \|\Sigma^X-\Sigma^Y\|_\infty \Ep\left[ \sum_{(k,j)\in [N]^2\times[p]^2} |m_{(k,j)}''(Z(t))|\right] \\
 & \leq \|\Sigma^X-\Sigma^Y\|_\infty \{\|g''\|_\infty+4\beta\|g'\|_\infty\},
 \end{array}$$
where the last step follows by Lemma \ref{lemma:Der2}.

Since $\|g''\|_\infty\leq \delta^{-2}K$ and $\|g'\|_\infty \leq \delta^{-1}$ for some universal constant $K$, and setting $\beta = \delta^{-1}\log(Np)$ we obtain
$$ \begin{array}{rl}
\P(T \in A ) & \leq \P( \widetilde T \in A^{2\beta^{-1}\log(Np)+3\delta}) + C\|\Sigma^X-\Sigma^Y\|_\infty \{ \delta^{-2} + \delta^{-1} \beta \}\\
& \leq \P( \widetilde T \in A^{5\delta})    + 2C \delta^{-2}\|\Sigma^X-\Sigma^Y\|_\infty \log(Np).\\
\end{array}
$$ 
\end{proof}

\section{Technical Lemmas}

\begin{lemma}\label{aux:Gbeta}
Let the function $G_\beta: \mathbb{R}^{N\times p}\to \mathbb{R}$ be defined by $ G_\beta(W)=-F_\beta(-F_\beta(W))$. Then,
\begin{align*}
	&|G_\beta(W)-G_\beta(\widetilde W)| ~\leq~ \| W-\widetilde W\|_\infty\\
	 &-\beta^{-1}\log p ~\leq~ \min_{k\in [N]}\max_{j\in [p]} W_{kj} - G_\beta(W) ~\leq~ \beta^{-1}\log N.
\end{align*}
\end{lemma}
\begin{proof} It is well know that $\|\nabla F_\beta(v)\|_1 \leq 1$ so that  $\| F_\beta(W)-F_\beta(\widetilde W)\|_\infty \leq \|W-\widetilde W\|_\infty$. Thus we have
$$
\begin{array}{rl}
|G_\beta(W)-G_\beta(\widetilde W)| & = |F_\beta(-F_\beta(\widetilde W))-F_\beta(-F_\beta(W))|\\
& \leq \| F_\beta(W)-F_\beta(\widetilde W)\|_\infty\\
& \leq \| W-\widetilde W\|_\infty.
\end{array}
$$
Next note that
$$
\begin{array}{rl}
\min_{k\in[N]}\max_{j\in[p]} W_{kj} & \geq \min_{k\in[N]} F_\beta(W_k)-\beta^{-1}\log p \\
& = -\max_{k\in[N]} -F_\beta(W_k)-\beta^{-1}\log p \\
& \geq -F_\beta(-F_\beta(W_k))-\beta^{-1}\log p ,
\end{array}
$$ where we used that for a vector $v \in R^p$, $\max_{j\in[p]} v_j \leq F_\beta(v) \leq \max_{j\in[p]} v_j + \beta^{-1}\log p$. Similarly, we have
$$
\begin{array}{rl}
\min_{k\in[N]}\max_{j\in[p]} W_{kj} & \leq \min_{k\in[N]} F_\beta(W_k)\\
& = -\max_{k\in[N]} -F_\beta(W_k) \\
& \leq -F_\beta(-F_\beta(W_k))+\beta^{-1}\log N .
\end{array}
$$
\end{proof}

Following similar notation in \citet{chernozhukov2014clt}, we let $\delta_{a b} = 1\{a=b\}$ and define
\begin{align*}
	\pi_{k}(v) &= \exp(\beta v_k)/\sum\nolimits_{\ell=1}^N\{\exp(\beta v_\ell)\},\\
	w_{k_1k_2}(v) &= (\pi_{k_1}\delta_{k_1k_2}-\pi_{k_1}\pi_{k_2})(v)\\
	q_{k_1k_2k_3}&=(\pi_{k_1}\delta_{k_1k_3}\delta_{k_1k_2}-\pi_{k_1}\pi_{k_3}\delta_{k_1k_2}-\pi_{k_1}\pi_{k_2}(\delta_{k_1k_3}+\delta_{k_2k_3})+2\pi_{k_1}\pi_{k_2}\pi_{k_3})(v).
\end{align*}
We also define
\begin{align*}
	\pi_{j}^{\bar \mu_k}(v) &= \exp(\beta (v_j+\bar\mu_{kj}))/\sum_{\ell=1}^p\{\exp(\beta (v_\ell+\bar \mu_{k\ell}))\}, \\
	w_{j_1j_2}^{\bar \mu_k}(v) &= (\pi_{j_1}^{\bar \mu_k}\delta_{j_1j_2}-\pi_{j_1}^{\bar \mu_k}\pi_{j_2}^{\bar \mu_k})(v)\\
	q_{j_1j_2j_3}^{\bar \mu_k}&=(\pi_{j_1}^{\bar \mu_k}\delta_{j_1j_3}\delta_{j_1j_2}-\pi_{j_1}^{\bar \mu_k}\pi_{j_3}^{\bar \mu_k}\delta_{j_1j_2}-\pi_{j_1}^{\bar \mu_k}\pi_{j_2}^{\bar \mu_k}(\delta_{j_1j_3}+\delta_{j_2j_3})+2\pi_{j_1}^{\bar \mu_k}\pi_{j_2}^{\bar \mu_k}\pi_{j_3}^{\bar \mu_k})(v),
\end{align*}
where those are tailored to the matrix structure.

In what follows we denote  $\partial_{X_{kj}}m(X)  = m'_{(k,j)}(X)$,  $\partial_{X_{k_2j_2}}\partial_{X_{k_1j_1}}m(X)  = m''_{(k,j)}(X)$,
$\partial_{X_{k_3j_3}}\partial_{X_{k_2j_2}}\partial_{X_{k_1j_1}}m(X)  = m'''_{(k,j)}(X)$

\begin{lemma}\label{lemma:Der1}
Consider $m(X)=g \circ G_\beta(X)$. Then,
\begin{itemize}
	\item[(1)] for any $(k,j)\in [N]\times[p]$,
$$
 m'_{kj}(X) = g'(G_\beta(X))\pi_k(-F_\beta(X))\pi_j^{\bar\mu_k}(X_{k\cdot}).
$$
\item[(2)] for any $(k,j)\in [N]^2\times[p]^2$,
$$
\begin{array}{rl}
 m''_{(k,j)}(X)  &= g''(G_\beta(X))\pi_{k_2}(-F_\beta(X))\pi_{j_2}^{\bar\mu_{k_2}}(X_{k_2\cdot})\pi_{k_1}(-F_\beta(X))\pi_{j_1}^{\bar\mu_{k_1}}(X_{k_1\cdot})\\
& - g'(G_\beta(X))\beta w_{k_1k_2}(-F_\beta(X)) \pi_{j_2}^{\bar\mu_{k_2}}(X_{k_2\cdot})\pi_{j_1}^{\bar\mu_{k_1}}(X_{k_1\cdot})\\
&+ g'(G_\beta(X))\pi_{k_1}(-F_\beta(X)) \delta_{k_1k_2} \beta w_{j_1j_2}^{\bar\mu_{k_1}}(X_{k_1\cdot}).
\end{array}
$$
\item[(3)] for any $(k,j)\in [N]^3\times[p]^3$,
{\small $$
\begin{array}{rl}
 m'''_{(k,j)}(X)  &=
 g'''(G_\beta(X))\prod_{\ell=1}^3\pi_{k_\ell}(-F_\beta(X))\pi_{j_\ell}^{\bar\mu_{k_\ell}}(X_{k_\ell\cdot})\\
& -  g''(G_\beta(X))\beta w_{k_2k_3}(-F_\beta(X))\pi_{k_1}(-F_\beta(X))\prod_{\ell=1}^3\pi_{j_\ell}^{\bar\mu_{k_\ell}}(X_{k_\ell\cdot})\\ & + g''(G_\beta(X))\pi_{k_2}(-F_\beta(X)) \delta_{k_2k_3}\beta w_{j_2j_3}^{\bar\mu_{k_2}}(X_{k_2\cdot})\pi_{k_1}(-F_\beta(X))\pi_{j_1}^{\bar\mu_{k_1}}(X_{k_1\cdot})\\  & - g''(G_\beta(X))\pi_{k_2}(-F_\beta(X))\beta w_{k_1k_3}(-F_\beta(X))\prod_{\ell=1}^3\pi_{j_\ell}^{\bar\mu_{k_\ell}}(X_{k_\ell\cdot})\\
& +g''(G_\beta(X))\pi_{k_2}(-F_\beta(X))\pi_{j_2}^{\bar\mu_{k_2}}(X_{k_2\cdot})\pi_{k_1}(-F_\beta(X)) \delta_{k_1k_3}\beta w_{j_1 j_3}^{\bar\mu_{k_1}}(X_{k_1\cdot})\\
& - g''(G_\beta(X))\pi_{k_3}(-F_\beta(X))\beta w_{k_1k_2}(-F_\beta(X)) \prod_{\ell=1}^3\pi_{j_\ell}^{\bar\mu_{k_\ell}}(X_{k_\ell\cdot})\\
& + g'(G_\beta(X))\beta^2 q_{k_1k_2k_3}(-F_\beta(X)) \prod_{\ell=1}^3\pi_{j_\ell}^{\bar\mu_{k_\ell}}(X_{k_\ell\cdot})\\
& - g'(G_\beta(X))\beta w_{k_1k_2}(-F_\beta(X)) \delta_{k_2k_3}\beta w_{j_2j_3}^{\bar\mu_{k_2}}(X_{k_2\cdot})\pi_{j_1}^{\bar\mu_{k_1}}(X_{k_1\cdot})\\
& - g'(G_\beta(X))\beta w_{k_1k_2}(-F_\beta(X)) \pi_{j_2}^{\bar\mu_{k_2}}(X_{k_2\cdot})\delta_{k_1k_3}\beta w_{j_1j_3}^{\bar\mu_{k_1}}(X_{k_1\cdot})\\
&+ g''(G_\beta(X))\pi_{k_3}(-F_\beta(X))\pi_{j_3}^{\bar\mu_{k_3}}(X_{k_3\cdot})\pi_{k_1}(-F_\beta(X)) \delta_{k_1k_2} \beta w_{j_1j_2}(X_{k_1\cdot})\\
&- g'(G_\beta(X))\beta w_{k_1k_3}(-F_\beta(X)) \pi_{j_3}^{\bar\mu_{k_3}}(X_{k_3\cdot}) \delta_{k_1k_2} \beta w_{j_1j_2}^{\bar\mu_{k_1}}(X_{k_1\cdot})\\
&+ g'(G_\beta(X))\pi_{k_1}(-F_\beta(X)) \delta_{k_1k_2k_3} \beta^2 q_{j_1j_2j_3}^{\bar\mu_{k_1}}(X_{k_1\cdot}).
\end{array}
$$}
\end{itemize}
\end{lemma}
\begin{proof}
The results follows from direct calculations.
\end{proof}

\begin{lemma}\label{lemma:Der2}
For any $\beta>0$, $g \in C^3(\mathbb{R})$ and $m=g\circ G_\beta$ we have
$$ \begin{array}{rl}
\displaystyle \sum_{(k,j)\in [N]\times[p]} |m'_{kj}(X)| & \leq \|g'\|_\infty\\
\displaystyle \sum_{(k,j)\in [N]^2\times[p]^2} |m''_{(k,j)}(X)| & \leq  \|g''\|_\infty + 4\beta\|g'\|_\infty\\
\displaystyle \sum_{(k,j)\in [N]^3\times[p]^3} |m'''_{(k,j)}(X)| & \leq  \|g'''\|_\infty+16\beta\|g''\|_\infty+24\beta^2\|g'\|_\infty.
\end{array}$$
\end{lemma}
\begin{proof}
To prove the first result we will use that $\sum_{k\in [N]}\pi_k(v) = 1$ and $\sum_{j\in [p]} \pi_j^{\bar\mu_k}(w)=1$.
Since $\pi_k \geq 0$ and $\pi_j^{\bar\mu_k}\geq 0$ we have that
$$\begin{array}{rl}
 \sum_{(k,j)\in [N]\times[p]} |m'_{kj}(X)| & \leq \|g'\|_\infty \sum_{(k,j)\in [N]\times[p]} \pi_k(-F_\beta(X))\pi_j^{\bar\mu_k}(X_{k\cdot}) \\
 & = \|g'\|_\infty \sum_{k\in [N]} \sum_{j\in [p]} \pi_k(-F_\beta(X)) \pi_j^{\bar\mu_k}(X_{k\cdot}) \\
 & = \|g'\|_\infty \sum_{k\in [N]} \pi_k(-F_\beta(X)) \sum_{j\in [p]} \pi_j^{\bar\mu_k}(X_{k\cdot}) \\
 & = \|g'\|_\infty \sum_{k\in [N]} \pi_k(-F_\beta(X)) \\
 & = \|g'\|_\infty.
\end{array}$$

To show the second relation we will use that $\sum_{k_1,k_2\in [N]}|w_{k_1,k_2}(v)|\leq 2$ and $\sum_{j_1,j_2\in [p]}|w_{j_1,j_2}^{\bar\mu_k}(w)|\leq 2$. Therefore we have
{\small $$
\begin{array}{ll}
\sum_{(k,j)\in [N]^2\times[p]^2}  |m''_{(k,j)}(X)|  = \sum_{k_1\in [N]} \sum_{k_2\in [N]}\sum_{j_1\in [p]}\sum_{j_2\in [p]} |m''_{(k_1j_1,k_2j_2)}(X)| \\
\displaystyle \leq \|g''\|_\infty\sum_{k_1\in [N]} \sum_{k_2\in [N]}\sum_{j_1\in [p]}\sum_{j_2\in [p]} \pi_{k_2}(-F_\beta(X))\pi_{j_2}^{\bar\mu_{k_2}}(X_{k_2\cdot})\pi_{k_1}(-F_\beta(X))\pi_{j_1}^{\bar\mu_{k_1}}(X_{k_1\cdot})\\
\displaystyle + \|g'\|_\infty\sum_{k_1\in [N]} \sum_{k_2\in [N]}\sum_{j_1\in [p]}\sum_{j_2\in [p]} \beta |w_{k_1k_2}(-F_\beta(X))| \pi_{j_2}^{\bar\mu_{k_2}}(X_{k_2\cdot})\pi_{j_1}^{\bar\mu_{k_1}}(X_{k_1\cdot})\\
\displaystyle + \|g'\|_\infty\sum_{k_1\in [N]} \sum_{k_2\in [N]}\sum_{j_1\in [p]}\sum_{j_2\in [p]} \pi_{k_1}(-F_\beta(X)) \delta_{k_1k_2} \beta |w_{j_1j_2}^{\bar\mu_{k_1}}(X_{k_1\cdot})|\\
\displaystyle \leq \|g''\|_\infty \sum_{k_2\in [N]} \pi_{k_2}(-F_\beta(X))\sum_{k_1\in [N]}\pi_{k_1}(-F_\beta(X))\sum_{j_2\in [p]} \pi_{j_2}^{\bar\mu_{k_2}}(X_{k_2\cdot})\sum_{j_1\in [p]}\pi_{j_1}^{\bar\mu_{k_1}}(X_{k_1\cdot})\\
\displaystyle + \|g'\|_\infty\sum_{k_1,k_2\in [N]}\beta |w_{k_1k_2}(-F_\beta(X))| \sum_{j_2\in [p]}  \pi_{j_2}^{\bar\mu_{k_2}}(X_{k_2\cdot})\sum_{j_1\in [p]}\pi_{j_1}^{\bar\mu_{k_1}}(X_{k_1\cdot})\\
\displaystyle + \|g'\|_\infty\sum_{k_1\in [N]}\pi_{k_1}(-F_\beta(X)) \sum_{j_1,j_2\in [p]} \beta |w_{j_1j_2}^{\bar\mu_{k_1}}(X_{k_1\cdot})|\\
\displaystyle \leq \|g''\|_\infty + 4\beta\|g'\|_\infty.
\end{array}$$}

Finally, the third result follows also using that $\sum_{k_1,k_2,k_3\in [N]} |q_{k_1k_2k_3}(v)|\leq 6$ and $\sum_{j_1,j_2,j_3\in [p]} |q_{j_1j_2j_3}^{\bar\mu_k}(v)|\leq 6$. Indeed
$$
\displaystyle \sum_{(k,j)\in [N]^3\times[p]^3}  |m'''_{(k,j)}(X)|\leq \sum_{(k,j)\in [N]^3\times[p]^3} A_{(k,j)}^1(X) + \cdots +  A_{(k,j)}^{12}(X), $$
where $A^m_{(k,j)}(X)$ corresponds to the $m$th term in the expression of $m'''_{(k,j)}$ in the statement of Lemma \ref{lemma:Der1} part (3), $m=1,\ldots,12$.

We proceed to bound each term  $A^m_{(k,j)}(X)$, $m=1,\ldots,12$. It is convenient to note that $\sum_{j_1,j_2,j_3 \in [p]}\prod_{\ell=1}^3\pi_{j_\ell}^{\bar\mu_{k_\ell}}(X_{k_\ell\cdot}) =1$. We have that
$$
\begin{array}{ll}
\sum_{(k,j)\in [N]^3\times[p]^3} A^1_{(k,j)}(X)
\leq  \|g'''\|_\infty,
\end{array}
$$
where we used the following relations  $\sum_{j_1,j_2,j_3 \in [p]}\prod_{\ell=1}^3\pi_{j_\ell}^{\bar\mu_{k_\ell}}(X_{k_\ell\cdot}) =1$ and $\sum_{k_1,k_2,k_3 \in [p]}\prod_{\ell=1}^3\pi_{k_\ell}(-F_\beta(X)) =1$.
For the second term, we have
$$
\begin{array}{ll}
\sum_{(k,j)\in [N]^3\times[p]^3} A^2_{(k,j)}(X)
 \\
 \leq_{(1)}  \|g''\|_\infty \beta \sum_{k_2,k_3\in[p]} |w_{k_2k_3}(-F_\beta(X))| \sum_{k_1\in [N]} \pi_{k_1}(-F_\beta(X))\\
 \leq_{(2)} \|g''\|_\infty \beta \sum_{k_2,k_3\in[p]} |w_{k_2k_3}(-F_\beta(X))| \\
 \leq_{(3)} 6\beta \|g''\|_\infty,
\end{array}
$$ where (1) follows from $\sum_{j_1,j_2,j_3 \in [p]}\prod_{\ell=1}^3\pi_{j_\ell}^{\bar\mu_{k_\ell}}(X_{k_\ell\cdot}) =1$, (2) follows from  $\sum_{k_1\in [N]} \pi_{k_1}(-F_\beta(X))=1$, and (3) from  $\sum_{k_2,k_3\in[p]} |w_{k_2k_3}(-F_\beta(X))| \leq 6$.

The third term is bounded as
$$
\begin{array}{ll}
\sum_{(k,j)\in [N]^3\times[p]^3} A^3_{(k,j)}(X) \\
\leq_{(1)} \| g''\|_\infty  \sum_{k_2, k_3 \in [N]} \pi_{k_2}(-F_\beta(X)) \delta_{k_2k_3}\sum_{j_2,j_3\in [p]}\beta |w_{j_2j_3}^{\bar\mu_{k_2}}(X_{k_2\cdot})|\\
=_{(2)} \| g''\|_\infty  \sum_{k_2 \in [N]} \pi_{k_2}(-F_\beta(X)) \sum_{j_2,j_3\in [p]}\beta |w_{j_2j_3}^{\bar\mu_{k_2}}(X_{k_2\cdot})|\\
\leq _{(3)} 2\beta\| g''\|_\infty  \sum_{k_2 \in [N]} \pi_{k_2}(-F_\beta(X)) \\
= 2\beta\| g''\|_\infty,
\end{array}
$$
where (1) follows from $\sum_{k_1,j_1} \pi_{k_1}(-F_\beta(X))\pi_{j_1}^{\bar\mu_{k_1}}(X_{k_1\cdot}) = 1$, (2) by applying $\delta_{k_2k_3}=1\{k_2=k_3\}$, (3) from $\sum_{j_2,j_3\in [p]}  |w_{j_2j_3}^{\bar\mu_{k_2}}(X_{k_2\cdot})|\leq 2$, and the last line from $\sum_{k_2 \in [N]} \pi_{k_2}(-F_\beta(X))=1$.

For the fourth term, we have
$$
\begin{array}{ll}
\sum_{(k,j)\in [N]^3\times[p]^3} A^4_{(k,j)}(X) \\
   \leq \|g''\|_\infty \sum_{(k,j)\in [N]^3\times[p]^3} \pi_{k_2}(-F_\beta(X))\beta |w_{k_1k_3}(-F_\beta(X))|\prod_{\ell=1}^3\pi_{j_\ell}^{\bar\mu_{k_\ell}}(X_{k_\ell\cdot})\\
   =_{(1)} \|g''\|_\infty \sum_{k_1,k_2,k_3 \in [N]} \pi_{k_2}(-F_\beta(X))\beta |w_{k_1k_3}(-F_\beta(X))|\\
   =_{(2)} \|g''\|_\infty \sum_{k_1,k_3 \in [N]} \beta |w_{k_1k_3}(-F_\beta(X))|\\
   \leq_{(3)} 2 \beta\|g''\|_\infty,
\end{array}
$$ where (1) follows from $\sum_{j_1,j_2,j_3 \in [p]}\prod_{\ell=1}^3\pi_{j_\ell}^{\bar\mu_{k_\ell}}(X_{k_\ell\cdot})=1$, (2) follows from $\sum_{k_2\in [N]}\pi_{k_2}(-F_\beta(X))=1$, and (3) from $\sum_{k_1,k_3 \in [N]} |w_{k_1k_3}(-F_\beta(X))|\leq 2$.

For the fifth term, we have
{\small $$
\begin{array}{ll}
\sum_{(k,j)\in [N]^3\times[p]^3} A^5_{(k,j)}(X) \\
   \leq \|g''\|_\infty \sum_{(k,j)\in [N]^3\times[p]^3} \pi_{k_2}(-F_\beta(X))\pi_{j_2}^{\bar\mu_{k_2}}(X_{k_2\cdot})\pi_{k_1}(-F_\beta(X)) \delta_{k_1k_3}\beta |w_{j_1 j_3}^{\bar\mu_{k_1}}(X_{k_1\cdot})| \\
\leq_{(1)} 2\beta \|g''\|_\infty \sum_{k_1,k_2,k_3 \in [N], j_2 \in [p]} \pi_{k_2}(-F_\beta(X))\pi_{j_2}^{\bar\mu_{k_2}}(X_{k_2\cdot})\pi_{k_1}(-F_\beta(X)) \delta_{k_1k_3} \\
=_{(2)} 2\beta \|g''\|_\infty \sum_{k_1,k_2 \in [N], j_2 \in [p]} \pi_{k_2}(-F_\beta(X))\pi_{j_2}^{\bar\mu_{k_2}}(X_{k_2\cdot})\pi_{k_1}(-F_\beta(X)) \\
=_{(3)} 2\beta \|g''\|_\infty \sum_{k_2 \in [N], j_2 \in [p]} \pi_{k_2}(-F_\beta(X))\pi_{j_2}^{\bar\mu_{k_2}}(X_{k_2\cdot}) \\
=_{(4)} 2\beta \|g''\|_\infty,
\end{array}
$$} where (1) follows from $ \sum_{j_1,j_3 \in [p]}|w_{j_1 j_3}^{\bar\mu_{k_1}}(X_{k_1\cdot})| \leq 2$, (2) by using that $\delta_{k_1k_3} = 1\{k_1 = k_3\}$, (3) from $\sum_{k_1\in [N]}\pi_{k_1}(-F_\beta(X))=1$, and (4) from $$\sum_{k_2 \in [N], j_2 \in [p]} \pi_{k_2}(-F_\beta(X))\pi_{j_2}^{\bar\mu_{k_2}}(X_{k_2\cdot}) = \sum_{k_2 \in [N]}\pi_{k_2}(-F_\beta(X)) \sum_{j_2 \in [p]} \pi_{j_2}^{\bar\mu_{k_2}}(X_{k_2\cdot}) = 1.$$

For the sixth term, we have
$$
\begin{array}{ll}
\sum_{(k,j)\in [N]^3\times[p]^3} A^6_{(k,j)}(X) \\
   \leq \|g''\|_\infty \sum_{(k,j)\in [N]^3\times[p]^3} \pi_{k_3}(-F_\beta(X))\beta |w_{k_1k_2}(-F_\beta(X))| \prod_{\ell=1}^3\pi_{j_\ell}^{\bar\mu_{k_\ell}}(X_{k_\ell\cdot})\\
  =_{(1)}  \|g''\|_\infty \sum_{k_1,k_2,k_3 \in [N]}  \pi_{k_3}(-F_\beta(X))\beta |w_{k_1k_2}(-F_\beta(X))| \\
  =_{(2)}  \|g''\|_\infty \sum_{k_1,k_2\in [N]}  \beta |w_{k_1k_2}(-F_\beta(X))| \\
  \leq_{(3)}  2\beta\|g''\|_\infty,
 \end{array}
$$ where (1) follows from $ \sum_{j_1,j_2,j_3 \in [p]}\prod_{\ell=1}^3\pi_{j_\ell}^{\bar\mu_{k_\ell}}(X_{k_\ell\cdot})=1$, (2) follows from $\sum_{k_3 \in [N]}\pi_{k_3}(-F_\beta(X))=1$, and (3) from $\sum_{k_1,k_2\in [N]} |w_{k_1k_2}(-F_\beta(X))| \leq 2$.

For the seventh term, we have
$$
\begin{array}{ll}
\sum_{(k,j)\in [N]^3\times[p]^3} A^7_{(k,j)}(X) \\
\leq \|g'\|_\infty \sum_{(k,j)\in [N]^3\times[p]^3} \beta^2 |q_{k_1k_2k_3}(-F_\beta(X))| \prod_{\ell=1}^3\pi_{j_\ell}^{\bar\mu_{k_\ell}}(X_{k_\ell\cdot})\\
=_{(1)} \|g'\|_\infty \beta^2 \sum_{k_1,k_2,k_3 \in [N]} |q_{k_1k_2k_3}(-F_\beta(X))| \\
\leq_{(2)} 6\beta^2\|g'\|_\infty,
 \end{array}
$$ where (1) follows from $ \sum_{j_1,j_2,j_3 \in [p]}\prod_{\ell=1}^3\pi_{j_\ell}^{\bar\mu_{k_\ell}}(X_{k_\ell\cdot})=1$, and (2) follows from $\sum_{k_1,k_2,k_3 \in [N]} |q_{k_1k_2k_3}(-F_\beta(X))|\leq 6$.

For the eighth term, we have
$$
\begin{array}{ll}
\sum_{(k,j)\in [N]^3\times[p]^3} A^8_{(k,j)}(X) \\
\leq \|g'\|_\infty \sum_{(k,j)\in [N]^3\times[p]^3} \beta |w_{k_1k_2}(-F_\beta(X)) | \delta_{k_2k_3}\beta |w_{j_2j_3}^{\bar\mu_{k_2}}(X_{k_2\cdot})|\pi_{j_1}^{\bar\mu_{k_1}}(X_{k_1\cdot})\\
=_{(1)} \beta^2\|g'\|_\infty \sum_{k_1,k_2,k_3 \in [N]}  |w_{k_1k_2}(-F_\beta(X))| \delta_{k_2k_3} \sum_{j_2,j_3\in[p]}  |w_{j_2j_3}^{\bar\mu_{k_2}}(X_{k_2\cdot})|\\
\leq_{(2)} 2\beta^2\|g'\|_\infty \sum_{k_1,k_2,k_3 \in [N]}  |w_{k_1k_2}(-F_\beta(X))| \delta_{k_2k_3}\\
=_{(3)} 2\beta^2\|g'\|_\infty \sum_{k_1,k_2 \in [N]}  |w_{k_1k_2}(-F_\beta(X))| \\
\leq_{(4)} 4\beta^2\|g'\|_\infty,
 \end{array}
$$ where (1)  from $\sum_{j_1\in[p]}\pi_{j_1}^{\bar\mu_{k_1}}(X_{k_1\cdot})=1$, (2) from $ \sum_{j_2,j_3\in[p]}  |w_{j_2j_3}^{\bar\mu_{k_2}}(X_{k_2\cdot})|\leq 2$, (3) from $\delta_{k_2k_3} = 1\{k_2=k_3\}$, and (4) from $\sum_{k_1,k_2 \in [N]}  |w_{k_1k_2}(-F_\beta(X))|\leq 2$.

For the ninth term, we have
$$
\begin{array}{ll}
\sum_{(k,j)\in [N]^3\times[p]^3} A^9_{(k,j)}(X) \\
\leq \|g'\|_\infty \sum_{(k,j)\in [N]^3\times[p]^3} \beta |w_{k_1k_2}(-F_\beta(X))| \pi_{j_2}^{\bar\mu_{k_2}}(X_{k_2\cdot})\delta_{k_1k_3}\beta |w_{j_1j_3}^{\bar\mu_{k_1}}(X_{k_1\cdot})|\\
\leq_{(1)} 2\beta^2\|g'\|_\infty \sum_{k_1,k_2,k_3\in [N]} |w_{k_1k_2}(-F_\beta(X))| \delta_{k_1k_3}\sum_{j_2\in[p]}\pi_{j_2}^{\bar\mu_{k_2}}(X_{k_2\cdot}) \\
=_{(2)} 2\beta^2\|g'\|_\infty \sum_{k_1,k_2,k_3\in [N]} |w_{k_1k_2}(-F_\beta(X))| \delta_{k_1k_3} \\
=_{(3)} 2\beta^2\|g'\|_\infty \sum_{k_1,k_2\in [N]} |w_{k_1k_2}(-F_\beta(X))| \\
\leq_{(4)} 4\beta^2\|g'\|_\infty,
 \end{array}
$$ where (1)  from $\sum_{j_1,j_3\in[p]}|w_{j_1j_3}^{\bar\mu_{k_1}}(X_{k_1\cdot})|\leq 2$, (2) from $\sum_{j_2\in[p]}\pi_{j_2}^{\bar\mu_{k_2}}(X_{k_2\cdot})=1$, (3) from $\delta_{k_1k_3}=1\{k_1=k_3\}$, and (4) from $\sum_{k_1,k_2\in [N]} |w_{k_1k_2}(-F_\beta(X))| \leq 2$.

For the tenth term, we have
{\small $$
\begin{array}{ll}
\sum_{(k,j)\in [N]^3\times[p]^3} A^{10}_{(k,j)}(X) \\
\leq \|g''\|_\infty \sum_{(k,j)\in [N]^3\times[p]^3} \pi_{k_3}(-F_\beta(X))\pi_{j_3}^{\bar\mu_{k_3}}(X_{k_3\cdot})\pi_{k_1}(-F_\beta(X)) \delta_{k_1k_2}|w_{j_1j_2}(X_{k_1\cdot})| \\
\leq_{(1)} 2\beta \|g''\|_\infty\sum_{k_1,k_2,k_3\in [N]} \pi_{k_3}(-F_\beta(X))\pi_{k_1}(-F_\beta(X)) \delta_{k_1k_2}\sum_{j_3\in[p]}\pi_{j_3}^{\bar\mu_{k_3}}(X_{k_3\cdot})\\
=_{(2)} 2\beta \|g''\|_\infty\sum_{k_1,k_2,k_3\in [N]} \pi_{k_3}(-F_\beta(X))\pi_{k_1}(-F_\beta(X)) \delta_{k_1k_2}\\
=_{(3)} 2\beta \|g''\|_\infty\sum_{k_1,k_3\in [N]} \pi_{k_3}(-F_\beta(X))\pi_{k_1}(-F_\beta(X))\\
=_{(4)} 2\beta \|g''\|_\infty,
 \end{array}
$$} where (1)  from $\sum_{j_1,j_2\in [p]} |w_{j_1j_2}(X_{k_1\cdot})|\leq 2$, (2) from $\sum_{j_3\in[p]}\pi_{j_3}^{\bar\mu_{k_3}}(X_{k_3\cdot})=1$, (3) from $\delta_{k_1k_2} = 1\{k_1=k_2\}$, and (4) from {\small $$\sum_{k_1,k_3\in [N]} \pi_{k_3}(-F_\beta(X))\pi_{k_1}(-F_\beta(X)) = \sum_{k_1\in [N]} \pi_{k_1}(-F_\beta(X)) \sum_{k_1\in [N]} \pi_{k_3}(-F_\beta(X))=1.$$}

For the eleventh term, we have
$$
\begin{array}{ll}
\sum_{(k,j)\in [N]^3\times[p]^3} A^{11}_{(k,j)}(X) \\
\leq \|g'\|_\infty \sum_{(k,j)\in [N]^3\times[p]^3} \beta |w_{k_1k_3}(-F_\beta(X))| \pi_{j_3}^{\bar\mu_{k_3}}(X_{k_3\cdot}) \delta_{k_1k_2} \beta |w_{j_1j_2}^{\bar\mu_{k_1}}(X_{k_1\cdot})| \\
\leq_{(1)} 2\beta^2\|g'\|_\infty \sum_{k_1,k_2,k_3 \in [N], j_3 \in [p]} |w_{k_1k_3}(-F_\beta(X))| \pi_{j_3}^{\bar\mu_{k_3}}(X_{k_3\cdot}) \delta_{k_1k_2}  \\
=_{(2)} 2\beta^2\|g'\|_\infty \sum_{k_1,k_3 \in [N], j_3 \in [p]} |w_{k_1k_3}(-F_\beta(X))| \pi_{j_3}^{\bar\mu_{k_3}}(X_{k_3\cdot})  \\
=_{(3)} 2\beta^2\|g'\|_\infty \sum_{k_1,k_3 \in [N]} |w_{k_1k_3}(-F_\beta(X))| \\
\leq_{(4)} 4\beta^2\|g'\|_\infty,
 \end{array}
$$ where (1)  from $\sum_{j_1,j_2\in[p]}|w_{j_1j_2}^{\bar\mu_{k_1}}(X_{k_1\cdot})|\leq 2$,
(2) from $ \delta_{k_1k_2}=1\{k_1=k_2\}$, (3) from $\sum_{j_3\in[p]} \pi_{j_3}^{\bar\mu_{k_3}}(X_{k_3\cdot}) = 1$, and (4) from $\sum_{k_1,k_3 \in [N]} |w_{k_1k_3}(-F_\beta(X))|\leq 2.$

Finally, for the twelfth term, we have
$$
\begin{array}{ll}
\sum_{(k,j)\in [N]^3\times[p]^3} A^{12}_{(k,j)}(X) \\
\leq \|g'\|_\infty \sum_{(k,j)\in [N]^3\times[p]^3} \pi_{k_1}(-F_\beta(X)) \delta_{k_1k_2k_3} \beta^2 |q_{j_1j_2j_3}^{\bar\mu_{k_1}}(X_{k_1\cdot})|\\
=_{(1)} \beta^2\|g'\|_\infty \sum_{k_1\in [N]} \pi_{k_1}(-F_\beta(X)) \sum_{j_1,j_2,j_3 \in [p]} |q_{j_1j_2j_3}^{\bar\mu_{k_1}}(X_{k_1\cdot})|\\
\leq_{(2)} 6\beta^2\|g'\|_\infty \sum_{k_1\in [N]} \pi_{k_1}(-F_\beta(X))\\
=_{(3)} 6\beta^2\|g'\|_\infty,
\end{array}
$$ where relation (1) follows  from $\delta_{k_1k_2k_3}=1\{k_1=k_2=k_3\}$, (2) follows from $\sum_{j_1,j_2,j_3 \in [p]} |q_{j_1j_2j_3}^{\bar\mu_{k_1}}(X_{k_1\cdot})|\leq 6$, and (3) from $\sum_{k_1\in [N]} \pi_{k_1}(-F_\beta(X))=1$.
Collecting all the twelve terms we have
$$
\displaystyle \sum_{(k,j)\in [N]^3\times[p]^3}  |m'''_{(k,j)}(X)|\leq \|g'''\|_\infty+16\beta\|g''\|_\infty+24\beta^2\|g'\|_\infty .$$
\end{proof}

\begin{lemma}\label{lemma:Der3}
For any $\beta>0$, $g \in C^3(\mathbb{R})$, $m=g\circ G_\beta$, and $(k,j)\in [N]^3\times[p]^3$, define:
\begin{equation}\label{Def:Ukj}
U_{(k,j)}(x)=\sup\{ |m'''_{(k,j)}(x+y)| : y \in \mathbb{R}^{N\times p}, \|y\|_\infty \leq \beta^{-1}\}.
\end{equation}
Then,
$$  \sum_{(k,j)\in [N]^3\times[p]^3} U_{(k,j)}(x) \leq e^{12} \{ \|g'''\|_\infty+16\beta\|g''\|_\infty+24\beta^2\|g'\|_\infty\}$$
\end{lemma}
\begin{proof}
As argued in page 1586 of \citet{chernozhukov2012Gaussian}, $\|y\|_\infty \leq \beta^{-1}$ implies $\pi_k(x+y)\leq e^2\pi_k(x)$. Therefore,
$$\begin{array}{rl}
\sum_{k\in[N]}\sup_{\|y\|_\infty \leq \beta^{-1}} \pi_k(x+y) & \leq e^2\\
\sum_{k_1,k_2 \in[N]}\sup_{\|y\|_\infty \leq \beta^{-1}} |w_{k_1,k_2}(x+y)|&  \leq 2 e^4\\
\sum_{k_1,k_2,k_3 \in[N]}\sup_{\|y\|_\infty \leq \beta^{-1}} |q_{k_1,k_2,k_3}(x+y)|& \leq 6e^6.\end{array}$$
In turn, we have that
$$  \sum_{(k,j)\in [N]^3\times[p]^3} U_{(k,j)}(x) \leq e^{12} \{ \|g'''\|_\infty+16\beta\|g''\|_\infty+24\beta^2\|g'\|_\infty\}, $$ where the factor $e^{12}$ accounts for the potential product of six factors of $e^2$, i.e., one such factor for each index of the sum.
\end{proof}

\end{appendix}

\bibliography{PIbib}

\begin{thebibliography}{53}
\providecommand{\natexlab}[1]{#1}
\providecommand{\url}[1]{\texttt{#1}}
\expandafter\ifx\csname urlstyle\endcsname\relax
  \providecommand{\doi}[1]{doi: #1}\else
  \providecommand{\doi}{doi: \begingroup \urlstyle{rm}\Url}\fi

\bibitem[Abramowitz and Stegun(1964)]{abramowitz/stegun:1964}
M.~Abramowitz and I.~A. Stegun.
\newblock \emph{Handbook of Mathematical Functions with Formulas, Graphs, and
  Mathematical Tables}.
\newblock National Bureau of Standards. Applied Mathematics Series, 1964.

\bibitem[Andrews and Guggenberger(2009)]{andrews2009validity}
D.~W.~K. Andrews and P.~Guggenberger.
\newblock Validity of subsampling and “plug-in asymptotic” inference for
  parameters defined by moment inequalities.
\newblock \emph{Econometric Theory}, 25\penalty0 (3):\penalty0 669--709, 2009.

\bibitem[Andrews and {Jia-Barwick}(2012)]{andrews2012inference}
D.~W.~K. Andrews and P.~{Jia-Barwick}.
\newblock Inference for parameters defined by moment inequalities: A
  recommended moment selection procedure.
\newblock \emph{Econometrica}, 80\penalty0 (6):\penalty0 2805--2826, 2012.

\bibitem[Andrews and Shi(2014)]{andrews2014nonparametric}
D.~W.~K. Andrews and X.~Shi.
\newblock Nonparametric inference based on conditional moment inequalities.
\newblock \emph{Journal of Econometrics}, 179\penalty0 (1):\penalty0 31--45,
  2014.

\bibitem[Andrews and Shi(2017)]{andrews2017inference}
D.~W.~K. Andrews and X.~Shi.
\newblock Inference based on many conditional moment inequalities.
\newblock \emph{Journal of Econometrics}, 196\penalty0 (2):\penalty0 275--287,
  2017.

\bibitem[Andrews and Shi(2013)]{andrews2013inference}
D.~W.~K. Andrews and Xiaoxia Shi.
\newblock Inference based on conditional moment inequalities.
\newblock \emph{Econometrica}, 81\penalty0 (2):\penalty0 609--666, 2013.

\bibitem[Andrews and Soares(2010)]{andrews2010inference}
D.~W.~K. Andrews and G.~Soares.
\newblock Inference for parameters defined by moment inequalities using
  generalized moment selection.
\newblock \emph{Econometrica}, 78\penalty0 (1):\penalty0 119--157, 2010.

\bibitem[Andrews et~al.(2004)Andrews, Berry, and
  {Jia-Barwick}]{andrews/berry/jiabarwick:2004}
D.~W.~K. Andrews, S.~Berry, and P.~{Jia-Barwick}.
\newblock Confidence regions for parameters in discrete games with multiple
  equilibria with an application to discount chain store location.
\newblock Mimeo: Yale University and M.I.T., May 2004.

\bibitem[Armstrong(2014)]{armstrong2014weighted}
T.~B. Armstrong.
\newblock Weighted ks statistics for inference on conditional moment
  inequalities.
\newblock \emph{Journal of Econometrics}, 181\penalty0 (2):\penalty0 92--116,
  2014.

\bibitem[Armstrong(2015)]{armstrong2015asymptotically}
T.~B. Armstrong.
\newblock Asymptotically exact inference in conditional moment inequality
  models.
\newblock \emph{Journal of Econometrics}, 186\penalty0 (1):\penalty0 51--65,
  2015.

\bibitem[Armstrong(2018)]{armstrong2018choice}
T.~B. Armstrong.
\newblock On the choice of test statistic for conditional moment inequalities.
\newblock \emph{Journal of Econometrics}, 2018.

\bibitem[Armstrong and Chan(2016)]{armstrong2016multiscale}
T.~B. Armstrong and H.~P. Chan.
\newblock Multiscale adaptive inference on conditional moment inequalities.
\newblock \emph{Journal of Econometrics}, 194\penalty0 (1):\penalty0 24--43,
  2016.

\bibitem[Bajari et~al.(2007)Bajari, Benkard, and Levin]{bajari2007estimating}
P.~Bajari, C.~L. Benkard, and J.~Levin.
\newblock Estimating dynamic models of imperfect competition.
\newblock \emph{Econometrica}, 75\penalty0 (5):\penalty0 1331--1370, 2007.

\bibitem[Belloni and Chernozhukov(2011)]{BC-SparseQR}
A.~Belloni and V.~Chernozhukov.
\newblock $\ell_1$-penalized quantile regression for high dimensional sparse
  models.
\newblock \emph{Annals of Statistics}, 39\penalty0 (1):\penalty0 82--130, 2011.

\bibitem[Belloni et~al.(2015{\natexlab{a}})Belloni, Chernozhukov, and
  Kato]{BCK-LAD}
A.~Belloni, V.~Chernozhukov, and K.~Kato.
\newblock Uniform post-selection inference for least absolute deviation
  regression and other z-estimation problems.
\newblock \emph{Biometrika}, 102\penalty0 (1):\penalty0 77--94,
  2015{\natexlab{a}}.

\bibitem[Belloni et~al.(2015{\natexlab{b}})Belloni, Chernozhukov, Chetverikov,
  and Wei]{belloni2015uniformly}
Alexandre Belloni, Victor Chernozhukov, Denis Chetverikov, and Ying Wei.
\newblock Uniformly valid post-regularization confidence regions for many
  functional parameters in z-estimation framework.
\newblock \emph{arXiv preprint arXiv:1512.07619}, 2015{\natexlab{b}}.

\bibitem[Beresteanu et~al.(2011)Beresteanu, Molchanov, and
  Molinari]{beresteanu2011sharp}
A.~Beresteanu, I.~Molchanov, and F.~Molinari.
\newblock Sharp identification regions in models with convex moment
  predictions.
\newblock \emph{Econometrica}, 79\penalty0 (6):\penalty0 1785--1821, 2011.

\bibitem[Bugni(2010)]{bugni:2010}
F.~A. Bugni.
\newblock Bootstrap inference in partially identified models defined by moment
  inequalities: Coverage of the identified set.
\newblock \emph{Econometrica}, 78\penalty0 (2):\penalty0 735--753, March 2010.

\bibitem[Bugni(2016)]{bugni2016comparison}
F.~A. Bugni.
\newblock Comparison of inferential methods in partially identified models in
  terms of error in coverage probability.
\newblock \emph{Econometric Theory}, 32\penalty0 (1):\penalty0 187--242, 2016.

\bibitem[Bugni et~al.(2015)Bugni, Canay, and Shi]{bugni2015specification}
F.~A. Bugni, I.~A. Canay, and X.~Shi.
\newblock Specification tests for partially identified models defined by moment
  inequalities.
\newblock \emph{Journal of Econometrics}, 185\penalty0 (1):\penalty0 259--282,
  2015.

\bibitem[Bugni et~al.(2017)Bugni, Canay, and Shi]{bugni2017inference}
F.~A. Bugni, I.~A. Canay, and X.~Shi.
\newblock Inference for subvectors and other functions of partially identified
  parameters in moment inequality models.
\newblock \emph{Quantitative Economics}, 8\penalty0 (1):\penalty0 1--38, 2017.

\bibitem[Canay(2010)]{canay2010inference}
I.~A. Canay.
\newblock {EL} inference for partially identified models: Large deviations
  optimality and bootstrap validity.
\newblock \emph{Journal of Econometrics}, 156\penalty0 (2):\penalty0 408--425,
  2010.

\bibitem[Canay and Shaikh(2017)]{canay/shaikh:2017}
I.~A. Canay and A.~M. Shaikh.
\newblock Practical and theoretical advances for inference in partially
  identified models.
\newblock In B.~Honor\'{e}, A.~Pakes, M.~Piazzesi, and L.~Samuelson, editors,
  \emph{Advances in Economics and Econometrics: Eleventh World Congress},
  volume~2, chapter~9, pages 271--306. Cambridge University Press, 2017.

\bibitem[Chernozhukov et~al.(2007)Chernozhukov, Hong, and
  Tamer]{chernozhukov2007estimation}
V.~Chernozhukov, H.~Hong, and E.~Tamer.
\newblock Estimation and confidence regions for parameter sets in econometric
  models.
\newblock \emph{Econometrica}, 75\penalty0 (5):\penalty0 1243--1284, 2007.

\bibitem[Chernozhukov et~al.(2013{\natexlab{a}})Chernozhukov, Chetverikov, and
  Kato]{chernozhukov2013testing}
V.~Chernozhukov, D.~Chetverikov, and K.~Kato.
\newblock Testing many moment inequalities.
\newblock \emph{arXiv preprint arXiv:1312.7614}, 2013{\natexlab{a}}.

\bibitem[Chernozhukov et~al.(2013{\natexlab{b}})Chernozhukov, Lee, and
  Rosen]{chernozhukov2013intersection}
V.~Chernozhukov, S.~Lee, and A.~M. Rosen.
\newblock Intersection bounds: estimation and inference.
\newblock \emph{Econometrica}, 81\penalty0 (2):\penalty0 667--737,
  2013{\natexlab{b}}.

\bibitem[Chernozhukov et~al.(2014{\natexlab{a}})Chernozhukov, Chetverikov, and
  Kato]{chernozhukov2012Gaussian}
V.~Chernozhukov, D.~Chetverikov, and K.~Kato.
\newblock Gaussian approximation of suprema of empirical processes.
\newblock \emph{The Annals of Statistics}, 42\penalty0 (4):\penalty0
  1564--1597, 2014{\natexlab{a}}.

\bibitem[Chernozhukov et~al.(2014{\natexlab{b}})Chernozhukov, Chetverikov, and
  Kato]{chernozhukov2014clt}
V.~Chernozhukov, D.~Chetverikov, and K.~Kato.
\newblock Central limit theorems and bootstrap in high dimensions.
\newblock \emph{arXiv preprint}, 2014{\natexlab{b}}.

\bibitem[Chernozhukov et~al.(2014{\natexlab{c}})Chernozhukov, Chetverikov, and
  Kato]{chernozhukov2014honestbands}
V.~Chernozhukov, D.~Chetverikov, and K.~Kato.
\newblock Anti-concentration and honest, adaptive confidence bands.
\newblock \emph{The Annals of Statistics}, 42\penalty0 (5):\penalty0
  1787--1818, 2014{\natexlab{c}}.

\bibitem[Chernozhukov et~al.(2015{\natexlab{a}})Chernozhukov, Chetverikov, and
  Kato]{chernozhukov2012comparison}
V.~Chernozhukov, D.~Chetverikov, and K.~Kato.
\newblock Comparison and anti-concentration bounds for maxima of gaussian
  random vectors.
\newblock \emph{Probability Theory and Related Fields}, 162:\penalty0 47--70,
  2015{\natexlab{a}}.

\bibitem[Chernozhukov et~al.(2015{\natexlab{b}})Chernozhukov, Chetverikov, and
  Kato]{chernozhukov2015noncenteredprocesses}
V.~Chernozhukov, D.~Chetverikov, and K.~Kato.
\newblock Empirical and multiplier bootstraps for supreme of empirical
  processes of increasing complexity, and related gaussian couplings.
\newblock \emph{arXiv preprint}, 2015{\natexlab{b}}.

\bibitem[Chernozhukov et~al.(2015{\natexlab{c}})Chernozhukov, Newey, and
  Santos]{chernozhukov2015constrained}
V.~Chernozhukov, W.~K. Newey, and A.~Santos.
\newblock Constrained conditional moment restriction models.
\newblock \emph{arXiv preprint arXiv:1509.06311}, 2015{\natexlab{c}}.

\bibitem[Chesher and Rosen(2017)]{chesher2017generalized}
A.~Chesher and A.~M. Rosen.
\newblock Generalized instrumental variable models.
\newblock \emph{Econometrica}, 85\penalty0 (3):\penalty0 959--989, 2017.

\bibitem[Chesher et~al.(2013)Chesher, Rosen, and
  Smolinski]{chesher2013instrumental}
A.~Chesher, A.~M. Rosen, and K.~Smolinski.
\newblock An instrumental variable model of multiple discrete choice.
\newblock \emph{Quantitative Economics}, 4\penalty0 (2):\penalty0 157--196,
  2013.

\bibitem[Chetverikov(2018)]{chetverikov2018adaptive}
Denis Chetverikov.
\newblock Adaptive tests of conditional moment inequalities.
\newblock \emph{Econometric Theory}, 34\penalty0 (1):\penalty0 186--227, 2018.

\bibitem[Ciliberto and Tamer(2009)]{ciliberto2009market}
F.~Ciliberto and E.~Tamer.
\newblock Market structure and multiple equilibria in airline markets.
\newblock \emph{Econometrica}, 77\penalty0 (6):\penalty0 1791--1828, 2009.

\bibitem[{de la Pe{\~n}a} et~al.(2009){de la Pe{\~n}a}, Lai, and
  Shao]{delapena}
V.~H. {de la Pe{\~n}a}, T.~L. Lai, and {Q.-M.} Shao.
\newblock \emph{Self-normalized processes}.
\newblock Probability and its Applications (New York). Springer-Verlag, Berlin,
  2009.
\newblock ISBN 978-3-540-85635-1.
\newblock Limit theory and statistical applications.

\bibitem[Dudley(1999)]{Dudley99}
R.~M. Dudley.
\newblock \emph{Uniform central limit theorems}, volume~63 of \emph{Cambridge
  Studies in Advanced Mathematics}.
\newblock Cambridge University Press, Cambridge, 1999.
\newblock ISBN 0-521-46102-2.
\newblock \doi{10.1017/CBO9780511665622}.
\newblock URL \url{http://dx.doi.org/10.1017/CBO9780511665622}.

\bibitem[Gafarov(2016)]{gafarovinference}
B.~Gafarov.
\newblock Inference on scalar parameters in set--identified affine models job
  market paper.
\newblock Mimeo: UC Davis, 2016.

\bibitem[Galichon and Henry(2011)]{galichon2011set}
A.~Galichon and M.~Henry.
\newblock Set identification in models with multiple equilibria.
\newblock \emph{The Review of Economic Studies}, 78\penalty0 (4):\penalty0
  1264--1298, 2011.

\bibitem[Ho and Rosen(2017)]{ho/rosen:2017}
K.~Ho and A.~M. Rosen.
\newblock Partial identification in applied research: Benefits and challenges.
\newblock In B.~Honor\'{e}, A.~Pakes, M.~Piazzesi, and L.~Samuelson, editors,
  \emph{Advances in Economics and Econometrics: Eleventh World Congress},
  volume~2, chapter~10, pages 307--359. Cambridge University Press, 2017.

\bibitem[Kaido et~al.(2016)Kaido, Molinari, and Stoye]{kaido2016confidence}
H.~Kaido, F.~Molinari, and J.~Stoye.
\newblock Confidence intervals for projections of partially identified
  parameters.
\newblock \emph{arXiv preprint arXiv:1601.00934}, 2016.

\bibitem[Kim(2008)]{kim:2008}
K.~Kim.
\newblock Set estimation and inference with models characterized by conditional
  moment inequalities.
\newblock Mimeo: Michigan State University, July 2008.

\bibitem[Lee et~al.(2013)Lee, Song, and Whang]{lee2013testing}
S.~Lee, K.~Song, and {Y.-J.} Whang.
\newblock Testing functional inequalities.
\newblock \emph{Journal of Econometrics}, 172\penalty0 (1):\penalty0 14--32,
  2013.

\bibitem[Menzel(2014)]{menzel2014consistent}
K.~Menzel.
\newblock Consistent estimation with many moment inequalities.
\newblock \emph{Journal of Econometrics}, 182\penalty0 (2):\penalty0 329--350,
  2014.

\bibitem[Pakes and Porter(2013)]{pakes2013moment}
A.~Pakes and J.~Porter.
\newblock Moment inequalities for semiparametric multinomial choice with fixed
  effects.
\newblock \emph{Working paper, Harvard University, Cambridge, MA}, 2013.

\bibitem[Pakes et~al.(2015)Pakes, Porter, Ho, and
  Ishii]{pakes/porter/ho/ishii:2015}
A.~Pakes, J.~Porter, K.~Ho, and J.~Ishii.
\newblock Moment inequalities and their application.
\newblock \emph{Econometrica}, 83\penalty0 (1):\penalty0 315--334, January
  2015.

\bibitem[Ponomareva(2010)]{ponomareva:2010}
M.~Ponomareva.
\newblock Inference in models defined by conditional moment inequalities with
  continuous covariates.
\newblock Mimeo: Northern Illinois University, June 2010.

\bibitem[Romano and Shaikh(2008)]{romano2008inference}
J.~P. Romano and A.~M. Shaikh.
\newblock Inference for identifiable parameters in partially identified
  econometric models.
\newblock \emph{Journal of Statistical Planning and Inference}, 138\penalty0
  (9):\penalty0 2786--2807, 2008.

\bibitem[Romano et~al.(2014)Romano, Shaikh, and Wolf]{romano2014practical}
J.~P. Romano, A.~M. Shaikh, and M.~Wolf.
\newblock A practical two-step method for testing moment inequalities.
\newblock \emph{Econometrica}, 82\penalty0 (5):\penalty0 1979--2002, 2014.

\bibitem[Rosen(2008)]{rosen2008confidence}
A.~M. Rosen.
\newblock Confidence sets for partially identified parameters that satisfy a
  finite number of moment inequalities.
\newblock \emph{Journal of Econometrics}, 146\penalty0 (1):\penalty0 107--117,
  2008.

\bibitem[Tamer(2010)]{tamer2010partial}
E.~Tamer.
\newblock Partial identification in econometrics.
\newblock \emph{Annu. Rev. Econ.}, 2\penalty0 (1):\penalty0 167--195, 2010.

\bibitem[van~der Vaart and Wellner(1996)]{vdV-W}
A.~W. van~der Vaart and J.~A. Wellner.
\newblock \emph{Weak Convergence and Empirical Processes}.
\newblock Springer Series in Statistics, 1996.

\end{thebibliography}
\bibliographystyle{plainnat}

\end{document}